\documentclass[a4paper,11pt,twoside]{amsart}
\usepackage[english]{babel}
\usepackage[utf8]{inputenc}

\usepackage[a4paper,inner=3cm,outer=3cm,top=4cm,bottom=4cm,pdftex]{geometry}
\usepackage{fancyhdr}
\usepackage{titlesec}
\usepackage{color}
\usepackage{bold-extra}
\usepackage{ mathrsfs }
\titleformat{\section}{\normalfont\scshape\centering}{\thesection}{1em}{}
  \titleformat{\subsection}{\bfseries}{\thesubsection}{1em}{}

\usepackage{comment}
\usepackage{graphics}
\usepackage{aliascnt}
\usepackage[pdftex,citecolor=green,linkcolor=red]{hyperref}

\usepackage{enumerate}
\usepackage{amsmath}
\usepackage{amsfonts}
\usepackage{amssymb}
\usepackage{amsthm}
\usepackage{comment}
\usepackage{mathtools}

\newtheorem{theorem}{Theorem}[section]
\newtheorem{corollary}[theorem]{Corollary}

\newtheorem{lemma}[theorem]{Lemma}
\newtheorem{proposition}[theorem]{Proposition}
\theoremstyle{definition}
\newtheorem{definition}[theorem]{Definition}
\newtheorem{remark}[theorem]{Remark}

\newtheorem{example}[theorem]{Example}
\numberwithin{equation}{section}

\newcommand\eps{\varepsilon}

\newcommand\E{\mathbb{E}}
\newcommand\R{\mathbb{R}}
\newcommand\Z{\mathbb{Z}}
\newcommand\N{\mathbb{N}}
\newcommand\C{\mathbb{C}}

\newcommand\Siegel{{\textnormal{Siegel}}}

\newcommand\poly{\mathrm{poly}}

\newcommand\Lip{\mathrm{Lip}}
\newcommand\id{\mathrm{id}}
\newcommand\lin{\mathrm{lin}}
\newcommand\Cramer{{\textnormal{Cram\'er}}}
\newcommand\local{{\textnormal{local}}}
\newcommand\ineff{{\textnormal{ineff}}}
\begin{document}
\title{Quantitative bounds for Gowers uniformity of the M\"obius and von Mangoldt functions}

\author[Terence Tao]{Terence Tao}
\address{Department of Mathematics, University of California, Los Angeles, CA 90095-1555,USA}
\email{tao@math.ucla.edu}

\author[Joni Ter\"{a}v\"{a}inen]{Joni Ter\"{a}v\"{a}inen}

\address{Department of Mathematics and Statistics, University of Turku, Turku, Finland}
\email{joni.p.teravainen@gmail.com}

\begin{abstract} We establish quantitative bounds on the $U^k[N]$ Gowers norms of the M\"obius function $\mu$ and the von Mangoldt function $\Lambda$ for all $k$, with error terms of the shape $O((\log\log N)^{-c})$.  As a consequence, we obtain quantitative bounds for the number of solutions to any linear system of equations of finite complexity in the primes, with the same shape of error terms. We also obtain the first quantitative bounds on the size of sets containing no $k$-term arithmetic progressions with shifted prime difference.
\end{abstract}

\maketitle

\section{Introduction}

Throughout this paper we fix an integer $k \geq 1$, and let $N > 1$ be a real parameter that is assumed to be sufficiently large depending on $k$.  We will also make frequent use of the somewhat smaller quantity
\begin{equation}\label{Q-def}
Q \coloneqq \exp(\log^{1/10} N),
\end{equation}
for instance by sieving out multiples of all primes less than $Q$.  We use $c$ to denote various small positive constants depending on $k$ that are allowed to vary from line to line, or even within the same line.  All the constants in our asymptotic notation\footnote{See Section~\ref{notation-sec} for a more detailed description of the asymptotic notation conventions used in this paper.} are permitted to depend on $k$. The implied constants will be effective, except when otherwise stated.  

In this paper we will be interested in quantitatively controlling  the Gowers norm uniformity of the M\"obius function $\mu$ and the von Mangoldt function $\Lambda$ on the interval $[N] \coloneqq \{n \in \N: 1 \leq n \leq N \}$, as well as various related statistics. Our methods can extend to some other arithmetic functions, such as sufficiently ``non-pretentious'' bounded multiplicative functions, but we focus on the classical functions $\mu, \Lambda$ here for ease of exposition.  Such quantitative control on the Gowers norms will be used to quantify the asymptotics for linear equations in primes obtained in~\cite{gt-linear}.

We begin by recalling the definition of the Gowers uniformity norms, first introduced by Gowers in~\cite{gowers}; we largely follow the notation of~\cite[Appendix B]{gt-linear} here, except that we will find it convenient to work with both normalized and unnormalized Gowers norms.

\begin{definition}[Gowers norms]\label{gowers-norm}  Let $k \geq 1$ be a natural number. 
\begin{itemize}
\item[(i)] If $\omega \in \{0,1\}^k$ is a $k$-tuple, we write $\omega_1,\dots,\omega_k \in \{0,1\}$ for the components of $\omega$, and $|\omega| \coloneqq \omega_1+\dots+\omega_k$.  Similarly, if $\vec h \in G^k$ is a $k$-tuple in some additive group $G$, we write $h_1,\dots,h_k \in G$ for the components of $h$, and write $\omega \cdot \vec h$ for the ``dot product''
$$ \omega \cdot \vec h \coloneqq \omega_1 h_1 + \dots + \omega_k h_k.$$
We often identify $G^{k+1}$ with $G \times G^k$, thus for instance the assertion $(n,\vec h) \in G^{k+1}$ means that $n \in G$ and $\vec h \in G^k$.
\item[(ii)] If $f \colon G \to \C$ is a finitely supported function on an additive group $G$, we define the (unnormalized) Gowers uniformity norm $\|f\|_{\tilde U^k(G)}$ to be the quantity
$$ \|f\|_{\tilde U^k(G)} \coloneqq \left( \sum_{(n,\vec h) \in G^{k+1}} \prod_{\omega \in \{0,1\}^k} {\mathcal C}^{|\omega|} f(n+\omega \cdot \vec h) \right)^{1/2^k},$$
where ${\mathcal C} \colon z \mapsto \overline{z}$ denotes complex conjugation. If $G$ is finite, we then define the normalized norm
$$ \|f\|_{U^k(G)} \coloneqq \| f \|_{\tilde U^k(G)} / \| 1 \|_{\tilde U^k(G)}.$$
\item[(iii)] For any function $f \colon \Z \to \C$ and natural number $N$, we define the local (normalized) Gowers uniformity norm
$$ \| f\|_{U^k[N]} \coloneqq \| f 1_{[N]} \|_{\tilde U^k(\Z)} / \| 1_{[N]} \|_{\tilde U^k(\Z)}$$
where $1_{[N]}$ is the indicator function of $[N]$.
\end{itemize}
\end{definition}

Thus for instance
\begin{align*}
 \| f\|_{U^1[N]} &= \left(\E_{n,h:\, (n,n+h) \in [N]^2} f(n) \overline{f(n+h)}\right)^{1/2} \\
&= \left| \E_{n \in [N]} f(n) \right|,
\end{align*}
where throughout this paper we use the averaging notation
$$ \E_{a \in A} f(a) \coloneqq \frac{1}{\# A} \sum_{a \in A} f(a)$$
for any non-empty set $A$ of some finite cardinality $\# A$, and  by the orthogonality of additive characters we can compute
\begin{align*}
\|f\|_{U^2[N]} &= \left(\E_{n,h:\, (n,n+h,n+k,n+h+k) \in [N]^4} f(n) \overline{f(n+h)} \overline{f(n+k)} f(n+h+k)\right)^{1/4} \\
&\asymp N^{-3/4} \left( \int_0^1 \left|\sum_{n \in [N]} f(n) e(n\theta)\right|^4\ d\theta \right)^{1/4},
\end{align*}
where we adopt the usual asymptotic notation (see Section~\ref{notation-sec}), and $e(\theta) \coloneqq e^{2\pi i \theta}$. While we will permit the functions $f$ to be complex-valued for compatibility with previous literature (particularly those that invoke the circle method), in this paper we will deal almost exclusively with real-valued functions. As is well known, the Gowers uniformity norms are indeed norms for $k \geq 2$, and seminorms for $k=1$; see for instance~\cite[Appendix B]{gt-linear}.  In particular, they obey the triangle inequality 
\begin{equation}\label{triangle}
\|f+g\|_{U^k[N]} \leq \|f\|_{U^k[N]} + \|g\|_{U^k[N]}
\end{equation}
(and similarly for the other variants of the Gowers norms in Definition~\ref{gowers-norm}), which we will rely on frequently in this paper.

The \emph{M\"obius pseudorandomness principle} (see e.g.,~\cite[p. 338]{iwaniec-kowalski}) informally makes the prediction
$$ \mu(n) \approx 0$$
in the metric given by the Gowers norms $U^k[N]$.  Similarly, the usual modification of the Cram\'er random model~\cite{cramer}, as refined by Granville~\cite{granville} in order to take into account the distribution at primes below some threshold $w$, makes the prediction
$$ \Lambda(n) \approx \Lambda_{\Cramer,w}(n)$$
for various small $2 \leq w \ll N$, where $\Lambda_{\Cramer,w} \colon \Z \to \R$ is the function
$$
 \Lambda_{\Cramer,w}(n) \coloneqq \frac{P(w)}{\phi(P(w))} 1_{(n,P(w))=1} =\prod_{p<w} \frac{p}{p-1} 1_{p \nmid n}$$
where $P(w)$ is the primorial\footnote{In some texts the constraint $p \leq w$ is used in place of $p<w$; the precise convention is not too important for our applications, but the choice $p<w$ is consistent with the conventions in~\cite{friedlander-iwaniec}.} of $w$,
$$ P(w) \coloneqq \prod_{p < w} p,$$
with $\phi$ the Euler totient function and $(n,P(w))$ the greatest common divisor of $n$ and $P(w)$.  Thus for instance $\Lambda_{\Cramer,2}=1$ (which corresponds to the original model of Cram\'er).  The precise choice of the parameter $w$ is not too important, as can be shown by the following standard sieve-theoretic calculation:

\begin{proposition}[Gowers norm stability of the Cram\'er model]\label{cramer-stable}  If $2 \leq w, z \leq Q$, then
\begin{equation}\label{qa}
 \| \Lambda_{\Cramer,w} - \Lambda_{\Cramer,z} \|_{U^k[N]} \ll \log^{-c} N + w^{-c} + z^{-c}.
\end{equation}
\end{proposition}

We establish this proposition in Section~\ref{sec:sieve}.  In our applications it will be convenient to focus on the Cram\'er models $\Lambda_{\Cramer,w}, \Lambda_{\Cramer,z}$ with $w = \log^{\kappa} N$, $z = Q$, for $\kappa>0$ a sufficiently small constant which may depend on $k$ (usually we can take $\kappa = 1/100$). However, using Proposition~\ref{cramer-stable} it is not difficult to also work with other suitable choices of parameters if desired, at least up to  logarithmic decay (and probably up to \emph{pseudopolynomial decay}\footnote{By a pseudopolynomially decaying function we mean one that decays faster than $\exp(-\log^c N)$ for some $c>0$.} as well, see Remark~\ref{Effort}).

We summarize the previous Gowers uniformity results on M\"obius and von Mangoldt as follows.

\begin{theorem}[Gowers uniformity of M\"obius and von Mangoldt]\label{known-unif}\   
\begin{itemize}
\item[(i)]  (Pseudopolynomial $U^1$ uniformity) We have 
$$ \| \mu \|_{U^1[N]}, \| \Lambda - 1 \|_{U^1[N]}  \ll \exp( - c (\log N)^{3/5} (\log\log N)^{-1/5} ).$$
\item[(ii)]  (Logarithmic and strongly logarithmic $U^2$ uniformity)  We have
$$ \| \mu \|_{U^2[N]} \ll^\ineff_A \log^{-A} N$$
and
\begin{equation}\label{lambdam}
 \| \Lambda - \Lambda_{\Cramer,w} \|_{U^2[N]} \ll^\ineff \log^{-c} N + w^{-c}
\end{equation}
for all $A>0$ and all $2 \leq w \leq Q$.
\item[(iii)]  (Qualitative higher uniformity)  For any fixed $k > 2$, we have
$$ \| \mu \|_{U^k[N]} = o^\ineff(1)$$
and
\begin{equation}\label{law}
 \| \Lambda - \Lambda_{\Cramer,w} \|_{U^k[N]} \ll w^{-c} + o^\ineff(1)
\end{equation}
as $N \to \infty$ uniformly for any $2 \leq w \leq Q$.
\end{itemize}
In the asymptotic notation superscripted with $\ineff$, the implied constants are permitted to be ineffective.
\end{theorem}

A short deduction of this theorem from results stated in the literature is given in Appendix~\ref{appendix_qualitative} for the sake of completeness.

The first main objective of this paper is to quantify (and make effective) the qualitative rate of decay $o^\ineff(1)$ in Theorem~\ref{known-unif}(iii). We are able to obtain doubly logarithmic bounds which are weaker than the $k=2$ logarithmic bound in Theorem~\ref{known-unif}(ii) only by a single additional logarithm:

\begin{theorem}[Doubly logarithmic uniformity of M\"obius and von Mangoldt]\label{main} For $k\geq 2$, we have
$$ \| \mu \|_{U^k[N]} \ll (\log\log N)^{-c}$$
and
$$ \| \Lambda - \Lambda_{\Cramer,w} \|_{U^k[N]} \ll (\log\log N)^{-c} + w^{-c}$$
whenever $2 \leq w \leq Q$.
\end{theorem}

This is new for $k \geq 3$; henceforth we will assume $k \geq 2$ in our arguments to avoid some minor degeneracies. We remark that this theorem (and hence all of our subsequent results) are dependent on the results in \cite{manners} (see also \cite{bloom-survey}), which are currently available in preprint form as of this time of writing.

For later use, we also state a version of Theorem~\ref{main} for $\Lambda$ where the $W$-trick has been implemented.

\begin{corollary}[$W$-tricked quantitative Gowers uniformity]\label{cor:Wtrick}
Let $w=(\log \log N)^{1/2}$ and $W=\prod_{p\leq w}p$. Then for $k\geq 2$ we have
\begin{align*}
\left\|\frac{\phi(W)}{W}\Lambda(W\cdot+b)-1\right\|_{U^k[\frac{N-b}{W}]}\ll (\log \log N)^{-c} \end{align*}
whenever $1\leq b\leq W$ is coprime to $W$.
\end{corollary}

In Corollary~\ref{cor:Wtrick}, unlike in Theorem~\ref{main}, the size of $w$ turns out to be important. Indeed, if we had $w/\log \log N\to \infty$, then for all we know there could be a Siegel zero to some modulus $q\leq Q$ such that all its prime factors divided $W$, and this would bias the main term $1$ in Corollary~\ref{cor:Wtrick}; cf. Theorem~\ref{main-2}.

\subsection{Applications to linear equations in primes and to progressions with shifted prime difference}

The main application of the qualitative uniformity result \eqref{law} in~\cite{gt-linear} was to obtain qualitative asymptotics on linear equations in the primes; now using Theorem~\ref{main} we can make that result quantitative.

\begin{theorem}[Quantitative linear equations in primes]\label{quant-thm} Let $N,d,t,L$ be positive integers, and let $\Psi = (\psi_1,\dots,\psi_t)$ be a system of affine-linear forms $\psi_i \colon \Z^d \to \Z$ of the form
$$ \psi_i(n) = n \cdot \dot \psi_i + \psi_i(0) $$
where $\dot \psi_i \in \Z^d$, $\psi_i(0) \in \Z$ are such that $|\dot \psi_i| \leq L$ and $|\psi_i(0)| \leq LN$.  Suppose that no two of the $\dot \psi_i$ are linearly dependent.  Let $\Omega \subset [-N,N]^d$ be a convex body.  Then
\begin{align}\label{quant-lin-eq}
\sum_{\vec n \in \Omega \cap \Z^d} \prod_{i=1}^t \Lambda(\psi_i(\vec n)) = \beta_\infty \prod_p \beta_p + O_{t,d,L}(N^d (\log\log N)^{-c})
\end{align}
as $N \to \infty$, where  $c=c_{t,d,L}>0$ depends only on $t,d,L$, $\Lambda$ is extended by zero to the integers, $\beta_\infty$ is the Archimedean factor
$$ \beta_\infty = \mathrm{vol}( \Omega \cap  \Psi^{-1}(\R_{>0}^t) ),$$
and for each prime $p$, $\beta_p$ is the local factor
$$ \beta_p \coloneqq \mathbb{E}_{\vec n \in (\Z/p\Z)^d} \prod_{i=1}^t \frac{p}{p-1} 1_{\psi_i(\vec n) \neq 0}$$
(viewing each $\psi_i$ also as an affine map from $(\Z/p\Z)^d$ to $\Z/p\Z$).
\end{theorem}

Note that, in the language of~\cite{gt-linear}, the assumption that $\dot \psi_i$ are pairwise linearly independent is equivalent to these forms having ``finite Cauchy--Schwarz complexity''.

In~\cite{gt-linear}, the result of Theorem~\ref{quant-thm} was established with the qualitative error term $o_{t,d,L}^{\ineff}(N^d)$ in \eqref{quant-lin-eq} (initially under the hypotheses of the M\"obius and nilsequences conjecture and the inverse Gowers-norm conjecture, but these were later proved in~\cite{gt-mobius},~\cite{gtz-uk}).

We outline the (rather straightforward) details of the deduction of Theorem~\ref{quant-thm} from Theorem~\ref{main} in Section~\ref{sec:linprimes}.

\begin{example} In~\cite[Example 8]{gt-linear} it is shown that the number of (increasing) arithmetic progressions of primes of a given length $k \geq 2$ in $[N]$ is equal to
$$ \left( \frac{1}{2(k-1)} \prod_p \beta_p + o^\ineff(1) \right) \frac{N^2}{\log^k N} $$
where $\beta_p$ is equal to $\frac{1}{p} \left(\frac{p}{p-1}\right)^{k-1}$ when $p \leq k$, and $\left( 1 - \frac{k-1}{p} \right) \left( \frac{p}{p-1} \right)^{k-1}$ otherwise.  Inserting Theorem~\ref{quant-thm} into the arguments from~\cite{gt-linear}, the qualitative error term $o^\ineff(1)$ can now be improved to the doubly logarithmic error $O((\log\log N)^{-c})$.  This is new for $k \geq 4$.
\end{example}

Another application of Theorem~\ref{quant-thm} is to sets containing no progressions with shifted prime difference\footnote{We are indebted to Sean Prendiville for bringing this application to our attention.}. It was shown by S\'ark\"ozy~\cite{sarkozy} that (for $N$ large) any subset of $[N]$ of size $\gg N$ contains a pattern of the form $x,x+p-1$ with $p$ a prime. After several improvements~\cite{lucier},~\cite{ruzsa-sanders}~\cite{wang}, the current best known quantitative version of this, proved recently by Green~\cite{green-sarkozy}, is that 
any subset of $[N]$ of size $\geq N^{1-c}$ contains a pattern of this form. S\'ark\"ozy's theorem was later generalized to longer progressions by Frantzikinakis--Host--Kra~\cite{fhk}, and Wooley--Ziegler~\cite{wooley-ziegler}, who showed that, for any $k\geq 3$ and $N$ large enough in terms of $k$, any subset of $[N]$  of size $\gg N$ contains a pattern of the form $x,x+p-1,x+2(p-1),\ldots, x+(k-1)(p-1)$ with $p$ a prime, that is, a $k$-term arithmetic progression with shifted prime difference. These proofs however did not provide quantitative bounds for the density of a set avoiding $k$-term progressions with shifted prime difference. Using our main theorem, we can now obtain the first quantitative bound for this problem. 

\begin{theorem}[A quantitative bound for sets missing progressions with shifted prime difference]\label{thm_sarkozy}
 Let $k\geq 3$, and let N be large enough in terms of $k$. Then any subset of $[N]$ of size $\geq N(\log \log \log \log N)^{-c}$ contains a $k$-term arithmetic progression whose common difference is a shifted prime of the form $p-1$. Moreover, if $k=4$, one can replace $N(\log \log \log \log N)^{-c}$ with $N(\log \log \log N)^{-c}$ above, and if $k=3$, one can replace it with $N\exp(-(\log \log \log N)^{c})$ . 
\end{theorem}

The proof of this is given in Section~\ref{sec:sarkozy}.

\begin{remark} It is likely that one can similarly now make other qualitative consequences of \eqref{law} quantitative.  Certainly the version of the generalized Hardy--Littlewood conjecture in~\cite[Conjecture 1.2]{gt-linear} (in the finite complexity case) can now be made quantitative, with doubly logarithmic savings, in a manner perfectly analogous to Theorem~\ref{quant-thm}, as can the version of the main theorem in~\cite[Theorem 1.8]{gt-linear}; we omit the details.  The more recent asymptotics on linear inequalities in primes in~\cite{walker} are also likely to now have a doubly logarithmic quantitative version, but we do not pursue this matter here.

Lastly, one can also use Theorem~\ref{quant-thm} to quantify a result of the authors~\cite{tt-oddorder} on the logarithmically averaged Chowla conjecture for odd order correlations (whose proof relied on the Gowers uniformity of $\Lambda$). A back of the envelope calculation suggests that one could quantify the error term there, for fixed odd $k\geq 3$, to triply logarithmic; thus,
\begin{align}\label{oddorder}
\frac{1}{\log x}\sum_{n\leq x}\frac{\mu(n+h_1)\mu(n+h_2)\cdots \mu(n+h_k)}{n}\ll (\log \log \log x)^{-c}    
\end{align}
for any fixed integers $0\leq h_1< \cdots < h_k$ (and the same with the Liouville function in place of $\mu$). Very briefly, by the entropy decrement argument~\cite[Theorem 3.1]{tt-oddorder} one can locate a scale $\exp((\log \log x)^{1/2})\leq P\leq \log x$ such that the left-hand side of \eqref{oddorder} can be replaced up to triply logarithmic error term with
\begin{align*}
\frac{(-1)^k}{\log \log P}\sum_{p\leq P}\frac{1}{p}\frac{1}{\log x}\sum_{n\leq x}\frac{\mu(n+ph_1)\mu(n+ph_2)\cdots \mu(n+ph_k)}{n}.   
\end{align*}
One would then split the $p$ sum into dyadic scales and proceed as in~\cite{tt-oddorder} by replacing the average over primes $p$ with an average over $w$-rough integers, using Theorem~\ref{quant-thm} and a quantitative version of the generalized von Neumann theorem as a substitute for Theorem~\ref{known-unif}, producing an admissible $O((\log \log P)^{-c})$ error term. The triply logarithmic error terms at this step are much worse than any other error terms arising in the rest of the proof, therefore leading to \eqref{oddorder}. We leave the details to the interested reader.
\end{remark}

\section{Discussion and set-up of the proof}

Until recently, there were two main obstacles to achieving the sort of quantitative (and effective) bound stated in Theorem~\ref{main}.  Firstly, the first proofs of the inverse conjecture for the Gowers norms in the large $k$ regime $k \geq 5$ were ineffective (using tools such as nonstandard analysis) and did not provide any quantitative dependence of constants.  Secondly, in order to overcome certain logarithmic losses in the estimates, it was necessary to invoke Siegel's theorem to control the correlation of the M\"obius function with nilsequences, and the decay rate in the $o(1)$ bounds in Theorem~\ref{known-unif}(iii) then depended on the rate at which the constants in Siegel's theorem $|L(1,\chi)| \gg^\ineff_\eps q^{-\eps}$ depended on $\eps$, which is completely ineffective with known methods.

The first issue was resolved recently with the quantitative inverse theorem of Manners~\cite{manners}, which provided a good quantitative dependence on all parameters in the inverse theory of Gowers norms.  
To resolve the second issue, we perform the technique of isolating out the contribution of a potential Siegel zero to obtain more refined approximations 
\begin{align*}
 \mu &\approx \mu_\Siegel(n) \\
 \Lambda &\approx \Lambda_\Siegel(n)
\end{align*}
to the arithmetic functions $\mu,\Lambda$.  To make this precise we introduce some notation:

\begin{definition}[Siegel model]\label{siegel-model}\ Recall that the quantity $Q$ was defined in \eqref{Q-def}.  
\begin{itemize}
\item[(i)]  We define a \emph{$Q$-Siegel zero} to be a real number $1 - \frac{c_0}{\log Q} < \beta < 1$ for which there exists a primitive real Dirichlet character $\chi_\Siegel$ (which we call the \emph{$Q$-Siegel character}) of conductor $q_\Siegel < Q$ such that $L(\beta,\chi_\Siegel)=0$, where $L(s,\chi)$ denotes the Dirichlet $L$-function associated to $\chi$.  Here $c_0$ is a sufficiently small absolute constant (and henceforth all implied constants are permitted to depend on $c_0$).  Note from the Landau--Page theorem (see e.g.,~\cite[Corollary 11.10]{mv}) that if a $Q$-Siegel zero exists, then it is unique (and similarly for the $Q$-Siegel character), and the zero $\beta$ is simple (so that $L'(\beta,\chi_\Siegel) \neq 0$).
\item[(ii)]  We define the \emph{$Q$-Siegel model} $\Lambda_{\Siegel}$ for the von Mangoldt function $\Lambda$ to be
$$ \Lambda_{\Siegel}(n) \coloneqq \Lambda_{\Cramer,Q}(n) = \frac{P(Q)}{\phi(P(Q))} 1_{(n,P(Q))=1}$$
if no $Q$-Siegel zero exists, and
$$ \Lambda_{\Siegel}(n) \coloneqq \Lambda_{\Cramer,Q}(n) \left( 1 - 
n^{\beta-1} \chi_\Siegel(n) \right)$$
otherwise.
\item[(iii)]  We define the \emph{$Q$-Siegel model} $\mu_\Siegel$ for the M\"obius function $\mu$ to be
$$ \mu_\Siegel(n) \coloneqq 0$$
if no $Q$-Siegel zero exists, and
\begin{equation}\label{mus-def}
 \mu_\Siegel(n) \coloneqq \mu_\local * \mu'(n)
\end{equation}
otherwise, where $\mu_\local$ is the local M\"obius function
$$ \mu_\local(n) \coloneqq \mu(n) 1_{n|P(Q)},$$
$\mu'$ is the function
\begin{equation}\label{mup-def}
 \mu'(n) \coloneqq \alpha n^{\beta-1} \chi_\Siegel(n) 1_{(n,P(Q))=1} = \alpha \frac{\phi(P(Q))}{P(Q)} (\Lambda_{\Cramer,Q}(n) - \Lambda_\Siegel(n)),
\end{equation}
$\alpha$ is the quantity
\begin{equation}\label{alpha-def}
\alpha \coloneqq \frac{1}{L'(\beta,\chi_\Siegel)} \prod_{p<Q} \left(1-\frac{1}{p}\right)^{-1} \left(1-\frac{\chi_\Siegel(p)}{p^\beta}\right)^{-1},
\end{equation}
and $\mu_\local * \mu'$ is the Dirichlet convolution of $\mu_\local$ and $\mu'$:
$$ \mu_\local * \mu'(n) \coloneqq \sum_{d|n} \mu_\local(d) \mu'(n/d).$$
(Note from the supports of $\mu_\local, \mu'$ that at most one term in this sum is non-zero for any given $n$.)
\end{itemize}
\end{definition}

The significance of these models is that $\Lambda$ and $\Lambda_{\Siegel}$ have very nearly the same statistics on arithmetic progressions (with error terms that improve over the main term by pseudopolynomial factors $O(\exp(-\log^c N))$, which are superior to the strongly logarithmic gains $O^\ineff_A(\log^{-A} N)$ provided by the Siegel--Walfisz theorem), and similarly for $\mu$ and $\mu_{\Siegel}$.  Indeed, in Section~\ref{sec:mobius} we will show the following estimates:

\begin{proposition}[Pseudopolynomial equidistribution in arithmetic progressions]\label{equiprop}  For any arithmetic progression $P \subset [N]$, we have
\begin{equation}\label{mu-equi-p}
 \sum_{n \in P} (\mu(n) - \mu_\Siegel(n)) \ll N \exp(-c \log^{1/10} N).
\end{equation}
and
\begin{equation}\label{lam-equi-p}
 \sum_{n \in P} (\Lambda(n) - \Lambda_{\Siegel}(n)) \ll N \exp(-c \log^{1/10} N)
\end{equation}
\end{proposition}

\begin{remark}
The construction of $\mu_\Siegel$ appears to be complicated, but it is a multiplicative construction and can be justified as follows. If $\chi$ is any character induced from $\chi_\Siegel$ of some period $q | [q_\Siegel,P(Q)]$, a short calculation reveals the Euler products
\begin{equation}\label{dir-1}
\begin{split} 
\sum_{n=1}^\infty \frac{\mu_\Siegel\chi(n) }{n^s}
&= 
\frac{\zeta(s+1-\beta)}{L'(\beta,\chi_\Siegel)}  \prod_{\substack{p<Q:\\ p \mid q}} \frac{1-\frac{\chi_\Siegel(p)}{p^s}}{1-\frac{\chi_\Siegel(p)}{p^\beta}}\prod_{p<Q} \frac{1-\frac{1}{p^{s+1-\beta}}}{1-\frac{1}{p}} 
\\
&\quad \times \prod_{p|q} \left(1-\frac{\chi_\Siegel(p)}{p^\beta}\right)^{-1} 
\end{split}
\end{equation}
and
\begin{equation}\label{dir-2}
\sum_{n=1}^\infty \frac{\mu\chi(n)}{n^s}
= \frac{1}{L(s,\chi_\Siegel)} \prod_{p|q} \left(1-\frac{\chi_\Siegel(p)}{p^s}\right)^{-1}
\end{equation}
whenever $\mathrm{Re}(s) > 1$.  One can then check that the meromorphic continuations of the two Dirichlet series \eqref{dir-1}, \eqref{dir-2} both have a simple pole at $s=\beta$ with the same residue (and when $\chi$ is not induced from $\chi_\Siegel$ there is no such pole), which helps justify why we expect $\mu_\Siegel$ to be a good approximation to $\mu$.  We experimented with simpler models to $\mu$ than $\mu_\Siegel$, but in order to get the pseudopolynomial error terms $\exp(-\log^c N)$ in \eqref{mu-equi-p} it seems essential that the model $\mu_\Siegel$ behaves almost identically to $\mu$ with respect to primes $p$ as large as $\exp(\log^c N)$, which necessitates a complicated construction such as \eqref{mus-def}.  We remark that a similar (though slightly less refined) approximant $\lambda_\Siegel$ to the Liouville function $\lambda$ was introduced by Germ\'an and Katai in~\cite{german-katai}, and recently used in~\cite{chinis} to establish Chowla's conjecture in the presence of a Siegel zero.
\end{remark}

For future reference we also observe the following crude pointwise bounds on $\Lambda, \mu$ and their approximate models:

\begin{lemma}[Pointwise bounds]\label{point-bound}  For $n \in [N]$ and $2 \leq w \leq Q$, one has
$$ \Lambda(n), \Lambda_{\Cramer,w}(n), \Lambda_\Siegel(n) \ll \log N$$
and
$$ \mu(n), \mu_\Siegel(n) \ll 1.$$
\end{lemma}

\begin{proof} All of these bounds are either trivial or immediate consequences of Mertens' theorem, except for the bound on $\mu_\Siegel$, which would follow if the quantity $\alpha$ in \eqref{alpha-def} were bounded.  This turns out to follow from standard bounds on the $L$-function $L(s,\chi_\Siegel)$ near a $Q$-Siegel zero $\beta$; see Lemma~\ref{alpha-bound}.
\end{proof}

In view of Proposition~\ref{cramer-stable} and the triangle inequality \eqref{triangle}, Theorem~\ref{main} then follows from the following two statements.

\begin{theorem}[Siegel corrections are logarithmically Gowers uniform]\label{siegel-uniform}  We have
\begin{equation}\label{musig}
 \| \mu_\Siegel \|_{U^k[N]} \ll q_\Siegel^{-c} \ll \log^{-c} N 
\end{equation}
and
\begin{equation}\label{lamsig}
 \| \Lambda_{\Siegel} - \Lambda_{\Cramer,Q} \|_{U^k[N]} \ll q_\Siegel^{-c} \ll \log^{-c} N
\end{equation}
with the convention that the expression $q_\Siegel^{-c}$ vanishes when no $Q$-Siegel zero exists.
\end{theorem}

\begin{theorem}[Doubly logarithmic uniformity of M\"obius and von Mangoldt, II]\label{main-2}  We have
\begin{equation}\label{musi}
 \| \mu - \mu_\Siegel \|_{U^k[N]} \ll (\log\log N)^{-c}
\end{equation}
and
\begin{equation}\label{lasi}
 \| \Lambda - \Lambda_{\Siegel} \|_{U^k[N]} \ll (\log\log N)^{-c}.
\end{equation}
\end{theorem}

Theorem~\ref{siegel-uniform} is an application of sieve-theoretic methods, smooth number estimates and the Weil bound, and is established in Section~\ref{siegel-sec}.  The main difficulty is to establish Theorem~\ref{main-2}.  In principle,  one can directly apply the quantitative inverse theory of Manners~\cite{manners}, and reduce matters to controlling the correlation of $\mu - \mu_\Siegel$, $\Lambda - \Lambda_\Siegel$ with nilsequences arising from nilmanifolds (although in the case of $\Lambda - \Lambda_\Siegel$ we have the obstacle that the function is unbounded -- the resolution of this is discussed below).  Indeed, in Section~\ref{sec:mobius} we will establish the following bounds that significantly extend the bounds in Proposition~\ref{equiprop}:

\begin{theorem}[Pseudopolynomial orthogonality of M\"obius and von Mangoldt with nilsequences]\label{quant-ortho} Let $\epsilon>0$ and $k\geq 1$. Let $c_1(\epsilon)>0$ be small enough in terms of $\epsilon$. Then we have the bounds 
\begin{equation}\label{mmus}
 \sum_{n \in P} (\mu - \mu_{\Siegel})(n) \overline{F}(g(n) \Gamma) \ll_{\epsilon} N \exp(-  \log^{1/10-\epsilon} N)
\end{equation}
and
\begin{equation}\label{mmus-2}
 \sum_{n \in P} (\Lambda - \Lambda_{\Siegel})(n) \overline{F}(g(n) \Gamma) \ll_{\epsilon} N \exp(- \log^{1/10-\epsilon} N)
\end{equation}
whenever $P \subset [N]$ is an arithmetic progression, $G/\Gamma$ is a filtered nilmanifold of degree $k-1$, dimension at most $(\log\log N)^{c_1(\epsilon)}$, and complexity at most $\exp(\log^{c_1(\epsilon)} N)$, $F: G/\Gamma \to \C$ is a $1$-bounded Lipschitz function\footnote{A function $F \colon X \to \C$ is \emph{$1$-bounded} if $|F(x)| \leq 1$ for all $x \in X$.  More generally, given any $\nu \colon X \to \R^+$, we say that $F \colon X \to \C$ is \emph{$\nu$-bounded} if $|F(x)| \leq \nu(x)$ for all $x \in X$.}  of Lipschitz constant at most $\exp(\log^{1/10-\epsilon} N)$, and $g \colon \Z \to G$ is a polynomial map.   (The relevant definitions of filtered nilmanifolds, etc., are reviewed in Definition~\ref{filnil}.)
\end{theorem}

\begin{remark}
If one redefined the Siegel models $\mu_{\Siegel}, \Lambda_{\Siegel}$ by assigning the parameter $Q$ the larger value $\exp((\log N)^{1/2})$, one could inspect that the exponent of logarithm in \eqref{mmus} and \eqref{mmus-2} (and in particular in Proposition~\ref{equiprop}) could be increased to $1/2-\epsilon$, hence essentially matching the shape of the error term in the classical prime number theorem. For this modification, one would have to tweak the exponents in Section~\ref{sec:sieve} a little; in particular in Proposition~\ref{lineq} the exponents $3/5$ and $4/5$ would have to be replaced with $1/2$. As the precise value of the exponent has very little influence on our bounds, we leave the details of this strengthening to the interested reader. 
\end{remark}

For sake of comparison, in~\cite{gt-mobius} the strongly logarithmic bound
$$
 \E_{n \in [N]} \mu(n) \overline{F}(g(n) \Gamma) \ll_{A,M}^\ineff \log^{-A} N$$
was established for any $A>0$ assuming that the dimension and complexity of $G/\Gamma$ and the Lipschitz constant of $F$ were all bounded by $M$;
using this bound, in~\cite{gt-linear} the qualitative bound
$$
 \E_{n \in [N]} \left(\frac{\phi(W)}{W} \Lambda(Wn+b) - 1\right) \overline{F}(g(n) \Gamma) = o^\ineff(1)$$
was shown for the same type of nilsequences $F(g(n) \Gamma)$, where $W = P(w)$ for some $w = w(N)$ growing sufficiently slowly to infinity with $N$ and any $b \in [W]$ coprime to $W$.  With a little additional effort, the latter bound then also implies the qualitative bound
$$
 \sum_{n \in P} (\Lambda - \Lambda_{\Cramer,w})(n) \overline{F}(g(n) \Gamma) = o^\ineff(N)$$
for these nilsequences and arbitrary arithmetic progressions $P \subset [N]$.  The arguments relied upon (and in fact imply) the Siegel--Walfisz theorem and thus could not give error terms better than strongly logarithmic, which would be unsuitable for our applications (particularly those involving the von Mangoldt function).  It is therefore necessary to account for the correction terms $\mu_\Siegel, \Lambda_\Siegel - \Lambda_{\Cramer,Q}$ to avoid any appeal to the Siegel--Walfisz theorem and to improve the bounds to be of pseudopolynomial type, despite the fact (from Theorem~\ref{siegel-uniform}) that these correction terms are already logarithmically small in the Gowers norm sense. 

Our proof of Theorem~\ref{quant-ortho} will broadly follow the same strategy as that in~\cite{gt-mobius}, relying on Proposition~\ref{equiprop} in the ``major arc'' case and on decomposition into ``Type I'' and ``Type II'' sums, followed by Cauchy--Schwarz and an appeal to the equidistribution theory of nilmanifolds, in the ``minor arc'' case.  A key new feature, compared to previous work, is that the dimension of the nilsequences is no longer bounded, but grows at a roughly doubly logarithmic rate in $N$.  Because of this, we are forced to perform a careful accounting on the dependence on dimension in the aforementioned equidistribution theory, and in particular ensure that the bounds only depend at most doubly exponentially on the dimension.  This is in fact one of the main reasons why our bounds in Theorem~\ref{main} are limited to be doubly logarithmic in nature; see Remarks~\ref{improv},~\ref{improv-2} below.

The estimate \eqref{musi} can be directly obtained from \eqref{mmus} using the inverse theorem of Manners~\cite{manners}, which we review in Section~\ref{sec:manners}; note that this theorem basically applies a double logarithm to the quantitative bounds, which is why the pseudopolynomial type terms in Theorem~\ref{quant-ortho} are reduced to doubly logarithmic type terms in Theorem~\ref{main-2}.  For the von Mangoldt estimate \eqref{lasi}, we encounter the familiar problem that $\Lambda-\Lambda_{\Siegel}$ is not bounded (see Lemma~\ref{point-bound}), so that Manners' quantitative inverse theorem does not immediately apply.  In~\cite{gt-linear}, this difficulty was resolved at the qualitative level by first using the ``$W$-trick'' of passing to an arithmetic progression $\{ Wn+b: n \in \N \}$ for some $W = P(w)$ and some $w$ growing slowly with $N$, and then dominating (an appropriately normalized version of) the von Mangoldt function on that progression by a divisor sum $\nu$ of Goldston--Y{\i}ld{\i}r{\i}m type that obeyed some ``pseudorandomness'' conditions. This enabled one to then apply a transference principle that roughly speaking allowed one to behave ``as if'' the normalized von Mangoldt function was bounded on this progression, at least for the purposes of applying an inverse theorem for the Gowers norms.

Here the biggest source of quantitative inefficiency is the transference principle, as the first few proofs of this principle~\cite{gt-longaps}, \cite{gt-linear}, \cite{gowers-hb}, \cite{RTTV} involved the Weierstrass approximation theorem, quantitative versions of which can generate exponential type losses.  However, in~\cite{cfz-rel} (see also~\cite{cfz}), Conlon, Fox, and Zhao introduced the method of \emph{densification}, which they used to obtain a transference principle in the context of Szemer\'edi-type theorems that involved only polynomial dependencies on the bounds (and they also relaxed the pseudorandomness hypotheses on the enveloping sieve $\nu$ by dropping the so-called ``correlation condition'').  As it turns out, the densification method can be adapted to inverse theorems as well with efficient quantitative bounds, at least when the correlation in the inverse theorem enjoys polynomial bounds; we formalize this observation (which seems to be of independent interest) as Theorem~\ref{transf}.  Fortunately for us, the arguments of Manners in~\cite[\S 5]{manners} already provide such a polynomial bound.  Using our quantitative transference result for the inverse theorem, it becomes a relatively routine matter to derive \eqref{lasi} from \eqref{mmus-2}, after making various necessary quantitative refinements (for instance, the parameter $w$ will now be taken to be of the shape $\log^\eps N$ for some small $\eps>0$, rather than growing in some unspecified slow fashion with $N$).  This will all be performed in Section~\ref{sec: densification}.

\begin{remark}\label{improv} Perhaps surprisingly, the bounds in Theorem~\ref{main} are not significantly improved if one assumes the generalized Riemann hypothesis; some pseudopolynomial bounds can now be sharpened to polynomial bounds (such as Theorem~\ref{quant-ortho}), but for the logarithmic and doubly logarithmic bounds only minor improvements in the unspecified constants $c$ are available under GRH (though of course in this case any terms involving $Q$-Siegel zeroes can simply be deleted).  On the other hand, it is tempting to conjecture that the doubly logarithmic bounds in our main results can be improved to logarithmic, given that several of the key estimates already have this quality of error term or better.  This is particularly appealing in the $k=3$ case where we have quite a good inverse $U^3$ theorem~\cite{gt-U3}.  The main difficulty is that to achieve this goal, it appears that one needs an equidistribution theory for $2$-step nilmanifolds (or quadratic bracket polynomials) that involves exponents that are merely \emph{polynomial} in the dimension of the nilmanifold (or complexity of the bracket polynomial) rather than exponential.  In analogy with the well known quadratic Diophantine approximation theory of Schmidt~\cite{schmidt}, it seems reasonable to expect such a theory to be feasible\footnote{Another option is to exploit improved the dimension bounds for the inverse $U^3$ theory now available \cite{sanders}, using the equivalences from \cite{gt-sumset}.  Since the initial release of this preprint, this option has in fact been carried out by Leng \cite{leng}, who significantly improved the $(\log \log N)^{-c}$ type bounds in Theorem \ref{main-2} to $\exp(-\log^c N)$ type bounds in the $k=3$ case.}, but we will not pursue this matter here.  On the other hand, we note that by combining Theorem~\ref{quant-ortho} with the circle method one can obtain the pseudopolynomial bounds 
$$\| \mu - \mu_\Siegel \|_{U^2[N]}, \| \Lambda - \Lambda_\Siegel \|_{U^2[N]} \ll\exp(-c\log^c N),$$ 
and one could optimistically conjecture that such pseudopolynomial (or even polynomial) bounds are also true for higher Gowers norms as well (such bounds would follow from a sufficiently uniform version of the Hardy--Littlewood prime tuples conjecture).
\end{remark}

\subsection{Acknowledgments}

TT was supported by a Simons Investigator grant, the James and Carol Collins Chair, the Mathematical Analysis \& Application Research Fund Endowment, and by NSF grant DMS-1764034. JT was supported by a Titchmarsh Fellowship and funding from
the European Union's Horizon
Europe research and innovation programme under Marie Sk\l{}odowska-Curie grant agreement no.
101058904. We thank the anonymous referee for a careful reading of the paper and for numerous helpful corrections. We thank Sean Prendiville for helpful discussions, and Andrew Granville, James Leng and Wataru Kai for corrections.

\section{Notation}\label{notation-sec}

As stated in the introduction, throughout this paper we fix an integer $k \geq 1$, and assume $N$ is a positive real number that is sufficiently large depending on $k$ (and $Q$ is given in terms of $N$ by \eqref{Q-def}).  We abbreviate $\{ n \in \N: 1 \leq n \leq N \}$ as $[N]$ (even when $N$ is not an integer).

We use the asymptotic notation $X \ll Y$, $Y \gg X$, or $X = O(Y)$ to denote an estimate of the form $|X| \leq CY$ for some constant $C>0$.  If $C$ depends on additional parameters, we indicate this by subscripts, for instance $X = O_d(Y)$ denotes the estimate $|X| \leq C_d Y$ for some $C_d>0$ depending on $d$.  However, as all of our constants will depend on the fixed parameter $k$, we omit this parameter from this subscripting notation.  Unless otherwise specified, the constants will depend in an effective fashion on the parameters; on the rare occasions (mostly involving citing previous literature) in which ineffective constants are used, we will use the superscript $\ineff$ to indicate this.  We write $X \asymp Y$ as an abbreviation for $X \ll Y \ll X$, subject to the same subscripting and superscripting conventions as before.  If $X,Y$ depend on an additional parameter $N$, we write $X = o(Y)$ as $N \to \infty$ to denote the claim that $|X| \leq c(N) Y$ for some quantity $c(N)$ that goes to zero as $N \to \infty$, again subject to the same subscripting and superscripting conventions as before.  As stated in the introduction, we use $c$ to denote various small positive constants depending on $k$ that can vary from line to line.

 We often refer to the following hierarchy of decay estimates, in increasing order of strength:
\begin{itemize}
\item \emph{Qualitative (and ineffective) decay}, in which $X = o^\ineff(Y)$ as $N \to \infty$;
\item \emph{Doubly logarithmic decay}, in which $X \ll (\log\log N)^{-c} Y$;
\item \emph{Logarithmic decay}, in which $X \ll (\log N)^{-c} Y$;
\item \emph{Strongly (but ineffectively) logarithmic decay}, in which $X \ll_A^\ineff (\log N)^{-A} Y$ for any $A>0$ (this is a typical shape for bounds obtained using the Siegel--Walfisz theorem);
\item \emph{Pseudopolynomial decay}, in which $X \ll \exp(-c \log^c N) Y$; and
\item \emph{Polynomial decay}, in which $X \ll N^{-c} Y$.
\end{itemize}
As the terminology suggests, pseudopolynomial decay will be a satisfactory substitute for polynomial decay in many of our arguments. 

We use $1_E$ to denote the indicator function of a set $E$, thus $1_E(n)$ equals $1$ when $n\in E$ and $0$ otherwise.  We also use $1_S$ to denote the indicator of a statement $S$, thus $1_S$ equals $1$ when $S$ is true and $0$ otherwise.

If $A$ is a finite set, we use $\# A$ to denote its cardinality.

All sums and products over the variable $p$ are understood to be over primes, and similarly all sums and products over variables such as $n$ or $d$ are understood to be over natural numbers, unless otherwise indicated.

\section{Some lemmas on Gowers norms}\label{sec: lemmas}

We state here a few lemmas concerning the Gowers norms that will be used later on.

In addition to the triangle inequality \eqref{triangle}, we shall also often use the closely related \emph{Gowers--Cauchy--Schwarz inequality}
\begin{equation}\label{gcz}
 \left|\E_{(x,\vec h) \in G^{k+1}} \prod_{\omega\in \{0,1\}^k} f_\omega(x+\omega \cdot \vec h)\right|
\leq \prod_{\omega \in \{0,1\}^k} \|f_\omega\|_{U^k(G)}
\end{equation}
for any finite additive group $G$ and any functions $f_\omega \colon G \to \C$ for $\omega \in \{0,1\}^k$; see for instance~\cite[Lemma B.2]{gt-linear}.  For arbitrary additive groups, we also have the non-normalized variant
\begin{equation}\label{gcz-norm}
 \left|\sum_{(x,\vec h) \in G^{k+1}} \prod_{\omega\in \{0,1\}^k} f_\omega(x+\omega \cdot \vec h)\right|
\leq \prod_{\omega \in \{0,1\}^k} \|f_\omega\|_{\tilde U^k(G)}.
\end{equation}

Observe that the Gowers norms behave well with respect to tensor products: if $f_1 \colon G_1 \to \C$, $f_2 \colon G_2 \to \C$ are finitely supported functions on additive groups $G_1,G_2$, then a short computation reveals that
\begin{equation}\label{tensor}
\|f_1 \otimes f_2 \|_{\tilde U^k(G_1 \times G_2)} = \|f_1\|_{\tilde U^k(G_1)} \|f_2\|_{\tilde U^k(G_2)}
\end{equation}
for any $k \geq 1$.

We now develop a variant of this identity \eqref{tensor}.
We localize the Gowers norm to cosets $a+H$ of a subgroup $H$ of an additive group $G$ as follows: if $k \geq 1$ and $f \colon G \to \C$ is finitely supported, we define $\|f\|_{\tilde U^k(a+H)} \coloneqq \| f(a+\cdot) \|_{\tilde U^k(H)}$, and similarly $\|f\|_{U^k(a+H)} \coloneqq \| f(a+\cdot) \|_{U^k(H)}$ if $H$ is finite.  Note that this definition does not depend on the choice of coset representative.  We have the following convenient Fubini type inequality (which is reasonably well known ``folklore'', although the only explicit prior reference to such an inequality that we are aware of is \cite[Lemma 4.3]{cladek}):

\begin{lemma}[Fubini type inequality]\label{fubini}  Let $k \geq 1$, let $G$ be an additive group, let $H$ be a subgroup of $G$, and let $f \colon G \to \C$ be a finitely supported function.  For each coset $a+H$ in the quotient group $G/H$, let $F(a+H)$ denote the quantity
$$ F(a+H) \coloneqq \|f\|_{\tilde U^k(a+H)};$$
note that $F \colon G/H \to \C$ is also a finitely supported function.  Then we have
\begin{equation}\label{fub}
 \|f\|_{\tilde U^k(G)} \leq \|F\|_{\tilde U^k(G/H)}.
\end{equation}
\end{lemma}

Informally, this lemma asserts that to bound the $U^k(G)$ norm of a function $f$, one can first evaluate the $\tilde U^k$ norm along the various cosets of $H$, and then compute the $\tilde U^k$ norm of the numbers obtained in that fashion.  If $G,H$ are finite we can obtain similar claims for the normalized $U^k$ norms in the obvious fashion. Note that the Fubini--Tonelli theorem establishes a similar claim for the $\ell^1$ (or more generally $\ell^p$) norms (and in this case one has equality in \eqref{fub} instead of inequality.  One can also verify that \eqref{fub} is consistent with \eqref{tensor}.

\begin{proof}
From Definition~\ref{gowers-norm} we have
$$ \|f\|_{\tilde U^k(G)}^{2^k} = \sum_{(n,\vec h) \in G^{k+1}} \prod_{\omega \in \{0,1\}^k} {\mathcal C}^{|\omega|} f(n+\omega \cdot \vec h).$$
Consider the contribution to the right-hand side where $n$ lies in a coset $a+H$ and $h_i$ lies in a coset $b_i+H$ for $i=1,\dots,k$.  
By the Gowers--Cauchy--Schwarz inequality \eqref{gcz-norm}, this contribution can be bounded in magnitude by
$$ \prod_{\omega \in \{0,1\}^k} F( a+\omega \cdot \vec b + H)$$
where $\vec b \coloneqq (b_1,\dots,b_k)$. Summing over all choices of $a,\vec b$ and applying Definition~\ref{gowers-norm} again, we conclude that
$$ \|f\|_{\tilde U^k(G)}^{2^k} \leq \|F\|_{\tilde U^k(G/H)}^{2^k}$$
giving \eqref{fub}.
\end{proof}

As a corollary of this inequality, we can estimate the Gowers norm of a function on $[N]$ in terms of its values on various arithmetic progressions:

\begin{corollary}[$W$-trick]\label{slice}  Let $1 \leq W \leq \frac{N}{10}$, and let $f: [N] \to \C$ be a function supported on the set $\{ n \in [N]: (n,W)=1\}$ that obeys the bounds
$$\left\| \frac{\phi(W)}{W} f(W\cdot+b) \right\|_{U^k[\frac{N-b}{W}]} \leq A$$
for all $b \in [W]$ coprime to $W$ and some $A>0$.  Then one has
$$ \|f\|_{U^k[N]} \ll A.$$
\end{corollary}

\begin{proof}  We extend $f$ by zero to the integers $\Z$ and work with the unnormalized Gowers norms.  Since
$$ \|1_{[N]} \|_{\tilde U^k(\Z)} \asymp N^{\frac{k+1}{2^k}}$$
and
$$ \|1_{[\frac{N-b}{W}]} \|_{\tilde U^k(\Z)} \asymp (N/W)^{\frac{k+1}{2^k}} $$
we have
$$ \| f\|_{\tilde U^k(W\Z+b)} \ll \frac{W}{\phi(W)} A (N/W)^{\frac{k+1}{2^k}} $$
for all $b \in [W]$ coprime to $W$,
and it will suffice to show that
$$ \|f\|_{\tilde U^k(\Z)} \ll A N^{\frac{k+1}{2^k}}.$$
Applying Lemma~\ref{fubini} with $G=\Z$ and $H=W\Z$, and normalizing the Gowers norms, it suffices to show that
$$ \left\| \frac{W}{\phi(W)} 1_{(\cdot,W)=1} \right\|_{U^k(\Z/W\Z)} \ll 1.$$
Expressing $W$ as the product of primes $p^{v_p(W)}$ and using the Chinese remainder theorem and \eqref{tensor} repeatedly, the left-hand side can be written as
$$ \prod_{p}  \left\| \frac{p}{p-1} 1_{(\cdot,p)=1} \right\|_{U^k(\Z/p^{v_p(W)}\Z)}.$$
However, direct computation using the inclusion-exclusion principle shows that
$$\| 1_{(\cdot,p)=1} \|_{U^k(\Z/p^{v_p(W)}\Z)}^{2^k} = 1 - \frac{2^k}{p} + O_k\left(\frac{1}{p^2}\right),$$
and hence
$$\left\| \frac{p}{p-1} 1_{(\cdot,p)=1} \right\|_{U^k(\Z/p^{v_p(W)}\Z)} = 1 + O_k\left( \frac{1}{p^2} \right).$$
The claim follows.
\end{proof}

Next, we give a variant of the triangle inequality that estimates a Gowers norm based on the greatest common divisor with a fixed modulus.

\begin{lemma}[Variant of triangle inequality]\label{variant}  Let $N \geq 100$, let $1 \leq q \leq N$,  and let $k \geq 1$ be an integer.  Let $f: [N] \to [-1,1]$ be a function.  Then
$$ \|f\|_{U^k[N]}^{2^k} \ll \sum_{d|q} \frac{1}{d} \| f(d \cdot) 1_{(\cdot,q/d)=1} \|_{U^k[N/d]}.$$
\end{lemma}

The key point here is the presence of the factor $\frac{1}{d}$, which ensures that the summation over $d$ can be estimated manageably.

\begin{proof}  We extend $f$ by zero outside of $[N]$.  From Definition~\ref{gowers-norm}, it suffices to show the unnormalized estimate
$$ \sum_{(n,\vec h) \in \Z^{k+1}} \prod_{\omega \in \{0,1\}^k} f(n+\omega \cdot \vec h)
\ll N^{k+1} \sum_{d|q} \frac{1}{d (N/d)^{\frac{k+1}{2^k}}} \| f 1_{(\cdot,q)=d} \|_{\tilde U^k(d\Z)}.$$
The left-hand side can be written as
$$ \sum_{n\in \Z} f(n) F(n)$$
where the \emph{dual function} $F(n)$ is defined as
$$ F(n) \coloneqq \sum_{\vec h \in \Z^k} \prod_{\omega \in \{0,1\}^k \backslash \{0\}^k} f(n+\omega \cdot \vec h).$$
We split this sum in terms of the value of $(n,q)$ as
$$ \sum_{n \in \Z} f(n) F(n) = \sum_{d|q} \sum_{n \in d\Z} f(n) 1_{(n,q)=d} F(n).$$
By the triangle inequality, it thus suffices to show that
$$ \sum_{n \in d\Z} f(n) 1_{(n,q)=d} F(n)
\ll \frac{N^{k+1}}{d (N/d)^{\frac{k+1}{2^k}}} \| f 1_{(\cdot,q)=d} \|_{\tilde U^k(d\Z)}$$
for each $d|q$.  Decomposing $h_1,\dots,h_k$ in the definition of $F(n)$ into cosets mod $d$, the left-hand side may be written as
$$ \sum_{\vec b \in [d]^k}
\sum_{(n,\vec h) \in (d\Z)^{k+1}} \prod_{\omega \in \{0,1\}^{k}} f(n+\omega \cdot (\vec h+\vec b)) 1_{(n,q)=d}.$$
By the Gowers--Cauchy--Schwarz inequality \eqref{gcz}, and noting that $f$ is bounded by $1_{[N]}$, we have
$$\sum_{(n,\vec h) \in (d\Z)^{k+1}} \prod_{\omega \in \{0,1\}^k} f(n+\omega \cdot (\vec h+\vec b)) 1_{(n,q)=d}
\ll \| f 1_{(\cdot,q)=d} \|_{\tilde U^k(d\Z)} ( (N/d)^{\frac{k+1}{2^k}} )^{2^k-1}.
$$
Summing over all the $d^k$ choices of $\vec b$, we thus obtain
$$ \sum_{n \in d\Z} f(n) 1_{(n,q)=d} F(n)
\ll \| f 1_{(\cdot,q)=d} \|_{\tilde U^k(d\Z)} d^k \left( (N/d)^{\frac{k+1}{2^k}} \right)^{2^k-1}$$
and the claim follows after a little algebra.
\end{proof}

\section{Some sieve theory}\label{sec:sieve}

\subsection{The Cram\'er model} In this section we use some standard sieve-theoretic tools to establish several estimates involving the Cram\'er models $\Lambda_{\Cramer,w}$, some of which will also be useful in controlling the Siegel models $\Lambda_\Siegel, \mu_\Siegel$ in later sections.
	
We first recall a form of the fundamental lemma of sieve theory (arising from an analysis of the beta sieve).

\begin{lemma}[Fundamental lemma of sieve theory]\label{fund-lem}  Let $(a_n)_{n \in \Z}$ be a collection of non-negative reals, let $\kappa > 0$, $z \geq 2$, and $D \geq z^{9\kappa+1}$.  Let $g \colon \N \to [0,1)$ be a multiplicative function obeying the estimates
\begin{equation}\label{wpz}
 \prod_{w \leq p < z} (1 - g(p))^{-1} \leq K \left(\frac{\log z}{\log w}\right)^\kappa
\end{equation}
for all $2 \leq w \leq z$ and some $K > 0$.  Suppose that for every $d \leq D$ dividing $P(z)$ one has the formula
\begin{equation}\label{sumdn}
 \sum_{d|n} a_n = X g(d) + r_d
\end{equation}
for some $X>0$ and some remainder $r_d$.  Then one has
$$ \sum_{\substack{n\\ (n,P(z))=1}} a_n = X \left(\prod_{p < z} (1-g(p))\right) (1 + O(e^{9\kappa-s} K^{10})) + O\left( \sum_{\substack{d \leq D\\ d|P(z)}} |r_d|\right)$$
where $s \coloneqq \frac{\log D}{\log z}$.
\end{lemma}

\begin{proof}  See~\cite[Theorem 6.9]{friedlander-iwaniec}.
\end{proof}

In our applications, the ratio $s = \frac{\log D}{\log z}$ will grow at a logarithmic rate, leading to pseudopolynomial accuracy when applying the fundamental lemma.

Using the fundamental lemma we can obtain satisfactory estimates (with pseudopolynomial accuracy) for counting linear equations in the Cram\'er model (compare with Theorem~\ref{quant-thm}).

\begin{proposition}[Linear equations in the Cram\'er model]\label{lineq}  Let $t,m \geq 1$ be integers, and let $N \geq 100$.  Let $\Omega$ be a convex subset of the cube $[-N,N]^d$, and let $\psi_1,\dots,\psi_t \colon \Z^m \to \Z$ be linear forms
$$ \psi_i(\vec n) = \vec n \cdot \dot \psi_i + \psi_i(0) $$
for some $\dot \psi_i \in \Z^m$ and $\psi_i(0) \in \Z$.  Assume that the linear coefficients $\dot \psi_1,\dots,\dot \psi_t \in \Z^m$ are all pairwise linearly independent and have magnitude at most $\exp(\log^{3/5} N)$ (say).  Then for any $2 \leq z \leq Q$, one has
$$ \sum_{\vec n \in \Omega \cap \Z^m} \prod_{i=1}^t \Lambda_{\Cramer,z}(\psi_i(\vec n)) = \mathrm{vol}(\Omega) \prod_{p<z} \beta_p + O_{t,m}( N^m \exp(-c \log^{4/5} N))$$
for some $c>0$ depending only on $t,m$, where for each $p$, $\beta_p$ is the local factor
$$ \beta_p \coloneqq \E_{\vec n \in (\Z/p\Z)^m} \prod_{i=1}^t \frac{p}{p-1} 1_{\psi_i(\vec n) \neq 0}$$
where $\psi_i$ is also viewed as a map from $(\Z/p\Z)^m$ to $\Z/p\Z$ in the obvious fashion.
\end{proposition}

\begin{proof} Without loss of generality we may assume that $N$ is sufficiently large depending on $t,m$; we now allow all implied constants to depend on $t,m$.

For any $d$ dividing $P(z)$, let $g(d) \in [0,1]$ denote the quantity
$$ g(d) \coloneqq \E_{\vec n \in (\Z/d\Z)^m} 1_{\prod_{i=1}^t \psi_i(\vec n)=0},$$
with the convention that $g(d)=0$ if $d$ does not divide $P(z)$.  In particular we have 
\begin{equation}\label{g-form}
g(p) = 1 - \left(\frac{p-1}{p}\right)^t \beta_p
\end{equation}
for all $p<z$. From the Chinese remainder theorem we see that $g$ is multiplicative.  Suppose first that $g(p)=1$ for some $p<z$, then $\beta_p=0$ and $\prod_{i=1}^k \Lambda_{\Cramer,z}(\psi_i(n))$ is identically zero.  Thus the proposition is trivial in this case, so we may assume that $g(p)<1$ for all $p$.  From construction we then have the crude bound
\begin{equation}\label{gpm}
 g(p) \leq 1 - \frac{1}{p^m}.
\end{equation}
Also, from construction we see that for any two distinct linear forms $\psi_i, \psi_j$, there is a positive integer $A_{ij} = \exp(O(\log^{3/5} N))$ such that $\dot \psi_i, \dot \psi_j$ are linearly independent in $(\Z/p\Z)^k$ whenever $p$ does not divide $A_{ij}$ (indeed, one can take $A_{ij}$ to be one of the non-zero coefficients of the wedge product of $\dot \psi_i$ and $\dot \psi_j$).  If we let $A = \exp( O(\log^{3/5} N))$ be the product of all the $A_{ij}$, we conclude in particular that
$$ \E_{n \in (\Z/p\Z)^m} 1_{\psi_i(\vec n) = \psi_j(\vec n)=0} \leq \frac{1}{p^2}$$
whenever $p$ does not divide $A$, hence by the inclusion-exclusion formula (or Bonferroni inequalities) we have
\begin{equation}\label{bilo}
 g(p) = \frac{t}{p} + O\left( \frac{1}{p^2}\right)
\end{equation}
whenever $p$ does not divide $A$.  In particular we have
$$ (1-g(p))^{-1} = \left(\frac{p}{p-1}\right)^t \left(1 + O\left(\frac{1}{p^2}\right)\right)$$
unless $p$ divides $A$ (using \eqref{gpm} to handle the case when $p$ is bounded).  For $p$ dividing $A$, \eqref{gpm} instead gives $(1-g(p))^{-1} \leq p^m$.  We conclude that for any $2 \leq w \leq z$, we have
$$ \prod_{w \leq p < z} (1 - g(p))^{-1} \ll \left(\prod_{p|A} p\right)^m \prod_{w \leq p < z} \left(\frac{p}{p-1}\right)^t
\leq A^m \prod_{w \leq p < z} \left(\frac{p}{p-1}\right)^t $$
and hence by Mertens theorem the axiom \eqref{wpz} is obeyed with $\kappa = t$ and some $K = O( \exp( O(\log^{3/5} N) ) )$.

We introduce the sequence
$$ a_n \coloneqq \sum_{\vec n \in \Omega \cap \Z^m} 1_{\prod_{i=1}^t \psi_i(\vec n) = n}.$$
Observe that the $a_n$ are non-negative with
$$ \sum_{\vec n \in \Omega \cap \Z^m} \prod_{i=1}^t \Lambda_{\Cramer,z}(\psi_i(n)) = \left(\prod_{p < z} \frac{p}{p-1}\right)^t
\sum_{\substack{n\\ (n,P(z)) = 1}} a_n.$$
Set 
$$D \coloneqq \exp( \log^{9/10} N).$$
For any $d \leq D$ dividing $P(z)$, we have
$$ \sum_{n\equiv 0\pmod d} a_n = \sum_{\vec n \in \Omega \cap \Z^m} 1_{d|\prod_{i=1}^t \psi_i(\vec n)}.$$
The condition $d|\prod_{i=1}^t \psi_i(\vec n)$ restricts $d$ to $g(d) d^m$ cosets of $(d\Z)^m$.  Applying a volume packing argument using~\cite[Corollary A.2]{gt-linear} gives
$$ \sum_{\vec n \in \Omega \cap \Z^m} 1_{d|\prod_{i=1}^t \psi_i(\vec n)} = g(d) \mathrm{vol}(\Omega) + O( d^{O(1)} N^{m-1} ) $$
and hence axiom \eqref{sumdn} is obeyed with $X \coloneqq \mathrm{vol}(\Omega)$ and some $r_d = O( D^{O(1)} N^{m-1} )$.  Applying Lemma~\ref{fund-lem}, we conclude that
\begin{align*}
\sum_{\vec n \in \Omega \cap \Z^m} \prod_{i=1}^t \Lambda_{\Cramer,z}(\psi_i(\vec n)) &= \left(\prod_{p < z} \frac{p}{p-1}\right)^t \mathrm{vol}(\Omega) \prod_{p < z} (1 - g(p))\\
&\quad \times \left( 1 + O\left( e^{-s} \exp( O(\log^{3/5} N) ) \right) \right) \\
&\quad + O\left( \left(\prod_{p < z} \frac{p}{p-1}\right)^t D^{O(1)} N^{m-1} \right)
\end{align*}
with $s = \frac{\log D}{\log Q} \gg \log^{4/5} N$.  We can then simplify the right-hand side using \eqref{g-form} and Mertens' theorem to
$$ \sum_{\vec n \in \Omega \cap \Z^m} \prod_{i=1}^t \Lambda_{\Cramer,z}(\psi_i(\vec n)) = \mathrm{vol}(\Omega) \left(\prod_{p < z} \beta_p\right)
\left( 1 + O\left( \exp( - c \log^{4/5} N ) \right) \right) + O\left( N^{m-1/2} \right)$$
(say) for some constant $c>0$ depending on $t,m$.  From \eqref{wpz}, \eqref{g-form} and Mertens' theorem we have the crude bound
$$\prod_{p < z} \beta_p \ll \exp( O( \log^{3/5} N))$$
and the claim follows.
\end{proof}

As a first application of this estimate, we have good estimates (basically of logarithmic type) for the Cram\'er model in the Gowers norm.

\begin{corollary}[Gowers uniformity of the Cram\'er model on arithmetic progressions]\label{cram}  Let $2 \leq w \leq z \leq Q$ be such that $w \leq \log^{1/100} N$.  Set $W \coloneqq P(w)$.  Then for any $1 \leq b \leq W$ coprime to $W$, one has
$$ \left\| \frac{\phi(W)}{W} \Lambda_{\Cramer,z}(W\cdot+b) - 1 \right\|_{U^k[\frac{N-b}{W}]} \ll w^{-c}.$$
\end{corollary}

\begin{proof}  Write $N' \coloneqq \frac{N-b}{W}$.  We can rewrite the desired estimate (after adjusting $c$ appropriately) as
$$ \sum_{(n,\vec h) \in \Omega\cap \mathbb{Z}^{k+1}} \prod_{\omega \in \{0,1\}^k} \left(\frac{\phi(W)}{W} \Lambda_{\Cramer,z}(W(n+\omega \cdot \vec h)+b)-1\right)
\ll (N')^{k+1} w^{-c}$$
where $\Omega$ is the convex body of tuples $(n,\vec h) \in \R^{k+1}$ such that
$$ 0 < n + \omega \cdot \vec h \leq N'$$
for all $\omega \in \{0,1\}^k$.  By inclusion-exclusion, it suffices to establish the bounds
$$ \sum_{(n,\vec h) \in \Omega\cap \mathbb{Z}^{k+1}} \prod_{\omega \in S} \frac{\phi(W)}{W} \Lambda_{\Cramer,z}(W(n+\omega \cdot \vec h)+b)
= \mathrm{vol}( \Omega)  + O( (N')^{k+1} w^{-c} )
$$
for all subsets $S \subset \{0,1\}^k$.  Applying Proposition~\ref{lineq} (and Mertens' theorem), the left-hand side is equal to
$$ \left(\frac{\phi(W)}{W}\right)^{\# S} \mathrm{vol}(\Omega) \prod_{p < z } \beta_p + O( (N')^{k+1} w^{-c} )$$
(in fact there is plenty of room to spare in the error term), where
$$ \beta_p \coloneqq \E_{(n,\vec h) \in (\Z/p\Z)^{k+1}} \prod_{\omega \in S} \frac{p}{p-1} 1_{W(n+\omega \cdot \vec h)+b \neq 0}.$$
If $p < w$, then $W$ vanishes modulo $p$ and $b$ is coprime to $p$, and hence $\beta_p = (\frac{p}{p-1})^{\# S}$.  Thus we have
$$\left(\frac{\phi(W)}{W}\right)^{\# S} \prod_{p < z } \beta_p = \prod_{w \leq p < z} \beta_p.$$
By the inclusion-exclusion argument used to establish \eqref{bilo} one has
$$ \beta_p = \left(\frac{p}{p-1}\right)^{\# S} \left(1 - \frac{\# S}{p} + O\left(\frac{1}{p^2}\right)\right) = 1 + O\left(\frac{1}{p^2}\right)$$
for any $w \leq p < z$, hence
$$ \prod_{w \leq p < z} \beta_p = 1 + O( w^{-1} ).$$
Since $\mathrm{vol}(\Omega) \ll (N')^{k+1}$, the claim follows.
\end{proof}

\begin{proof}[Proof of Proposition~\ref{cramer-stable}]
Combining Corollary~\ref{cram} with Corollary~\ref{slice}, we see that
$$ \| \Lambda_{\Cramer,z} - \Lambda_{\Cramer,w} \|_{U^k[N]} \ll w^{-c}$$
whenever  $2 \leq w \leq z \leq \exp(\log^{1/10} N)$ are such that $w \leq \log^{1/100} N$.  Proposition~\ref{cramer-stable} now follows from the triangle inequality \eqref{triangle} (note the case $N=O(1)$ is trivial, so we may assume $N$ is large enough that $\log^{1/100} N > 2$).
\end{proof}

\begin{remark}\label{Effort} With more effort it may be possible to delete the $\log^{-c} N$ term in \eqref{qa}, but we will not need to do so here as there are several other error terms in our analysis that are of the same order of magnitude as $\log^{-c} N$, or worse.
\end{remark}

\subsection{Controlling the Siegel correction}\label{siegel-sec}

Now suppose that there is a $Q$-Siegel zero $\beta$, with associated quadratic character $\chi_\Siegel$ and conductor $q_\Siegel$.  In this subsection we combine the previous sieve-theoretic estimates with Weil sum estimates to obtain good control on the Siegel models $\Lambda_\Siegel, \mu_\Siegel$.  

We begin with some basic estimates on the $Q$-Siegel zero $\beta$ and the $Q$-Siegel conductor $q_\Siegel$. As $\chi_{\Siegel}$ is a primitive real character, $q_\Siegel$ is must either be square-free or four times a square-free number or eight times a square-free number.
From construction one has the upper bound
$$ q_\Siegel \leq Q = \exp(\log^{1/10} N).$$
From~\cite[Chapter 14, (12)]{davenport-book} one has the estimate
$$ 1-\beta \gg q_\Siegel^{-1/2} \log^{-2} q_\Siegel$$
which when combined with the upper bound $1 - \beta \ll \frac{1}{\log Q} \ll \log^{-1/10} N$ gives the lower bound
\begin{equation}\label{qsie}
 q_\Siegel \gg \frac{\log^{1/5} N}{(\log\log N)^2}.
\end{equation}
One could improve this lower bound using Siegel's theorem to strongly logarithmic, but we will not do so here in order to keep the estimates effective.  In particular, any bound of the shape $O(q_\Siegel^{-c})$ will lead to logarithmic decay.

From~\cite[Theorem 2.9]{mv} we observe the doubly logarithmic bound
\begin{equation}\label{sumdq}
\prod_{p|q_\Siegel} \left(1-\frac{1}{p}\right)^{-1} = \frac{q_\Siegel}{\phi(q_\Siegel)} \ll \log\log q_\Siegel \ll \log\log N.
\end{equation}

Next, we show that the quantity $\alpha$ in \eqref{alpha-def} is bounded, which was the missing step needed to establish Lemma~\ref{point-bound}:

\begin{lemma}\label{alpha-bound} We have $\alpha \ll 1$.  In particular, Lemma~\ref{point-bound} holds.
\end{lemma}

\begin{proof}  Consider the meromorphic function
$$ F(s) = \frac{1}{L(s,\chi_\Siegel)} \prod_{p < Q} \left(1 - \frac{\chi_\Siegel(p)}{p^s}\right)^{-1}.$$
This function has a simple pole at $\beta$ with residue
$$ \mathrm{Res}(F,\beta) = \frac{1}{L'(\beta,\chi_\Siegel)} \prod_{p < Q} \left(1 - \frac{\chi_\Siegel(p)}{p^\beta}\right)^{-1} = \alpha \prod_{p < Q} \left(1 - \frac{1}{p}\right)$$
and no other poles in the disk $\{ s: |s-\beta| \leq \frac{2c_0}{\log Q} \}$ if $c_0$ is small enough, by~\cite[Theorem 11.3]{mv}.  By Mertens' theorem, it thus suffices to establish the bound
$$ \mathrm{Res}(F,\beta)  \ll \frac{1}{\log Q}.$$
By the residue theorem, it suffices to show that
\begin{equation}\label{fs1}
 F(s) \ll 1
\end{equation}
on the circle $|s-\beta| = \frac{2c_0}{\log Q}$.  On the rightmost point $s_0 = \beta+\frac{2c_0}{\log Q} \geq 1 + \frac{c_0}{\log Q}$ of this circle, we can use the Euler product representation
$$ F(s_0) = \prod_{p \geq Q} \left(1 - \frac{\chi_\Siegel(p)}{p^{s_0}}\right)$$
followed by the triangle inequality to estimate
\begin{equation}\label{lem}
 |F(s_0)| \leq \prod_{p \geq Q} \left(1 + \frac{1}{p^{1+\frac{c_0}{\log Q}}}\right) \ll 1
\end{equation}
thanks to Mertens' theorem.  For more general points $s$ on this circle, we have from~\cite[Theorem 11.4]{mv} that
$$ \frac{L'}{L}(s,\chi_\Siegel) \ll \log Q.$$
Since 
\begin{align*}
 \frac{F'}{F}(s) &= -\frac{L'}{L}(s,\chi_\Siegel)  - \sum_{p<Q} \frac{\chi_\Siegel(p) \log p}{p^s - \chi_\Siegel(p)} \\
&= -\frac{L'}{L}(s,\chi_\Siegel) + O\left( \sum_{p<Q} \frac{\log p}{p^{1-\frac{3c_0}{\log Q}}} \right)
\end{align*}
(noting that $\mathrm{Re} s \geq 1 - \frac{3c_0}{\log Q}$), we conclude from Mertens' theorem that
$$ \frac{F'}{F}(s) \ll \log Q$$
on the entire circle; integrating this and using \eqref{lem}, we obtain \eqref{fs1} as required.
\end{proof}

From~\cite[Theorem 11.4]{mv} we have
$$ \frac{L'}{L}(s,\chi_\Siegel) = \frac{1}{s-\beta} + O(\log q_\Siegel)$$
and
$$ L(s,\chi_\Siegel) \gg |s-\beta|$$
for $s \neq \beta$ sufficiently close to $\beta$; multiplying the two estimates and taking limits as $s \to \beta$, we also obtain the bound
\begin{equation}\label{crude-lp}
 \frac{1}{L'(\beta,\chi_\Siegel)} \ll 1.
\end{equation}

We can view $\chi_\Siegel$ as a function on $\Z/q_\Siegel\Z$.  Crucially, it exhibits some cancellation in the Gowers norms (of polynomial type in $q_\Siegel$, and hence of logarithmic type in $N$):

\begin{lemma}[Gowers norm cancellation]\label{gnc}  For any $\eps>0$, we have
$$ \|\chi_\Siegel\|_{U^k(\Z/q_\Siegel\Z)} \ll_{\eps} q_\Siegel^{- \frac{1}{2^{k+1}}+\eps}.$$
\end{lemma}

\begin{proof} By the Chinese remainder theorem, we can express $\Z/q_\Siegel\Z$ as the product of prime cyclic groups $\Z/p\Z$ of odd order, as well as $\Z/2^j\Z$ for some $0\leq j \leq 3$.  The quadratic character $\chi_\Siegel$ can then be expressed as the tensor product of quadratic characters on these groups.  Using \eqref{tensor} and the divisor bound, it thus suffices to show that
$$ \| \chi \|_{U^k(\Z/p\Z)} \ll p^{- \frac{1}{2^{k+1}}}$$
for all odd primes $p$, with $\chi$ the quadratic character on $\Z/p\Z$.  By Definition~\ref{gowers-norm}, this is equivalent to
$$ \E_{(n,\vec h) \in (\Z/p\Z)^{k+1}} \prod_{\omega \in \{0,1\}^k} \chi(n+\omega \cdot \vec h)
\ll p^{-1/2}.$$
The contribution of any given tuple $\vec h \in (\Z/p\Z)^k$ to the left-hand side is trivially bounded by $O(p^{-k})$.  When the dot products $\omega \cdot \vec h$ are all distinct, the Weil bounds (see e.g.,~\cite[Corollary 11.24]{iwaniec-kowalski}) give instead the bound $O(p^{-k-1/2})$.  Since there are $p^k$ tuples $h$ and collisions between the $\omega \cdot \vec h$  only occur for $O(p^{k-1})$ of these tuples, the claim follows.
\end{proof}

We can now use this cancellation to prove Theorem~\ref{siegel-uniform}.

\begin{proof}[Proof of Theorem~\ref{siegel-uniform}]  We may assume $N$ is sufficiently large depending on $k$, and allow all implied constants to depend on $k$. Obviously we may assume that a $Q$-Siegel zero exists, as the claim is trivial otherwise. 

We first establish \eqref{lamsig}. It suffices to show the polynomial (in $q_\Siegel$) bound
\begin{align}\label{cramersiegel}
\| \Lambda_{\Cramer,Q} \chi_\Siegel (\cdot)^{\beta-1} \|_{U^k[N]} \ll q_\Siegel^{-c}
\end{align}
where $(\cdot)^{\beta-1}$ denotes the function $n \mapsto n^{\beta-1}$.
By the fundamental theorem of calculus, we have
\begin{equation}\label{int}
 1_{[N]}(t) t^{\beta-1} = \int_1^N 1_{[M]}(t) (1-\beta) M^{\beta-2}\ dM + N^{\beta-1} 1_{[N]}(t)
\end{equation}
and
\begin{equation}\label{int2}
1 = \int_1^N (1-\beta) M^{\beta-2}\ dM + N^{\beta-1}.
\end{equation}
Substituting \eqref{int} and \eqref{int2} on the left and right-hand sides of \eqref{cramersiegel}, respectively, and applying Minkowski's integral inequality to the Banach space norm $\|\cdot\|_{U^k[N]}$, it  suffices to show that\footnote{Alternatively, instead of applying Minkowski's integral inequality one could open the definition of the $U^k[N]$ norm, exchange the order of integration and averaging, and apply the Gowers--Cauchy--Schwarz inequality.}
$$
\| \Lambda_{\Cramer,Q} \chi_\Siegel 1_{[M]} \|_{U^k[N]} \ll q_\Siegel^{-c}.
$$
uniformly for all $1 \leq M \leq N$.
By Definition~\ref{gowers-norm}, we can rewrite this estimate as
\begin{equation}\label{als}
 \sum_{(n,\vec h) \in \Omega \cap \Z^{k+1}} \prod_{\omega \in \{0,1\}^k} \Lambda_{\Cramer,Q} \chi_\Siegel(n+\omega \cdot \vec h) \ll N^{k+1} q_\Siegel^{-c}
\end{equation}
for some $c>0$ and all $1 \leq M \leq N$, where $\Omega = \Omega_M$ is the convex body
$$ \Omega \coloneqq \{ (x,\vec y) \in \R^{k+1}: 0 < x+\omega \cdot \vec y \leq M \hbox{ for all } \omega \in \{0,1\}^k \}.$$
Splitting $n,h_1,\dots,h_k$ into cosets of $q_\Siegel$, we can write the left-hand side of \eqref{als} as
\begin{equation}\label{chos}
 \sum_{(a,\vec b) \in [q_\Siegel]^{k+1}} \prod_{\omega \in \{0,1\}^k} \chi_\Siegel(a + \omega \cdot \vec b) G(a,\vec b)
\end{equation}
where
$$ G(a,\vec b) \coloneqq \sum_{(n,\vec h) \in \frac{1}{q_\Siegel} (\Omega-(a,\vec b)) \cap \Z^{k+1}} \prod_{\omega \in \{0,1\}^k} \Lambda_{\Cramer,Q}(q_\Siegel n+q_\Siegel \omega \cdot \vec h + a + \omega \cdot \vec b).$$
Applying Proposition~\ref{lineq} (with $N$ replaced by $N/q_\Siegel$), we can estimate
$$ G(a,\vec b) = q_\Siegel^{-k-1} \mathrm{vol}(\Omega) \prod_{p<z} \beta_p + O( (N/q_\Siegel)^{k+1} \exp(-c \log^{4/5} N))$$
where
$$ \beta_p \coloneqq \E_{(n,\vec h) \in (\Z/p\Z)^{k+1}} \prod_{\omega \in \{0,1\}^k}  \frac{p}{p-1} 1_{q_\Siegel n+q_\Siegel \omega \cdot \vec h + a + \omega \cdot \vec b \neq 0}.$$
Because of the $\chi_\Siegel$ factor in \eqref{chos}, we can restrict attention to the case where $a + \omega \cdot \vec b$ is coprime to $q_\Siegel$.  This implies that $\beta_p = (\frac{p}{p-1})^{2^k}$ when $p|q_\Siegel$.  When $p \nmid q_\Siegel$, we can dilate $n,\vec h$ by $1/q_\Siegel$ (performing the division over the field $\Z/p\Z$) and then shift both variables to simplify
$$ \beta_p = \E_{(n,\vec h) \in (\Z/p\Z)^{k+1}} \prod_{\omega \in \{0,1\}^k}  \frac{p}{p-1} 1_{n+\omega \cdot \vec h \neq 0}.$$
In particular the $\beta_p$ are not dependent on $a,\vec b,q_{\textnormal{Siegel}}$.  Summing in $a,\vec b$, we can thus write the left-hand side of \eqref{als} as
$$ \mathrm{vol}(\Omega) \prod_{p<z} \beta_p \E_{(a,\vec b) \in [q_\Siegel]^{k+1}} \prod_{\omega \in \{0,1\}^k} \chi_\Siegel(a + \omega \cdot \vec b) + O( N^{k+1} \exp(-c \log^{4/5} N) ).$$
The error term is certainly negligible.  From Lemma~\ref{gnc} we have
$$ \E_{(a,\vec b) \in [q_\Siegel]^{k+1}} \prod_{\omega \in \{0,1\}^k} \chi_\Siegel(a + \omega \cdot \vec b) \ll q_\Siegel^{-1/4}$$
(say), and we can of course bound $\mathrm{vol}(\Omega) \ll N^{k+1}$.  Finally, direct calculation shows that $\beta_p = 1 + O(1/p^2)$ when $p \nmid q_\Siegel$, thus
$$ \prod_{p<z} \beta_p \ll \prod_{p|q_\Siegel} \left(\frac{p}{p-1}\right)^{2^k} \ll \left(\frac{q_\Siegel}{\phi(q_\Siegel)}\right)^{2^k} \ll (\log\log q_\Siegel)^{2^k}$$
thanks to \eqref{sumdq}.  Putting these estimates together, we obtain the claim \eqref{lamsig}.

Now we establish \eqref{musig}, which is a similar calculation but a little more involved because of the $\mu_\local$ factor.  
By Lemma~\ref{variant}, \eqref{qsie} it suffices to show that
$$
\sum_{d|q_\Siegel} \frac{1}{d} \| \mu_{\Siegel}(d \cdot) 1_{(\cdot,q_\Siegel/d)=1} \|_{U^k[N/d]}
\ll q_\Siegel^{-c}$$
for some $c>0$ depending on $k$.
From \eqref{sumdq} we have
$$
 \sum_{d|q_\Siegel} \frac{1}{d} \ll \log\log q_\Siegel
$$
so it suffices to show that
$$ \| \mu_{\Siegel}(d \cdot) 1_{(\cdot,q_\Siegel/d)=1}  \|_{U^k[N/d]} \ll q_\Siegel^{-c}$$
for each $d|q_\Siegel$.  

Fix $d$.  We rewrite this estimate as
\begin{equation}\label{moa-0}
 \| \mu_{\Siegel}(d \cdot) 1_{(\cdot,q_\Siegel/d)=1} 1_{[N/d]} \|_{\tilde U^k(\Z)}^{2^k} \ll q_\Siegel^{-c}
\| 1_{[N/d]} \|_{\tilde U^k(\Z)}^{2^k}.
\end{equation}
Using Definition~\ref{siegel-model}, we can write
\begin{align}\label{moa}\begin{split}
 &\mu_{\Siegel}(d n) 1_{(\cdot,q_\Siegel/d)=1}(n) 1_{[N/d]}(n) \\
 &=\alpha \sum_{d' \in {\mathcal D}} \mu(d) \mu(d') 1_{d'|n} (n/d')^{\beta-1} \chi_\Siegel(n/d') 1_{(n/d',P(Q))=1} 1_{[N/d]}(n)
 \end{split}
\end{align}
where ${\mathcal D}$ is the set of all $d' | P(Q)$ with $(d',q_\Siegel)=1$.  By Lemma~\ref{alpha-bound}, it thus suffices to show that
$$ \left\|\sum_{d' \in {\mathcal D}} \mu(d') 1_{d'|\cdot} (\cdot/d')^{\beta-1} \chi_\Siegel(\cdot/d') 1_{(\cdot/d',P(Q))=1} 1_{[N/d]} 
\right\|_{\tilde U^k(\Z)}^{2^k} \ll q_\Siegel^{-c}
\| 1_{[N/d]} \|_{\tilde U^k(\Z)}^{2^k}.$$
Using \eqref{int}, \eqref{int2} and Minkowski's integral inequality, it suffices to show
\begin{equation}\label{moa-2}
 \left\|\sum_{d' \in {\mathcal D}} \mu(d') 1_{d'|\cdot} \chi_{\Siegel,M}(\cdot/d')  1_{[N/d]}
\right\|_{\tilde U^k(\Z)}^{2^k} \ll q_\Siegel^{-c}
\| 1_{[N/d]} \|_{\tilde U^k(\Z)}^{2^k}
\end{equation}
for any $M \geq 1$, where
$$ \chi_{\Siegel,M}(n) \coloneqq \chi_\Siegel(n) 1_{[M]}(n) 1_{(n,P(Q))=1}.$$

We decompose $\mathcal{D}={\mathcal D}_{\leq} \cup {\mathcal D}_{>}$, where ${\mathcal D}_{\leq}$ are those $d' \in {\mathcal D}$ with $d' \leq \exp(\log^{1/2} N)$ (say) and ${\mathcal D}_>$ are those $d' \in {\mathcal D}$ with $d' > \exp(\log^{1/2} N)$.  
We first dispose of the contribution of the large $d'$, i.e. those that satisfy $d' \in {\mathcal D}_>$.  Their contribution to the expression inside the norm on the left-hand side of \eqref{moa-2} is supported on a set of numbers $n$ of size
$$ \sum_{d' \in {\mathcal D}_>} \frac{N}{dd'}.$$
From basic estimates on smooth numbers~\cite[Theorem 1.1]{hildebrand-tenenbaum}, the number of elements of ${\mathcal D}_>$ in any dyadic range $[M,2M]$ with $M \in [Q,N]$ is $O( M u^{-u/2})$ (say) where $u \coloneqq \frac{\log M}{\log Q}$.  From this and a routine dyadic decomposition we see that
$$ \sum_{d' \in {\mathcal D}_>} \frac{N}{dd'} \ll N \exp(-\log^{-1/10} N)$$
(say).  We thus see that the contribution to the left-hand side of \eqref{moa-2} can be bounded by $O( N^{k+1} \exp(-\log^{-1/10} N))$, which is acceptable.  Thus, by the triangle inequality \eqref{triangle}, it suffices to control the contribution of ${\mathcal D}_{\leq}$, i.e. to show that
$$
 \left\|\sum_{d' \in {\mathcal D}_{\leq}} \mu(d') 1_{d'|\cdot} \chi_{\Siegel,M} 1_{[N/d]} 
\right\|_{\tilde U^k(\Z)}^{2^k} \ll q_\Siegel^{-c}
\| 1_{[N/d]} \|_{\tilde U^k(\Z)}^{2^k}
$$
We can expand out the left-hand side as
$$\sum_{d' \in {\mathcal D}_{\leq}^{\{0,1\}^k}} A_{d'}$$
where for $d' = (d'_\omega)_{\omega \in \{0,1\}^k}$ we have
$$ A_{d'} \coloneqq \sum_{(n,\vec h) \in \Omega} \prod_{\omega \in \{0,1\}^k} \mu(d'_\omega) 1_{d'_\omega|n+\omega \cdot \vec h} \chi_{\Siegel,M}\left(\frac{n+\omega \cdot \vec h}{d'_\omega}\right) $$
where $\Omega$ is the set of all tuples $(n,\vec h) \in \Z^{k+1}$ such that $n+\omega \cdot \vec h \in [N/d]$ for all $\omega \in \{0,1\}^k$.
Meanwhile, using the pointwise bound
$$ 0 \leq \sum_{d' \in {\mathcal D}_{\leq}}  1_{d'|\cdot} 1_{[N/d]} 1_{(\cdot/d',P(Q)) = 1} \leq 1_{[N/d]}$$
(reflecting the fact that every number $n$ has a unique decomposition $n = d' (n/d')$ where $d'|P(Q)$ and $(n/d',P(Q))=1$)
one has
$$\sum_{d' \in {\mathcal D}_{\leq}^{\{0,1\}^k}} B_{d'} \leq \| 1_{[N/d]} \|_{\tilde U^k(\Z)}^{2^k}
$$
where
$$ B_{d'} \coloneqq \sum_{(n,\vec h) \in \Omega} \prod_{\omega \in \{0,1\}^k} 1_{d'_\omega|n+\omega \cdot \vec h} 1_{(\frac{n+\omega \cdot \vec h}{d'_\omega},P(Q)) = 1}.$$
Hence it will suffice to show that
$$ A_{d'} \ll q_\Siegel^{-c} B_{d'}$$
for all $d' \in {\mathcal D}_{\leq}^{\{0,1\}^k}$.

The constraints $1_{d'_\omega|n+\omega \cdot \vec h}$ restrict $(n,\vec h)$ to some finite union of cosets $(a,\vec b)+D\Z^{k+1}$ of $D \Z^{k+1}$ where $D \coloneqq \prod_{\omega \in \{0,1\}^k} d'_\omega$, with the property that $d'_\omega$ divides $a+\omega \cdot \vec b$ for all $\omega \in \{0,1\}^k$.  Note from construction that $D$ is coprime to $q_\Siegel$ and of size $O( \exp( O( \log^{1/2} N)))$.  So, denoting for brevity $\Omega^{(a,\vec b)}:=\Omega\cap ((a,\vec b)+D\mathbb{Z}^{k+1})$, it will suffice to show that
\begin{equation}\label{ban}
 \sum_{(n,\vec h) \in \Omega^{(a,\vec b)}} \prod_{\omega \in \{0,1\}^k} \chi_{\Siegel,M}\left(\frac{n+\omega \cdot \vec h}{d'_\omega}\right) \ll q_\Siegel^{-c}
\sum_{(n,\vec h) \in \Omega^{(a,\vec b)}} \prod_{\omega \in \{0,1\}^k} 1_{(\frac{n+\omega \cdot \vec h}{d'_\omega},P(Q)) = 1}
\end{equation}
for all such cosets $(a,\vec b) + D\Z^{k+1}$.   Using Proposition~\ref{lineq} and some elementary rescaling, we have
$$
\sum_{(n,\vec h) \in \Omega^{(a,\vec b)}} \prod_{\omega \in \{0,1\}^k} 1_{(\frac{n+\omega \cdot \vec h}{d'_\omega},P(Q)) = 1}
= D^{-k-1} \mathrm{vol}(\Omega) \prod_{p<Q} \tilde \beta_p + O((N/D)^{k+1} \exp(-c \log^{4/5} N))$$
where
$$ \tilde \beta_p \coloneqq \E_{(n,\vec h) \in (\Z/p\Z)^{k+1}} \prod_{\omega \in \{0,1\}^k} 1_{\frac{a+\omega \cdot \vec b}{d'_\omega} + (n + \omega \cdot \vec h) \frac{D}{d'_\omega} \neq 0}.$$
If any of the $\tilde{\beta}_p$ vanish then both sides of \eqref{ban} vanish and we are done.  For $p$ not dividing $D$ we have the crude bound
\begin{equation}\label{crude-1}
 \tilde \beta_p = 1 - O(1/p)
\end{equation}
and for all $p$ we have the lower bound
\begin{equation}\label{crude-2}
\tilde \beta_p \geq \frac{1}{p^{k+1}}
\end{equation}
since the $\tilde \beta_p$ are non-vanishing integer multiples of $1/p^{k+1}$.  This gives the crude lower bound
\begin{equation}\label{crude-lower}
 \prod_{p<Q} \tilde \beta_p \gg D^{-O(1)} \log^{-O(1)} N
\end{equation}
and hence the right-hand side of \eqref{ban} is comparable to $q_\Siegel^{-c} (N/D)^{k+1} \prod_{p<Q} \tilde \beta_p$.  Next, we partition the left-hand side of \eqref{ban} as
\begin{equation}\label{qq}
 \sum_{(r,\vec s) \in [q_\Siegel]^{k+1}} \prod_{\omega \in \{0,1\}^k} \chi_\Siegel(d'_\omega) \chi_\Siegel(r + \omega \cdot \vec s) F_{r,\vec s}
\end{equation}
where
$$ F_{r,\vec s} \coloneqq \sum_{(n,\vec h) \in \Omega \cap ((a,\vec b) + D \Z^{k+1}) \cap ((r,\vec s) + q_\Siegel \Z^{k+1})} 
\prod_{\omega \in \{0,1\}^k} 1_{[M]}\left(\frac{n + \omega \cdot \vec h}{d'_\omega}\right) 1_{(n+\omega \cdot \vec h,P(Q))=1}.$$
We can restrict attention to those $(r,\vec s)$ for which $r + \omega \cdot \vec s$ is coprime to $q_\Siegel$ for all $\omega \in \{0,1\}^k$, since otherwise the product in \eqref{qq} vanishes.  Under this assumption, we can apply Proposition~\ref{lineq}, the Chinese remainder theorem, and some further rescaling (using the fact that $D, q_\Siegel$ are coprime), to conclude that
$$ 
F_{r,\vec s} = (D q_\Siegel)^{-k-1} \mathrm{vol}(\Omega') \prod_{\substack{p<Q\\ p \nmid q_\Siegel}} \tilde \beta_{p} + O\left( (N/Dq_\Siegel)^{k+1} \exp(-c \log^{4/5} N) \right)$$
where
$$ \Omega' \coloneqq \left\{ (n,\vec h) \in \Omega: \frac{n + \omega \cdot \vec h}{d'_\omega} \in [M] \,\, \forall \omega \in \{0,1\}^k \right\}.$$
Note the main term here is independent of $r,\vec s$. In particular, we can rewrite \eqref{qq} as
\begin{align*}
&\left(\prod_{\omega \in \{0,1\}^k} \chi_\Siegel(d'_\omega)\right) \| \chi_\Siegel \|_{\tilde{U}^k(\Z/q_\Siegel)}^{2^k} (Dq_\Siegel)^{-k-1}
\mathrm{vol}(\Omega') \prod_{\substack{p<Q\\ p \nmid q_\Siegel}} \tilde \beta_{p}\\
&+ O( (N/D)^{k+1} \exp(-c \log^{4/5} N) ).
\end{align*}
Applying Lemma~\ref{gnc}, this quantity is
\begin{equation}\label{al}
\ll q_\Siegel^{-c} (N/D)^{k+1} \prod_{\substack{p<Q\\ p \nmid q_\Siegel}} \tilde \beta_{p} + (N/D)^{k+1} \exp(-c \log^{4/5} N).
\end{equation}
The second term in \eqref{al} is acceptable thanks to \eqref{crude-lower}.  From \eqref{crude-1}, \eqref{crude-2}, \eqref{sumdq} we have
$$\prod_{p | q_\Siegel} \tilde \beta_{p}  \gg (\log\log q_\Siegel)^{-O(1)}$$
and so the first term in \eqref{al} is also acceptable.
\end{proof}

\section{The Manners inverse theorem}\label{sec:manners}

We are now ready to state a version of the inverse theorem of Manners~\cite{manners}, though formulated in a slightly different language (in particular, using the complexity notions from~\cite{gt-leibman} rather than~\cite{manners}).  

\begin{definition}[Nilmanifolds]\label{filnil}  Let $s \geq 1$ be an integer, and let $M>0$.  A (filtered) nilmanifold $G/\Gamma$ of degree $s$ and complexity at most $M$ consists of the following data:
\begin{itemize}
\item[(i)]  A nilpotent connected and simply connected Lie group $G$ of some dimension $m$, which can be identified with its Lie algebra $\log G$ via the exponential map $\exp \colon \log G \to G$ or its inverse $\log \colon G \to \log G$;
\item[(ii)]  A filtration $G_\bullet = (G_i)_{i \geq 0}$ of closed connected subgroups $G_i$ of $G$ with 
$$ G = G_0 = G_1 \geq G_2 \geq \dots \geq G_s \geq G_{s+1} = \{ \mathrm{id}_G\}$$
(and $G_i$ trivial for all $i \geq s+1$), such that\footnote{We use $[,]$ to denote both the commutator in the Lie group $G$ and the Lie bracket in the Lie algebra $\log G$, with the two being related to each other by the Baker--Campbell--Hausdorff formula.} $[G_i,G_j] \subset G_{i+j}$ for all $i,j \geq 0$ (or equivalently, $[\log G_i, \log G_j] \subset \log G_{i+j}$ in the Lie algebra $\log G$);
\item[(iii)]  A discrete cocompact subgroup $\Gamma$ of $G$;
\item[(iv)] A linear basis $X_1,\dots,X_{\mathrm{dim} G}$ of $\log G$, known as a \emph{Mal'cev basis (of the second kind)}.
\end{itemize}
We require this data to obey the following axioms:
\begin{itemize}
\item[(a)] For $1 \leq i, j \leq \mathrm{dim}(G)$, one has
\begin{equation}\label{xio}
 [X_i,X_j] = \sum_{i,j < k \leq \mathrm{dim}(G)} c_{ijk} X_k
\end{equation}
for some rational numbers $c_{ijk}$ with numerator and denominator bounded in magnitude by $M$.
\item[(b)]  For each $1 \leq i \leq s$, the Lie algebra $\log G_i$ is spanned by the $X_j$ with $\mathrm{dim}(G) - \mathrm{dim}(G_i) < j \leq \mathrm{dim}(G)$.
\item[(c)]  The subgroup $\Gamma$ consists of all elements of the form $\exp(t_1 X_1) \cdots \exp(t_{\mathrm{dim} G} X_{\mathrm{dim} G})$ with $t_1,\dots,t_{\mathrm{dim} G} \in \Z$.
\end{itemize}
This data defines a metric on $G/\Gamma$ as described in~\cite[Definition 2.2]{gt-leibman}, as well as the notion of a polynomial map $g \colon \Z \to G$, defined in~\cite[Definition 1.8]{gt-leibman}.
\end{definition}

A function $f: X \to \C$ is said to be \emph{$1$-bounded} if $|f(n)| \leq 1$ for all $n \in X$.  

\begin{theorem}[Manners inverse theorem]\label{manners-inv}  Let $0 < \delta < 1$.  Let $f: [N] \to \C$ be a $1$-bounded function such that
$$ \|f\|_{U^k[N]} \geq \delta.$$
Then there exist a (filtered) nilmanifold $G/\Gamma$ of degree $k-1$, dimension $O(\delta^{-O(1)})$, and complexity at most $\exp\exp(O(1/\delta^{O(1)}))$, a $1$-bounded Lipschitz function $F \colon G/\Gamma \to \C$ of Lipschitz constant at most $\exp\exp(O(1/\delta^{O(1)}))$, and a polynomial map $g \colon \Z \to G$, such that
$$ |\E_{n \in [N]} f(n) \overline{F}(g(n) \Gamma)| \gg \exp(-\exp(O(1/\delta^{O(1)}))).$$
\end{theorem}

\begin{proof}  By Bertrand's postulate we can find a prime $N'$ such that $10 N \leq N' \leq 20N'$.  If we embed $[N]$ into the cyclic group $\Z/N'\Z$ and extend $f$ by zero we may view $f$ as a $1$-bounded function on $\Z/N'\Z$, and a brief calculation reveals that
$$ \|f\|_{U^k(\Z/N'\Z)} \gg \delta.$$
We now apply~\cite[Theorem 1.1.2]{manners} with $s \coloneqq k-1$ to produce the required data $G/\Gamma$, $g$, $F$, $X_i$, save for two differences.  Firstly, the polynomial $g$ is described as a map from $\Z/N'\Z$ to $G/\Gamma$ rather than from $\Z$ to $G$, but one can lift the map from the former to the latter using~\cite[Proposition C.17]{manners}.  Secondly, instead of axiom (a) of Definition~\ref{filnil}, the basis elements $X_i$ are instead required to obey a decomposition
\begin{equation}\label{expi}
 [\exp(X_i), \exp(X_j)] = \prod_{i,j < l \leq \mathrm{dim}(G)} \exp( a_{ijl} X_l )
\end{equation}
for some integers $a_{ijl}$ bounded  in magnitude by some bound $M_0 \ll \exp\exp(O(1/\delta^{O(1)}))$, where the product is taken from left to right.  	However, as briefly noted in~\cite[\S C.2]{manners}, one can pass from this control \eqref{expi} to the control \eqref{xio} (with $M$ a suitable polynomial of $M_0$), as follows.  For any $1 \leq a \leq k-1$, we let $P(a)$ denote the claim that one has \eqref{xio} with $M$ of the form $\exp\exp(O(1/\delta^{O(1)}))$ whenever one of $X_i,X_j$ lies in $\log G_a$.  The claim $P(a)$ is certainly true for $a=k-1$ since $\log G_{k-1}$ is central, and we will be done if $P(1)$ is true, so it suffices by downward induction (with at most $k-2$ steps) to show that $P(a+1)$ implies $P(a)$ for any $1 \leq a \leq k-2$, where the implied constants in the $O_k()$ notation are allowed to vary with each step of the induction.  Call a rational number \emph{good} if its numerator and denominator are bounded in magnitude by $\exp\exp(O(1/\delta^{O(1)}))$.  If one of $X_i, X_j$ lie in $\log G_a$, then from \eqref{expi}, the induction hypothesis, and the Baker--Campbell--Hausdorff formula we see that
\begin{equation}\label{lij}
 \log[\exp(X_i), \exp(X_j)] = \sum_{l > i,j} c'_{ijl} X_l
\end{equation}
for some good rationals $c'_{ijl}$ (and furthermore one can restrict to those $X_k$ lying in $\log G_{a+1}$).  On the other hand, a further application of Baker--Campbell--Hausdorff reveals that $\log[\exp(X_i), \exp(X_j)]$ is equal to $[X_i,X_j]$ plus $O_k(1)$ additional terms, which consist of a good rational number times an iterated Lie bracket formed by starting with $[X_i,X_j]$ and taking the Lie bracket with either $X_i$ or $X_j$ one or more times (but no more than $O(1)$ times in all).  Inverting this formula, we can then write $[X_i,X_j]$ as $\log[\exp(X_i),\exp(X_j)]$ plus $O(1)$ additional terms, which consist of a good rational number times an iterated Lie bracket formed by starting with $\log[\exp(X_i), \exp(X_j)]$ and taking the Lie bracket with either $X_i$ or $X_j$ one or more times (but no more than $O(1)$ times in all).  Using \eqref{lij} and the induction hypothesis $P(a+1)$ repeatedly, we conclude $P(a)$, thus closing the induction.
\end{proof}

\begin{remark} As noted in~\cite{manners}, improved bounds are available for $k \leq 4$~\cite{gt-U3,gt-U4}, but we will not be able to take advantage of these bounds due to inefficiencies elsewhere in the arguments (in particular, our nilsequence equidistribution theory involves exponents that are exponential in the dimension rather than polynomial).
\end{remark}

From Lemma~\ref{point-bound} we see that the function $\mu - \mu_\Siegel$ can be made $1$-bounded by multiplying 
by a small absolute constant.  Applying Theorem~\ref{manners-inv} in the contrapositive (setting $\delta$ equal to a small power of $(\log\log N)^{-1}$, we conclude that the bound \eqref{musi} is an immediate consequence of \eqref{mmus}.  The same argument does not work directly for $\Lambda - \Lambda_\Siegel$ due to the additional factor of $\log N$ in the pointwise bounds; but we will be able to get around this in Section~\ref{sec: densification} by employing the densification technology of Conlon, Fox, and Zhao~\cite{cfz-rel}.  Assuming this for the moment, the only remaining step needed to establish Theorem~\ref{main} is to prove Theorem~\ref{quant-ortho}, to which we now turn.

\begin{remark}\label{improv-2} When $k=3$, one can appeal instead of Theorem~\ref{manners-inv} to the quantitative inverse theorem in~\cite{gt-U3}, and when $k=4$ one can use the fact that Manners proved in~\cite{manners} a stronger form of Theorem~\ref{manners-inv} for $k=4$ than for $k\geq 5$.  If one does so, one eventually finds that one would be able to improve the doubly logarithmic bounds in Theorem~\ref{main} for $k\leq 4$ to singly logarithmic, provided that one could increase the bound on the dimension of $G/\Gamma$ in Theorem~\ref{quant-ortho} from $(\log\log N)^{c_1}$ to $\log^{c_1} N$.  Unfortunately, our equidistribution theory on nilmanifolds is currently not satisfactory at this high a dimension, although in principle it is conceivable that some variant of the methods of Schmidt~\cite{schmidt} could resolve this issue.  We will not pursue this question further here.
\end{remark}

\section{Orthogonality to nilsequences}\label{sec:mobius}

In this section we prove Theorem~\ref{quant-ortho}.  
We begin by establishing Proposition~\ref{equiprop}, which will be used to establish the ``major arc'' case of Theorem~\ref{quant-ortho}.

\begin{proof} (Proof of Proposition~\ref{equiprop})  
We adopt the convention that any factor involving the $Q$-Siegel character $\chi_\Siegel$ is deleted if no such character exists.
Any arithmetic progression $P \subset [N]$ can be expressed in the form $\{ N'' < n \leq N': n = a\ (q) \}$ for some $1 \leq a \leq q$ and $0 < N'' \leq N' \leq N$.  By the triangle inequality, it thus suffices to establish the bounds
\begin{equation}\label{lam-equi}
\sum_{\substack{n \leq N'\\ n = a\ (q)}} \Lambda(n) = \sum_{\substack{n \leq N'\\ n = a\ (q)}} \Lambda_{\Siegel}(n)+O(N\exp(-c\log^{1/10}N))
\end{equation}
and
\begin{equation}\label{mu-equi}
\sum_{\substack{n \leq N'\\ n = a\ (q)}} \mu(n) = \sum_{\substack{n \leq N'\\ n = a\ (q)}} \mu_{\Siegel}(n)+O(N\exp(-c\log^{1/10}N))
\end{equation}
for any $1 \leq a \leq q$ and $0 < N' \leq N$.

If $q > \exp(c_2 \log^{1/10} N)$ for any constant $c_2>0$ then the triangle inequality (and Lemma~\ref{point-bound}) give the desired bounds after adjusting the value of $c$, so we may assume that $q \leq \exp(c_2 \log^{1/10} N)$ for some small absolute constant $c_2$.  In particular $q \leq Q$.  Similarly we may assume $N' \geq N \exp(-c_2\log^{1/10} N)$.

We begin with \eqref{lam-equi}.  From~\cite[Theorem 5.27]{iwaniec-kowalski} one has
$$
 \sum_{\substack{n \leq N'\\ n = a\ (q)}} \Lambda(n) = \frac{N'}{\phi(q)} \left(1 - \chi_\Siegel(a) 1_{q_\Siegel|q} \frac{(N')^{\beta-1}}{\beta}\right) 1_{(a,q)=1}+O(N\exp(-c\log^{1/10}N)).$$
Therefore, it will certainly suffice from the triangle inequality to show for $1\leq a\leq q\leq \exp(\log^{3/5}N)$ that\footnote{It would of course suffice to show this for $q\leq \exp(\log^{1/10}N)$ and with savings $\exp(-c\log^{1/10}N)$, but the larger powers of $\log N$ will be useful later on.}
\begin{equation}\label{lamavg}
 \sum_{\substack{n \leq N'\\ n = a\ (q)}} \Lambda_{\Cramer,Q}(n) = \frac{N'}{\phi(q)} 1_{(a,q)=1}+O(N\exp(-c\log^{4/5}N))
\end{equation}
and
\begin{equation}\label{lamsiavg}
 \sum_{\substack{n \leq N'\\ n = a\ (q)}} (\Lambda_{\Cramer,Q}(n) - \Lambda_{\Siegel}(n)) = \frac{(N')^\beta}{\beta \phi(q)} \chi_\Siegel(a) 1_{q_\Siegel|q} 1_{(a,q)=1}+O(N\exp(-c\log^{4/5}N)).  
\end{equation}
We first show \eqref{lamavg}.  By a change of variables we have
$$ \sum_{\substack{n \leq N'\\ n = a\ (q)}} \Lambda_{\Cramer,Q}(n) = \sum_{\frac{-a}{q} \leq n \leq \frac{N'-a}{q}} \Lambda_{\Cramer,Q}(qn+a)$$
and then on applying Proposition~\ref{lineq} we have
$$ \sum_{\substack{n \leq N'\\ n = a\ (q)}} \Lambda_{\Cramer,Q}(n) = \frac{N'}{q} \prod_{p<Q} \beta_p+O(N \exp(-c \log^{4/5} N))$$
where
$$ \beta_p \coloneqq \E_{n \in \Z/p\Z} \frac{p}{p-1} 1_{qn+a \neq 0}.$$
If $(a,q)>1$ then $(a,q)$ will be divisible by some prime $p \leq q <Q$, in which case $\beta_p=0$ and the claim follows.  If instead $(a,q)=1$,
then $\beta_p=1$ for all $p<Q$ not dividing $q$, and $\beta_p = \frac{p}{p-1}$ for all $p<Q$ dividing $q$, and the claim \eqref{lamavg} follows.

Now we show \eqref{lamsiavg}.  We may of course assume there is a $Q$-Siegel zero, in which case (by Definition~\ref{siegel-model}(ii)) our task is to show that
$$
\sum_{\substack{n \leq N'\\ n = a\ (q)}} \Lambda_{\Cramer,Q}(n) n^{\beta-1} \chi_\Siegel(n) = \frac{(N')^\beta}{\beta \phi(q)} \chi_\Siegel(a) 1_{q_\Siegel|q} 1_{(a,q)=1}+O(N\exp(-c\log^{4/5}N)).$$
From the fundamental theorem of calculus we have
$$ n^{\beta-1} 1_{[N']}(n) = \int_1^{N'} (1-\beta) M^{\beta-2} 1_{[M]}(n)\ dM+ (N')^{\beta-1} 1_{[N']}(n)$$
and
$$ \frac{(N')^\beta}{\beta} - \frac{1}{\beta}+1 = \int_1^{N'} (1-\beta) M^{\beta-2} M\ dM + (N')^{\beta-1} N'$$
so from the triangle inequality it suffices to show that
$$
\sum_{\substack{n \leq M\\ n = a\ (q)}} \Lambda_{\Cramer,Q}(n) \chi_\Siegel(n) = \frac{M}{\phi(q)} \chi_\Siegel(a) 1_{q_\Siegel|q} 1_{(a,q)=1}+O(N\exp(-c\log^{4/5}N))$$
for all $1 \leq M \leq N$.  We split the left-hand side as
$$ \sum_{\substack{1 \leq b \leq q'\\ b = a\ (q)}} \chi_\Siegel(b) \sum_{\substack{n \leq M\\ n = b\ (q')}} \Lambda_{\Cramer,Q}(n)$$
where $q' \coloneqq [q,q_\Siegel]$ is the least common multiple of $q$ and $q_\Siegel$.  By  \eqref{lamavg} we have
$$ \sum_{\substack{n \leq M\\ n = b\ (q')}} \Lambda_{\Cramer,Q}(n) = \frac{M}{\phi(q')} 1_{(b,q')=1} + O( N \exp(-c \log^{4/5} N))$$
and thus
$$
\sum_{\substack{n \leq M\\ n = a\ (q)}} \Lambda_{\Cramer}(n) \chi_\Siegel(n) = \frac{M}{\phi(q')} 
\sum_{\substack{1 \leq b \leq q'\\ b = a\ (q)}} \chi_\Siegel(b) 1_{(b,q')=1}+O(N\exp(-c\log^{4/5}N)).$$
The right-hand side vanishes if $(a,q)>1$, and also vanishes if $q' > q$ due to the orthogonality properties of Dirichlet characters.  If instead $(a,q)=1$ and $q'=q$ then the right-hand side is equal to $\frac{M}{\phi(q)} \chi_\Siegel(a)$, and the claim \eqref{lamsiavg} follows.  

Now we turn to \eqref{mu-equi}.  We first do an easy reduction to the case of primitive residue classes.  Let $d \coloneqq (a,q)$.  Observe that for any natural number $n$ one has
$$ \mu(dn) = \mu(d) \mu(n) 1_{(n,d)=1}$$
and also from Definition~\ref{siegel-model}(ii) we similarly have
$$ \mu_\Siegel(dn) = \mu(d) \mu_\Siegel(n) 1_{(n,d)=1}$$
and thus
\begin{equation}\label{almo}
\begin{split}
\sum_{\substack{n \leq N'\\ n = a\ (q)}} (\mu(n) - \mu_{\Siegel(n)}) &= \mu(d) \sum_{\substack{n \leq N'/d\\ n = a/d\ (q/d)\\ (n,d)=1}}(\mu(n) - \mu_{\Siegel}(n))\\
&=\mu(d) \sum_{\substack{1 \leq b \leq d\\ (b,d)=1}} \sum_{\substack{n \leq N'/d\\ n = a/d\ (q/d)\\ n=b\ (d)}}(\mu(n) - \mu_{\Siegel}(n)).
\end{split}
\end{equation}
Since $d \leq q \leq \exp(c_2\log^{1/10} N)$, it thus suffices to establish the pseudopolynomial decay estimate
$$ \sum_{\substack{n \leq N'/d\\ n = a/d\ (q/d)\\ n=b\ (d)}} (\mu(n) - \mu_{\Siegel}(n)) \ll N \exp( - c \log^{1/10} N )$$
for all $1 \leq b \leq d$ coprime to $d$ (where the constant $c$ here is uniform in $c_2$).  Writing $q' \coloneqq [q/d,d]$, we see from the Chinese remainder theorem that the constraints
$n = a/d\ (q/d); n=b\ (d)$ are either inconsistent, or constrain $n$ to precisely one primitive residue class $a'\ (q')$ with $(a',q')=1$.  Thus it suffices to show the pseudopolynomial decay bound
$$ \sum_{\substack{n \leq N'\\ n = a'\ (q')}}(\mu(n) - \mu_{\Siegel(n)}) \ll N \exp( - c \log^{1/10} N )$$
whenever $1 \leq N' \leq N$ and $1 \leq a' \leq q' \leq \exp(2c_2 \log^{1/10} N)$ with $(a',q')=1$.

When there is no $Q$-Siegel zero the claim is immediate from~\cite[Exercise 11.3.12]{mv} (modified slightly due to our slightly different definition of a Siegel zero).  Now suppose that there is a $Q$-Siegel zero.  The result previously cited in~\cite[Exercise 11.3.12]{mv} (again modified slightly to account for our slightly different notion of Siegel zero) then gives the pseudopolynomially accurate asymptotic
$$ \sum_{\substack{n \leq N'\\ n = a'\ (q')}} \mu(n) = 1_{q_\Siegel|q'} \frac{\chi_{q'}(a') (N')^\beta}{\phi(q') L'(\beta,\chi_{q'}) \beta}  + O( N \exp(-c \log^{1/10} N) )$$
where $\chi_{q'}(n) \coloneqq \chi_\Siegel(n) 1_{(n,q')=1}$ is the character of modulus $q'$ induced from $\chi_\Siegel$ when $q'$ is a multiple of $q_\Siegel$.  Note that
$$ L(s,\chi_{q'}) = L(s,\chi_\Siegel) \prod_{\substack{p|q'\\ p \nmid q_\Siegel}} \left(1 - \frac{\chi_\Siegel(p)}{p^s}\right),$$
and thus by the product rule (and the fact that $L(\beta,\chi_{Siegel})=0$)
$$ L'(\beta,\chi_q) = L'(\beta,\chi_\Siegel) \prod_{\substack{p|q'\\ p \nmid q_\Siegel}} \left(1 - \frac{\chi_\Siegel(p)}{p^\beta}\right).$$
We conclude that
\begin{align*}
 \sum_{\substack{n \leq N'\\ n = a'\ (q')}} \mu(n) &= 1_{q_\Siegel|q'} \frac{(N')^\beta \chi_\Siegel(a')}{\beta \phi(q) L'(\beta,\chi_\Siegel)} \prod_{\substack{p|q'\\ p \nmid q_\Siegel}} \left(1 - \frac{\chi_\Siegel(p)}{p^\beta}\right)^{-1}\\
&+ O( N \exp(-c \log^{1/10} N) )
\end{align*}
It will thus suffice to establish the corresponding pseodupolynomially accurate asymptotic
\begin{align}\label{qp}\begin{split}
 \sum_{\substack{n \leq N'\\ n = a'\ (q')}} \mu_\Siegel(n) &= 1_{q_\Siegel|q'} \frac{(N')^\beta \chi_\Siegel(a')}{\beta \phi(q') L'(\beta,\chi_\Siegel)} \prod_{\substack{p|q'\\ p \nmid q_\Siegel}} \left(1 - \frac{\chi_\Siegel(p)}{p^\beta}\right)^{-1}\\
 &+ O( N \exp(-c \log^{1/10} N) )
 \end{split}
\end{align}
for $\mu_\Siegel$.  It suffices to establish the variant estimate
\begin{equation}\label{mup2}
 \sum_{\substack{n \leq N'\\ n = a'\ (q')}} \mu_\Siegel(n) = \frac{(N')^\beta \chi_\Siegel(a')}{\beta \phi(q') L'(\beta,\chi_\Siegel)} \prod_{\substack{p|q'\\ p \nmid q_\Siegel}} \left(1 - \frac{\chi_\Siegel(p)}{p^\beta}\right)^{-1} + O( N \exp(-c \log^{1/10} N) )
\end{equation}
(say) whenever $1 \leq a' \leq q' \leq \exp(O(\log^{1/10} N))$ with $(a',q')=1$ and $q_\Siegel|q'$.  Indeed, this estimate immediately implies \eqref{qp} when $q_\Siegel$ divides $q'$, and when $q_\Siegel$ does not divide $q'$, one splits up the primitive residue class $a'\ (q')$ into primitive residue classes modulo $[q',q_\Siegel]$ on the support of $\mu_\Siegel$, applies \eqref{mup2} to each such class, and sums, using the orthogonality of Dirichlet characters to cancel out the main term.  

We use Definition~\ref{siegel-model} to expand the left-hand of \eqref{mup2} as
$$ \sum_{d \in {\mathcal D}} \mu(d) \sum_{\substack{n \leq N'/d\\ dn = a'\ (q')}} \mu'(n)$$
where ${\mathcal D}$ consists of all the factors $d$ of $P(Q)$ with $(d,q')=1$. As in the proof of \eqref{moa-2}, we can decompose ${\mathcal D}_{\leq} \cup {\mathcal D}_{>}$, where ${\mathcal D}_{\leq}$ are those $d' \in {\mathcal D}$ with $d' \leq \exp(\log^{1/2} N)$ (say) and ${\mathcal D}_>$ are those $d' \in {\mathcal D}$ with $d' > \exp(\log^{1/2} N)$.  The contribution of ${\mathcal D}_{>}$ can be disposed of by the same argument used to prove \eqref{moa-2}, so it remains to show that
\begin{align*}
\sum_{d \in {\mathcal D}_{\leq}} \mu(d) \sum_{\substack{n \leq N'/d\\ dn = a'\ (q')}} \mu'(n)
&= \frac{(N')^\beta \chi_\Siegel(a')}{\beta \phi(q') L'(\beta,\chi_\Siegel)} \prod_{\substack{p|q'\\ p \nmid q_\Siegel}} \left(1 - \frac{\chi_\Siegel(p)}{p^\beta}\right)^{-1}\\
&+ O( N \exp(-c \log^{1/10} N) ).
\end{align*}
By Definition~\ref{siegel-model}, we have
$$ \sum_{\substack{n \leq N'/d\\ dn = a'\ (q')}} \mu'(n) = 
\alpha \frac{\phi(P(Q))}{P(Q)}  \sum_{\substack{n \leq N'/d\\ n = a'/d\ (q')}} (\Lambda_{\Cramer,Q}(n) - \Lambda_{\Siegel}(n)).$$
Applying \eqref{lamsiavg}, as well as Lemma~\ref{alpha-bound}, we can write this as
$$ \alpha \frac{\phi(P(Q))}{P(Q)} \frac{(N'/d)^\beta}{\beta \phi(q')} \chi_\Siegel(a') \chi_\Siegel(d) $$
up to acceptable error terms.  Canceling some terms, it thus suffices to show that
\begin{align*} \alpha \frac{\phi(P(Q))}{P(Q)} \sum_{d \in {\mathcal D}_{\leq}} \frac{\mu(d) \chi_\Siegel(d)}{d^\beta}
&= \frac{1}{L'(\beta,\chi_\Siegel)} \prod_{\substack{p|q'\\ p \nmid q_\Siegel}} \left(1 - \frac{\chi_\Siegel(p)}{p^\beta}\right)^{-1}\\
&+ O( \exp(-c \log^{1/10} N) ).
\end{align*}
A standard Euler product calculation using \eqref{alpha-def} gives
$$ \alpha \frac{\phi(P(Q))}{P(Q)} \sum_{d \in {\mathcal D}} \frac{\mu(d) \chi_\Siegel(d)}{d^\beta}
= \frac{1}{L'(\beta,\chi_\Siegel)} \prod_{\substack{p|q'\\ p \nmid q_\Siegel}} \left(1 - \frac{\chi_\Siegel(p)}{p^\beta}\right)^{-1} $$
so it suffices to show that
$$ \alpha \frac{\phi(P(Q))}{P(Q)} \sum_{d \in {\mathcal D}_>} \frac{\mu(d) \chi_\Siegel(d)}{d^\beta}
\ll \exp(-c \log^{1/10} N) )$$
By Lemma~\ref{alpha-bound} and the triangle inequality it suffices to show that
$$ \sum_{d \in {\mathcal D}_>} \frac{1}{d^\beta} \ll \exp(-c \log^{1/5} N) ).$$
But we can bound 
$$\frac{1}{d^{\beta}} \leq \frac{1}{d^{\beta - \log^{-1/10} N}} \exp( - c \log^{2/5} N )
\leq \frac{1}{d^{1 - 2\log^{-1/10} N}} \exp( - c \log^{2/5} N )$$
when $d \in {\mathcal D}_>$, and from Euler products we have
$$ \sum_{d \in {\mathcal D}} \frac{1}{d^{1 - 2\log^{-1/10} N}} 
\leq \prod_{p \leq Q} \left(1 + \frac{1}{p^{1 - 2\log^{-1/10} N}}\right) \ll \exp(O(\log \log N))$$
and the claim follows.
\end{proof}

We return now to the proof of Theorem~\ref{quant-ortho}. Throughout this section we assume that $\epsilon>0$ is fixed and small in terms of $k$, and that $c_1(\epsilon)>0$ is sufficiently small depending on $k$ (and we reserve the right to decrease $c_1(\epsilon)$ later in the argument as necessary).  We can assume that $N$ is sufficiently large depending on $c_1(\epsilon),k$, as the claim is trivial otherwise.  Let $P$, $G/\Gamma$, $F$, $g$ be as in that theorem. We use $m = O((\log\log N)^{c_1(\epsilon)})$ to denote the dimension of $G$; to avoid some minor notational issues we will assume that $m \geq 2$ (as can be achieved trivially by adding some dummy dimensions).  

We repeat the arguments from~\cite{gt-mobius}, but now performing a more quantitative accounting of the dependence on constants (particularly on the dimension).  We first use a dimension-uniform version of the factorization theorem in~\cite[Theorem 1.19]{gt-leibman}, which we establish in Theorem~\ref{factor-thm}.  We apply that theorem with $M_0 \coloneqq \exp(\log^{1/10-\epsilon/2} N)$ and $A \coloneqq \exp((\log \log N)^{1/2})$ to obtain a quantity
\begin{align}\label{mbound}\exp(\log^{1/10-\epsilon/2} N) \leq M \leq \exp(\log^{1/10-\epsilon/3} N),
\end{align}
a subgroup $G' \subset G$ which is $M$-rational with respect to $G$, and a decomposition $g = \eps g' \gamma$ into polynomial sequences $\eps, g', \gamma \colon \Z \to G$ such that
\begin{itemize}
\item[(i)] $\eps$ is $(M,N)$-smooth;
\item[(ii)] $g'$ takes values in $G'$ and $(g'(n)\Gamma)_{n \in [N]}$ is totally $1/M^A$-equidistributed in $G'/\Gamma'$, with respect to a Mal'cev basis ${\mathcal X}'$ consisting of $M$-rational linear combinations of the basis elements of ${\mathcal X}$;
\item[(iii)] $\gamma$ is $M$-rational and $\gamma(n) \Gamma$ is periodic with period at most $M$.
\end{itemize}

We can partition the arithmetic progression $P$ into $O( M^{m^{O(1)}})$ components $P'$, such that on each of these components the periodic function $\gamma(n) \Gamma$ is equal to an $M$-rational constant $\gamma_{P'} \Gamma$, and the smooth sequence $\eps$ differs by at most $O( M^{-m^C} )$ from a constant $\eps_{P'} \in G$ of distance at most $M$ from the origin, for a large constant $C$.  We can also normalize $\gamma_{P'}$ to be distance $O(M^{m^{O(1)}})$ from the origin.  From this and the Lipschitz nature of $F$, we see (for $C$ large enough) that
$$ F(g(n) \Gamma) = F(\eps_{P'} g'(n) \gamma_{P'} \Gamma ) + O( M^{-1} )$$
for $n \in P'$.  By \eqref{mbound}, the triangle inequality, and Lemma~\ref{point-bound}, it thus suffices to establish the bounds
$$ \sum_{n \in P'} (\mu - \mu_{\Siegel})(n) \overline{F}(\eps_{P'} g'(n) \gamma_{P'} \Gamma) \ll N M^{\exp(m^{O(1)})} ( \exp(- \log^{1/10-\epsilon/4} N) + M^{-A/\exp(m^{O(1)})})$$
and
$$ \sum_{n \in P'} (\Lambda - \Lambda_{\Siegel})(n) \overline{F}(\eps_{P'} g'(n) \gamma_{P'} \Gamma) \ll N M^{\exp(m^{O(1)})} ( \exp(- \log^{1/10-\epsilon/4} N) + M^{-A/\exp(m^{O(1)})} )$$
for all of the progressions $P'$, where the implied constants in the $O(1)$ notation on the right-hand sides of the estimates can be taken to be uniform in $\epsilon$ for $\epsilon$ sufficiently small.  We introduce the conjugated group
$$ G_{P'} \coloneqq \gamma_{P'}^{-1} G' \gamma_{P'}$$
and conjugated polynomial
$$ g_{P'} \coloneqq \gamma_{P'}^{-1} g' \gamma_{P'}$$
that takes values in $G_{P'}$, and the normalized function
$$ F_{P'}(x) \coloneqq \overline{F}(\eps_{P'} \gamma_{P'} x ) - \int_{G_{P'} / (G_{P'} \cap \Gamma)} \overline{F}(\eps_{P'} \gamma_{P'} \cdot )$$
where the integral is with respect to the Haar probability measure on $G_{P'} / (G_{P'} \cap \Gamma)$ (which we can view as a subnilmanifold of $G/\Gamma$).  Using Proposition~\ref{equiprop} to dispose of the contribution of the constant $\int_{G_{P'} / (G_{P'} \cap \Gamma)} \overline{F}(\eps_{P'} \gamma_{P'} \cdot )$ (which can be viewed as the ``major arc'' contribution to these correlations), we are reduced to establishing the bounds
$$ \sum_{n \in P'} (\mu - \mu_{\Siegel})(n) F_{P'}(g_{P'}(n) \Gamma) \ll N M^{\exp(m^{O(1)})} ( \exp(- \log^{1/10-\epsilon/4} N) + M^{-A/\exp(m^{O(1)})} )$$
and
$$ \sum_{n \in P'} (\Lambda - \Lambda_{\Siegel})(n) F_{P'}(g_{P'}(n) \Gamma) \ll N M^{\exp(m^{O(1)})} ( \exp(- \log^{1/10-\epsilon/4} N) + M^{-A/\exp(m^{O(1)})} ).$$
The advantages of this reduction are that the function $F_{P'}$ is not only $1$-bounded and $O(M^{m^{O(1)}})$-Lipschitz (with respect to the Mal'cev basis of $G_{P'} / (G_{P'} \cap \Gamma)$, which is a filtered nilmanifold of complexity $O(M^{m^{O(1)}})$), but it also has mean zero.  By repeating the arguments from \cite[p. 547]{gt-mobius} and keeping track of the constants, we see that the polynomial sequence $g_{P'}$ is totally $1/M^{A / m^{O(1)}}$-equidistributed (note that multiplicative factors of $\exp(\exp(m^{O(1)}))$ can be absorbed into the $M^{A / m^{O(1)}}$ denominator, and that all the $O_m(1)$ exponents appearing in this portion of \cite{gt-mobius} (and \cite{gt-leibman}) are polynomial in $m$).

We can use the Gowers uniformity of $\chi_\Siegel$ to obtain the following bound on the Siegel terms which is acceptable when $q_\Siegel$ is large enough:

\begin{proposition}  We have
$$ \sum_{n \in P'} \mu_{\Siegel}(n) F_{P'}(g_{P'}(n) \Gamma) \ll N M^{m^{O(1)}} q_\Siegel^{-1 / m^{O(1)}} $$
and
$$ \sum_{n \in P'} (\Lambda_{\Siegel}(n) -\Lambda_{\Cramer,Q}(n)) F_{P'}(g_{P'}(n) \Gamma )\ll N M^{m^{O(1)}} q_\Siegel^{-1 / m^{O(1)}}.$$
\end{proposition}

\begin{proof}  We apply~\cite[Proposition 11.2]{gt-linear}, noting that all bounds\footnote{The argument as stated in that paper appeals to the Stone--Weierstrass theorem and the Arzel\'a--Ascoli theorem, but this can be replaced by more quantitative approximation results without difficulty, such as~\cite[Lemma A.9]{gt-quadratic}, combined with standard smooth partitions of unity to allow one to work on regions such as the unit cube rather than on the original nilmanifold. As pointed out to us by James Leng, the required smoothness bounds on the function $P$ constructed in \cite[Proposition 11.5]{gt-linear} also need to be established.  To do this, one can first take advantage of the fact that $\operatorname{HK}^{s+1}(G)$ acts transitively on the graph of $P$ to reduce to establishing smoothness bounds at the origin.  Then one can lift from $G/\Gamma$ to $G$, and reduce to establishing that one corner of a parallelepiped in $HK^{s+1}(G)$ is a smooth function of all the other corners near the origin with the required bounds.  But one can express the first corner as a word in the other corners of length depending only on $s$, and from many applications of the Baker--Campbell--Hausdorff formula this will give the desired quantitative bounds on this corner completion function.} can be shown to be polynomial in the parameters $M, \eps$ with exponents that are polynomial in the dimension $m$, to decompose
$$ F_{P'}(g_{P'}(n) \Gamma) = F_1(n) + F_2(n)$$
where $F_1$ obeys the dual norm bound
$$ \E_{n \in [N]} f(n) F_1(n) \ll (M/\eps)^{m^{O(1)}} \|f\|_{U^k[N]}$$
for any $f \colon [N] \to \C$, and $F_2$ obeys the pointwise bound
$$ F_2(n) \ll \eps$$
for all $n \in [N]$.  Here $0 < \eps \leq 1$ is a parameter that we are at liberty to choose.  By Theorem~\ref{siegel-uniform}, the functions $\mu_\Siegel, \Lambda_{\Siegel} - \Lambda_{\Cramer,Q}$ already have a $U^k[N]$ norm of $O(q_\Siegel^{-c})$; a standard Fourier expansion of $1_{P'}(n)$ in terms of additive characters and the triangle inequality then show that the truncated versions $1_{P'} \mu_\Siegel$, $1_{P'} (\Lambda_{\Siegel} - \Lambda_{\Cramer,Q})$ have a $U^k[N]$ norm of $O(M^{O(1)} q_\Siegel^{-c})$ (note that any logarithmic factors can be easily absorbed into the $M^{O(1)}$ factor).  Applying the above decomposition as well as Lemma~\ref{point-bound}, we see that
$$ \sum_{n \in P'} \mu_{\Siegel}(n) F_{P'}(g_{P'}(n) \Gamma) \ll N M^{O(1)} (M/\eps)^{m^{O(1)}} q_\Siegel^{-c} + \eps N $$
and
$$ \sum_{n \in P'} (\Lambda_{\Siegel}(n) -\Lambda_{\Cramer,Q}(n)) F_{P'}(g_{P'}(n) \Gamma) \ll N M^{O(1)} (M/\eps)^{m^{O(1)}} q_\Siegel^{-c} + \eps N \log N,$$
and the claim then follows by a suitable choice of $\eps$ (noting that the $\log N$ factor can be absorbed into the $M$ factor).
\end{proof}

Based on this proposition, we may now delete the $Q$-Siegel zero contributions except in the regime where
\begin{equation}\label{siegel-small}
q_\Siegel \leq M^{A/\exp(m^{C_1})}
\end{equation}
where $C_1$ is a large constant depending on $k$ (but not on $\epsilon$) that we are at liberty to choose; we can also assume $N$ to be sufficiently large depending on $C_1$ (as well as $k$ and $\epsilon$).  To simplify the notation we assume henceforth that the $Q$-Siegel zero exists and obeys \eqref{siegel-small}; the remaining cases follow by a simplified version of the same argument that deletes all the steps and terms that treat the contribution of the $Q$-Siegel zero.  It will now suffice to obtain estimates of the form
$$ \sum_{n \in P'} (\mu - \mu_{\Siegel})(n) F_{P'}(g_{P'}(n) \Gamma) \ll N (M q_\Siegel)^{\exp(m^{O(1)})} ( \exp(-\log^{1/10-\epsilon/4} N) + M^{-A/\exp(m^{O(1)})} )$$
and
$$ \sum_{n \in P'} (\Lambda - \Lambda_{\Siegel})(n) F_{P'}(g_{P'}(n) \Gamma) \ll N (M q_\Siegel)^{\exp(m^{O(1)})} ( \exp(- \log^{1/10-\epsilon/4} N) + M^{-A/\exp(m^{O(1)})} ),$$
where the implied constants do not depend on $C_1$.

To treat these sums, we make the following standard Vaughan-type decompositions.  Call a sequence $a_d, d \in \N$ of complex numbers \emph{divisor bounded} if one has $a_d \ll (\log N)^{O(1)} \tau^{O(1)}(d)$ for all $d \in [N]$, where $\tau(n) \coloneqq \sum_{d|n} 1$ is the divisor function.

\begin{proposition}[Vaughan-type decompositions]  Any of the four functions $\mu, \mu_\Siegel, \Lambda, \Lambda_\Siegel$ on $[N]$ can be expressed as a convex linear combination of functions one of the following four classes (with uniform constants in the bounds):
\begin{itemize}
\item[(i)] (Type I sum)  A function of the form
$$ n \mapsto \sum_{d \leq N^{2/3}} a_d 1_{d|n} 1_{[N']}(n)$$
where the coefficients $a_d$ are divisor-bounded and $1 \leq N' \leq N$.
\item[(ii)] (Twisted type I sum)  A function of the form
$$ n \mapsto \sum_{d \leq N^{2/3}} a_d 1_{d|n} \chi_\Siegel(n/d) 1_{[N']}(n)$$
where the coefficients $a_d$ are divisor-bounded and $1 \leq N' \leq N$.
\item[(iii)] (Type II sum)  A function of the form
$$ n \mapsto \sum_{d,w > N^{1/3}} a_d b_w1_{dw=n}$$
for some divisor-bounded coefficients $a_d, b_w$.
\item[(iv)] (Negligible sum) A divisor-bounded function $n \mapsto f(n)$ with
$$ \sum_{n \in [N]} |f(n)| \ll N \exp(-\log^{1/2} N).$$
\end{itemize}
\end{proposition}

\begin{proof}  For $\Lambda$ we can use the familiar Vaughan identity~\cite{vaughan}
$$ \Lambda(n) = \Lambda(n) 1_{n \leq N^{1/3}} - \sum_{d \leq N^{2/3}} a_d 1_{d|n} + \sum_{d \leq N^{1/3}} \mu(d) 1_{d|n} \log \frac{n}{d}
+ \sum_{d,w > N^{1/3}} \Lambda(d) b_w 1_{dw=n}$$
where $a_d \coloneqq \sum_{bc=d:\, b,c \leq N^{1/3}} \mu(b) \Lambda(c)$ and $b_w \coloneqq \sum_{c|w:\, c > N^{1/3}} \mu(c)$.  The first term is negligible, the second term is a Type I sum (restricting to $[N]$), and the fourth term is a Type II sum; the third term can be converted to a convex combination of Type I sums by using the fundamental theorem of calculus to write
$$ \log \frac{n}{d} = \log N - \int_1^N 1_{t>n} \frac{dt}{t} - \log d$$
and absorbing all the various logarithmic factors into the divisor-bounded coefficients.  Similarly, for $\mu$ we can use the variant identity
$$ \mu(n) =  \sum_{d \leq N^{2/3}} a'_d 1_{d|n} - \sum_{d,w > N^{1/3}} \mu(d) b_w 1_{dw=n}$$
where $a'_d \coloneqq \sum_{\substack{bc=d\\ b,c \leq N^{1/3}}} \mu(b) \mu(c)$ and $b_w$ is as before; see e.g.,~\cite[Lemma 4.1]{gt-quadratic}.

To handle $\Lambda_\Siegel$, it suffices (using the estimate $P(Q)/\phi(P(Q))\ll (\log N)^{O(1)}$ coming from Mertens' theorem) to show that the functions
\begin{equation}\label{npq}
 n \mapsto 1_{(n,P(Q))=1}
\end{equation}
and
$$ n \mapsto n^{\beta-1} 1_{(n,P(Q))=1} \chi_\Siegel(n)$$
can be expressed in the desired form (absorbing all the constant factors into the divisor-bounded coefficients).  But if $\lambda_d^{+},\lambda_d^{-}$ are the upper and lower linear sieve coefficients, respectively, with level $D=Q^{10(\log N)^{3/5}}$ and sifting parameter $Q$, one can write
\begin{align*}
\sum_{d\leq D}\lambda_d^{-}  1_{d\mid n} \leq 1_{(n,P(Q))}\leq  \sum_{d\leq D}\lambda_d^{+}  1_{d\mid n},
\end{align*}
and by the fundamental lemma~\cite[Lemma 6.3]{iwaniec-kowalski} (bounding the error terms $R^{\pm}$ there as $O(D)$) we have
\begin{align*}
\sum_{n\in [N]}\left|1_{(n,P(Q))}-\sum_{d\leq D}\lambda_d^{\pm}  1_{d\mid n}\right|\ll N\exp(-10 \log^{3/5} N)    
\end{align*}
(say).  Therefore, one can express \eqref{npq} as a Type I sum plus an error term of  $L^{1}[N]$ norm $\ll N\exp(-10 \log^{3/5} N)$, and by multiplying by $\chi_\Siegel$ one can then express $n \mapsto 1_{(n,P(Q))=1} \chi_\Siegel(n)$ as a twisted Type I sum plus an error term of  $L^{1}[N]$ norm at most $\ll N\exp(-10 \log^{3/5} N)$.  Indeed in these cases one can lower the $N^{2/3}$ threshold on $d$ to something much smaller, such as $\exp(O(\log^{7/10} N))$.  Finally, the $n^{\beta-1}$ weight can be handled using the fundamental theorem of calculus identity \eqref{int}.

Now we turn to $\mu_\Siegel = \mu_\local * \mu'$.  From the previous discussion and Lemma~\ref{alpha-bound}, $\mu'$ is already expressible as a convex combination of twisted Type I sums  (where $d$ can be constrained to be at most $\exp(O(\log^{7/10} N))$) plus  an error term of  $L^{1}[N]$ norm $\ll N\exp(-10 \log^{3/5} N)$.  We can then convolve by $\mu_\local 1_{[\exp(5\log^{3/5} N)]}$ and conclude that $\mu_\local 1_{[\exp(5\log^{3/5} N)]} * \mu'$ is also expressible as a convex combination of twisted Type I sums plus a negligible error (note that the values of $d$ encountered stay well below the threshold $N^{2/3}$).  Finally, the remaining term $\mu_\local (1 - 1_{[\exp(5\log^{3/5} N)]}) * \mu'$ can be seen to be negligible by the same arguments used to dispose of the ${\mathcal D}_>$ contributions to \eqref{moa-2} (namely, using the fact that the density of $Q$-smooth numbers in any dyadic interval $[M,2M]$ with $\exp(5\log^{3/5} N)\leq M\leq N$ is $\ll \exp(-5(\log ^{1/2} N)$).
\end{proof}

The contributions of the negligible sums to the previous estimates are acceptable from the triangle inequality.  By a further application of the triangle inequality, it thus suffices to establish the bound
\begin{equation}\label{samp}
\sum_{n \in P'} f(n) F_{P'}(g_{P'}(n) \Gamma) \ll N (M q_\Siegel)^{\exp(m^{O(1)})} M^{-A/\exp(m^{O(1)})} 
\end{equation}
whenever $f$ is a Type I sum, a twisted Type I sum, or a Type II sum.

The Type I and Type II sums were already essentially treated in~\cite[\S 3]{gt-mobius}, and it turns out that the methods also easily extend to cover the twisted Type I case.  We briefly review the argument as follows.  We begin with the twisted Type I case; the Type I case is treated by a simplification of the argument that deletes the role of the $Q$-Siegel character, and is omitted here (and in any case would follow closely the treatment in~\cite[\S 3]{gt-mobius}).  Suppose that we have
\begin{equation}\label{sumnp}
 \left|\sum_{n \in P'} f(n) F_{P'}(g_{P'}(n) \Gamma)\right| \geq \delta N
\end{equation}
for some $0 < \delta < \frac{1}{M q_\Siegel}$ and a twisted Type I sum $f$.  By the definition of such sums and the triangle inequality, this implies that
$$ \sum_{d \leq N^{2/3}} \tau^C(d) \left|\sum_{n \in P'' \cap d\Z} \chi_\Siegel(n/d)F_{P'}(g_{P'}(n) \Gamma)\right| \gg \delta^{O(1)} N  $$
for some constant $C=O(1)$, where $P'' \coloneqq P' \cap [N']$ (note that all $\log^{O(1)} N$ terms can be easily absorbed into the $\delta^{O(1)}$ factor).  Standard divisor sum estimates give
$$ \sum_{d \leq N^{2/3}} \tau^{2C}(d)/d \ll \log^{O(1)} N$$
(with the implied constant depending on $C$), hence by Cauchy--Schwarz
$$ \sum_{d \leq N^{2/3}} d\left|\sum_{n \in P'' \cap d\Z} \chi_\Siegel(n/d)F_{P'}(g_{P'}(n) \Gamma)\right|^2 \gg \delta^{O(1)} N^2,$$
and hence by dyadic decomposition there exists $1 \leq D \leq N^{2/3}$ such that
$$ \sum_{D \leq d \leq 2D} \left|\sum_{n \in P'' \cap d\Z} \chi_\Siegel(n/d)F_{P'}(g_{P'}(n) \Gamma)\right|^2 \gg \delta^{O(1)} \frac{N^2}{D}.$$
Since the inner sum is $O(N/D)$, we conclude that
$$ \left|\sum_{n \in P'' \cap d\Z} \chi_\Siegel(n/d)F_{P'}(g_{P'}(n) \Gamma)\right| \gg \delta^{O(1)} \frac{N}{D} $$
for $\gg \delta^{O(1)} D \log^{-O(1)} N$ natural numbers $d$ in $[D,2D]$.  For such a $d$, we partition into residue classes modulo $dq_\Siegel$ and use the triangle inequality to conclude that
$$ \left|\sum_{n \in [N_d]} F_{P'}(g_{P'}(d(q_\Siegel n + a_d))\Gamma)\right| \gg \delta^{O(1)} \frac{N}{D} $$
for some $1 \leq N_d \leq N/D$ and $1 \leq a_d \leq q_\Siegel$ (note that all $q_\Siegel$ factors can be absorbed into the $\delta^{O(1)}$ factor).  Applying Theorem~\ref{main-variant}, we can then find a horizontal character $\eta_d$ of $G'$ with 
\begin{equation}\label{etad}
 0 < |\eta_d| \ll \delta^{-\exp(m^{O(1)})}
\end{equation}
such that
$$ \| \eta_d \circ g_{P'}(d(q_\Siegel \cdot + a_d)) \|_{C^\infty[N/D]} \ll \delta^{-\exp(m^{O(1)})},$$
where the $\|\cdot\|_{C^{\infty}}$ is defined in~\cite[Definition 2.7]{gt-leibman}.  The parameter $a_d$ is annoying, but we can remove\footnote{We thank the anonymous referee for this suggestion, which patched a gap in a previous version of this argument.} it by applying \cite[Lemma 8.4]{gt-leibman} to conclude that
$$ \| \eta'_d \circ g_{P'}(d(q_\Siegel \cdot)) \|_{C^\infty[N/D]} \ll \delta^{-\exp(m^{O(1)})}$$
for some $\eta'_d$ that continues to obey \eqref{etad}.  The total number of such $\eta'_d$ is $O( \delta^{-\exp(m^{O(1)})})$.  Thus by the pigeonhole principle, we can find one such horizontal character $\eta$ such that
$$ \| \eta \circ g_{P'}(d(q_\Siegel \cdot)) \|_{C^\infty[N/D]} \ll \delta^{-\exp(m^{O(1)})}$$
for $\gg \delta^{\exp(m^{O(1)})} D$ values of $d \in [D,2D]$.  If we expand out the polynomial
\begin{equation}\label{etaq}
\eta \circ g_{P'}(q_\Siegel n) = \beta_k n^k + \dots + \beta_0 \hbox{ mod } 1
\end{equation}
for some real numbers $\beta_0,\dots,\beta_k$, then by applying~\cite[Lemma 3.2]{gt-mobius} we conclude that there is a positive integer $q=O(1)$ such that
$$ \| q d^j \beta_j \|_{\R/\Z} \ll (N/D)^{-j} \delta^{-\exp(m^{O(1)})}$$
for all $j=0,\dots,k$, where $\|x\|_{\R/\Z}$ denotes the distance to the nearest integer.  Applying a Waring-type result from~\cite[Lemma 3.3]{gt-mobius}, we then have for each $j=0,\dots,k$ that
$$ \| q d' \beta_j \|_{\R/\Z} \ll (N/D)^{-j} \delta^{-\exp(m^{O(1)})}$$
for $\gg \delta^{\exp(m^{O(1)})} D^j$ integers $d'$ of size $d'=O(D^j)$.  Applying Vinogradov's lemma~\cite[Lemma 3.4]{gt-mobius}, and clearing denominators, we then conclude that there is a positive integer $K \ll  \delta^{\exp(m^{O(1)})}$ such that
$$ \| K \beta_j \|_{\R/\Z} \ll N^{-j} \delta^{-\exp(m^{O(1)})}$$
for all $j=0,\dots,k$, and thus by \eqref{etaq} 
$$ \| K q_\Siegel^k \eta \circ g_{P'}\|_{C^\infty[N]} \ll \delta^{-\exp(m^{O(1)})}.$$
On the other hand, $g_{P'}$ is totally $1/M^{A / m^{O(1)}}$-equidistributed.
Arguing as in~\cite[\S 3]{gt-mobius} and noting that all exponents of the form $O_m(1)$ are in fact polynomial in $m$, these two facts are incompatible unless
\begin{equation}\label{party}
\delta^{-\exp(m^{O(1)})} \gg M^{A/m^{O(1)}}
\end{equation}
which (when combined with the constraint $\delta \leq \frac{1}{Mq_\Siegel}$) gives the desired bound \eqref{samp}.

For the Type II case, we can again start by assuming \eqref{sumnp} for some $0 < \delta < \frac{1}{M}$ and some Type II sum $f$.  The contribution of those $n$ less than $\delta^C N$ for a large absolute constant $C$ can easily be seen to be negligible, so one can assume without loss of generality that $|P'|$ lies in the interval $[\delta^C N, N]$.  One has
$$ \sum_{d > N^{1/3}} \sum_{w > N^{1/3}} a_d b_w F_{P'}(g_{P'}(dw) \Gamma) 1_{P'}(dw) \gg \delta^{O(1)} N$$
for some divisor-bounded $a_d, b_w$, and then after some dyadic decomposition and Cauchy--Schwarz (cf.,~\cite[Proposition 7.2]{gt-quadratic}) one can find $N^{1/3} \ll D, W \ll \delta^{-O(1)} N^{2/3}$ with $DW = \delta^{O(1)} N$ such that
$$ \sum_{d,d' \in [D,2D]} \sum_{w,w' \in [W,2W]} F_{P'}(g_{P'}(dw) \Gamma) \overline{F_{P'}}(g_{P'}(dw') \Gamma) \overline{F_{P'}}(g_{P'}(d'w) \Gamma) F_{P'}(g_{P'}(d'w') \Gamma) \gg \delta^{O(1)} N.$$
One now repeats the arguments used to treat the Type II case in~\cite[\S 3]{gt-mobius} more or less verbatim (noting that all exponents are of order $\exp(m^{O(1)})$ at worst) to obtain a contradiction to the total $1/M^{A / m^{O(1)}}$-equidistribution of $g_{P'}$ unless \eqref{party} holds, and we again obtain \eqref{samp} as desired.  This concludes the proof of Theorem~\ref{quant-ortho}.

\section{Applying densification} \label{sec: densification}

We now use densification methods to establish a general transference principle (which seems of independent interest) that converts inverse theorems for the Gowers norms for $1$-bounded functions to inverse theorems for Gowers norms for $\nu$-bounded functions for various ``pseudorandom'' weights $\nu$.  Our pseudorandomness condition will be relatively mild (a $U^{2k}$ estimate on $\nu-1$), and the losses in the transference argument will only be polynomial in nature.  However, one drawback of the theorem is that the input inverse theorem must also have polynomial bounds. 

In Subsection~\ref{subsec:mainproof}, we will use Theorem~\ref{transf} to complete the proof of Theorem~\ref{main-2} in the von Mangoldt case.

\subsection{Transferring inverse theorems}

\begin{theorem}[Transference principle for $U^k$ inverse theorems]\label{transf}  Let $k\geq 2$ be fixed. Let $G = (G,+)$ be a finite abelian group.  Suppose that for every $0 < \delta \leq 1/2$ there is a family $\Psi_\delta$ of $1$-bounded functions $\psi \colon G \to \C$, non-increasing in $\delta$ and closed under translations and complex conjugation, obeying the following $U^k$ inverse theorem:
\begin{itemize}
\item[(i)]  If $0 < \delta \leq 1/2$ and $f \colon G \to \C$ is $1$-bounded with $\|f\|_{U^k(G)} \geq \delta$, then there exists $\psi \in \Psi_\delta$ such that $|\E_{x \in G} f(x) \overline{\psi(x)}| \gg \delta^{B}$ for some $B>0$.
\end{itemize}
Let $C_0$ be sufficiently large depending on $k$, let $0 < \delta \leq 1/2$, and let $\nu \colon G \to \R^+$ be a weight with 
\begin{equation}\label{nu-pseud}
\| \nu - 1 \|_{U^{2k}(G)} \leq \delta^{C_0}.
\end{equation}
Let $f \colon G \to \C$ be $\nu$-bounded with 
\begin{equation}\label{fukg}
\|f\|_{U^k(G)} \geq \delta.
\end{equation}
Then there exists $\psi_1,\dots,\psi_{2^k-1} \in \Psi_{\delta^{O(1)}}$ such that
$$ \left|\E_{x \in G} f(x) \prod_{j=1}^{2^k-1} \overline{\psi_j}(x) \right| \gg \delta^{O(1)}.$$
\end{theorem}

We remark that this theorem strengthens a similar result in~\cite{dodos-k}, in that the class $\Psi_\delta$ is allowed to be more general than the space of ``dual functions'', and the bounds are polynomial in nature rather than qualitative.

We now begin the proof of this theorem. Let the notation and hypotheses be as in Theorem~\ref{transf}. From \eqref{fukg} we have
\begin{equation}\label{fip}
 \left|\E_{(x,\vec h) \in G^{k+1}} \prod_{\omega \in \{0,1\}^k} f_\omega(x + \omega \cdot \vec h)\right| \gg \delta^{O(1)}
\end{equation}
where $f_0 = f$, and all the other $f_\omega \colon G \to \C$ are either equal to $f$ or its complex conjugate.  The key step is

\begin{proposition}[Densification of a single factor]\label{dense}  Suppose that the bound \eqref{fip} holds for some $\nu+1$-bounded functions $f_\omega, \omega \in \{0,1\}^k$.  Let $\omega_0 \in \{0,1\}^k$.  Then we have
$$ \left|\E_{(x,\vec h) \in G^{k+1}} \prod_{\omega \in \{0,1\}^k} \tilde f_\omega(x + \omega \cdot \vec h)\right| \gg \delta^{O(1)}$$
where $\tilde f_\omega = f_\omega$ for $\omega \in \{0,1\}^k \backslash \{\omega_0\}$, and $\tilde f_{\omega_0} \in \Psi_{\delta^{O(1)}}$.
\end{proposition}

Indeed, after applying this proposition $2^k-1$ times starting with \eqref{fip}, we conclude that
$$
 \left|\E_{(x,\vec h) \in G^{k+1}} f(x) \prod_{\omega \in \{0,1\}^k \backslash \{0\}^k} \psi_\omega(x + \omega \cdot \vec h)\right| \gg \delta^{O(1)}
$$
for some $\psi_\omega \in \Psi_{\delta^{O(1)}}$ for all $\omega \in \{0,1\}^k \backslash \{0\}^k$ (one can use the non-decreasing nature of $\Psi$ to make the implied constant in $O(1)$ uniform in $\omega$).  In particular, by the pigeonhole principle there exists $h_1,\dots,h_k \in G$ such that
$$
\left |\E_{x \in G} f(x) \prod_{\omega \in \{0,1\}^k \backslash \{0\}^k} \psi_\omega(x + \omega \cdot \vec h)\right| \gg \delta^{O(1)}
$$
giving Theorem~\ref{transf} thanks to the translation and conjugation invariance of $\Psi_{\delta^{O(1)}}$.  

It remains to prove Proposition~\ref{dense}.  By relabeling we may assume $\omega_0 = 0^k$.  By replacing $\nu$ with $\frac{\nu+1}{2}$ (and adjusting $C_0$ if necessary), and then rescaling by various factors of $2$, we may assume that the $f_\omega$ are $\nu$-bounded rather than $\nu+1$-bounded. Now we adapt the arguments of Conlon--Fox--Zhao~\cite{cfz-rel}. We have
$$ |\E_{x \in G} f_{0^k}(x) F(x)| \gg \delta^{O(1)}$$
where $F \colon G \to \C$ is the dual function
$$ F(x) \coloneqq \E_{\vec h \in G^k} \prod_{\omega \in \{0,1\}^k \backslash \{0\}^k} f_\omega(x + \omega \cdot \vec h).$$
Since $f_{0^k}$ is $\nu$-bounded, we conclude from Cauchy-Schwarz that
$$ (\E_{x \in G} \nu(x)) (\E_{x \in G} \nu(x) |F(x)|^2) \gg \delta^{O(1)}.$$
Since 
$$ \E_{x \in G} \nu(x) = \|\nu\|_{U^1(G)} \leq \|\nu\|_{U^k(G)} \leq 1 + \|\nu-1\|_{U^k(G)} \ll 1$$
we conclude that
\begin{equation}\label{e1}
 \E_{x \in G} \nu(x) |F(x)|^2 \gg \delta^{O(1)}.
\end{equation}
Next we claim that
\begin{equation}\label{exg}
 \E_{x \in G} (\nu-1)(x) |F(x)|^2 \ll \delta^{C_0}.
\end{equation}
We can write the left-hand side of \eqref{exg} as
$$ \E_{(x,\vec h) \in G^{2k+1}} \prod_{\omega \in \{0,1\}^{2k}} f_\omega(x + \omega \cdot \vec h)$$
where we have
\begin{align*}
f_{0^{2k}}(x) &\coloneqq \nu(x)-1 \\
f_{\vec \omega, 0^k}(x) &\coloneqq f_{\vec \omega}(x)\\
f_{0^k, \vec \omega}(x) &\coloneqq \overline{f}_{\vec \omega}(x)
\end{align*}
for $\vec  \omega \in \{0,1\}^k \backslash \{0\}^k$, and $f_\omega(x) \coloneqq 1$ for all other $\omega \in \{0,1\}^{2k}$ not covered by the preceding definitions.  By the Gowers--Cauchy--Schwarz inequality \eqref{gcz}, we thus have
$$
 \E_{x \in G} (\nu-1)(x) |F(x)|^2 \leq \prod_{\omega \in \{0,1\}^{2k}} \| f_\omega\|_{U^{2k}(G)}
\leq \| \nu-1\|_{U^{2k}(G)} \|\nu+1\|_{U^{2k}(G)}^{2^{2k}-1},$$
and the claim now follows from \eqref{nu-pseud} and the triangle inequality.

From \eqref{e1}, \eqref{exg} and the triangle inequality we conclude (for $C_0$ large enough) that
\begin{equation}\label{exf}
 \E_{x \in G} |F(x)|^2 \gg \delta^{O(1)}.
\end{equation}
The function $F$ is not quite bounded.  However, as the $f_\omega$ are all $\nu$-bounded, we certainly have the pointwise bound $|F| \leq {\mathcal D} \nu$, where ${\mathcal D} \nu$ is the dual function
$$ {\mathcal D} \nu(x) \coloneqq \E_{h \in G^k} \prod_{\omega \in \{0,1\}^k \backslash \{0\}^k} \nu(x + \omega \cdot \vec h).$$
We observe the moment estimates
\begin{equation}\label{mom}
 \E_{x \in G} {\mathcal D} \nu(x)^j = 1 + O( \delta^{C_0} )
\end{equation}
for $j=0,1,2$.  We just prove this for $j=2$, as the $j=0,1$ claims are similar (and easier).  We can expand
$$ \E_{x \in G} {\mathcal D} \nu(x)^2 = \E_{(x,\vec h) \in G^{2k+1}} \prod_{\omega \in \{0,1\}^{2k}} g_\omega(x + \omega \cdot \vec h)$$
where
\begin{align*}
 g_{\vec \omega, 0^k}(x) &\coloneqq \nu(x)\\
g_{0^k, \vec \omega}(x) &\coloneqq \nu(x)
\end{align*}
for $\vec \omega \in \{0,1\}^k \backslash \{0\}^k$, and $g_\omega(x) \coloneqq 1$ for all other $\omega \in \{0,1\}^{2k}$ not covered by the preceding definitions.  We split each $g_\omega$ that is of the form $\nu$ into $1$ and $\nu-1$.  Applying the triangle inequality \eqref{triangle} and the Gowers--Cauchy--Schwarz inequality \eqref{gcz}, we can thus write
$$ \E_{x \in G} {\mathcal D} \nu(x)^2 = 1 + O( \| \nu-1\|_{U^{2k}(G)} (1 + \| \nu-1\|_{U^{2k}(G)})^{2^{2k}-1} ),$$
and the claim follows from \eqref{nu-pseud}.  

From \eqref{mom} we have
\begin{equation}\label{lad}
 \E_{x \in G} |{\mathcal D} \nu(x)-1|^2 \ll \delta^{C_0}.
\end{equation}
Now define the truncated version
$$ \tilde F(x) \coloneqq \min(|F(x)|,1) \mathrm{sgn}(F(x)),$$
where $\mathrm{sgn}(F(x))$ is equal to $F(x)/|F(x)|$ when $F(x) \neq 0$ and equal to zero when $F(x) = 0$.  Then $\tilde F$ is $1$-bounded and
\begin{equation}\label{ftf}
|F(x) - \tilde F(x)| \leq \max(|F(x)|-1, 0) \leq |{\mathcal D} \nu(x) - 1|
\end{equation}
so from \eqref{lad} and Cauchy--Schwarz we have
$$ \E_{x \in G} \overline{F}(x) (F(x) - \tilde F(x)) \leq \E_{x \in G} |F(x)-\tilde F(x)|^2 + |F(x)-\tilde F(x)| \ll \delta^{C_0/2}.$$
Hence by \eqref{exf} and the triangle inequality we have
$$
 |\E_{x \in G} \overline{F}(x) \tilde F(x)| \gg \delta^{O(1)}.$$
We rewrite the left-hand side as
$$ |\E_{\vec h \in G^k} \prod_{\omega \in \{0,1\}^{k}} f^*_\omega(x + \omega \cdot \vec h)|$$
where
\begin{align*}
 f^*_{0^k}(x) &\coloneqq \tilde F(x) \\
 f^*_\omega(x) &\coloneqq \overline{f_\omega}(x)
\end{align*}
for $\omega \in \{0,1\}^k \backslash \{0\}^k$.  The $f^*_\omega$ all have $U^k(G)$ norm of at most $\| \nu \|_{U^k(G)} \ll 1$ thanks to \eqref{nu-pseud}, hence by the Gowers--Cauchy--Schwarz inequality \eqref{gcz} one has
$$ \| \tilde F \|_{U^k(G)} \gg \delta^{O(1)}.$$
Applying the hypothesis in Theorem~\ref{transf}(i), we conclude that there exists $\psi \in \Psi_{\delta^{O(1)}}$ such that
$$ |\E_{x \in G} \tilde F(x) \psi(x)| \gg \delta^{O(1)}.$$
On the other hand, from Cauchy--Schwarz we have
$$ \E_{x \in G} (F(x)-\tilde F(x)) \psi(x) \ll (\E_{x \in G} |F(x)-\tilde F(x)|^2)^{1/2} \ll \delta^{C_0/2}$$
thanks to \eqref{lad}, \eqref{ftf}.  Hence by the triangle inequality (for $C_0$ large enough) we have
$$ \E_{x \in G} F(x) \psi(x) \gg \delta^{O(1)}.$$
But this rearranges to give the conclusion of Proposition~\ref{dense}.  The proof of Theorem~\ref{transf} is now complete.

We now combine this theorem with Manners' inverse theorem to obtain

\begin{theorem}[Transferred inverse theorem]\label{transfinv} Let $0 < \delta < 1/2$, and let $\nu \colon [N] \to \C$ be such that
$$ \| \nu - 1 \|_{U^{2k}[N]} \leq \delta^{C_0}$$
for some constant $C_0$ that is sufficiently large depending on $k$.  Let $f \colon [N] \to \C$ be a $\nu$-bounded function such that
\begin{equation}\label{fund}
\|f\|_{U^k[N]} \geq \delta.
\end{equation}
Then there exist a (filtered) nilmanifold $G/\Gamma$ of degree $k-1$, dimension $O(\delta^{-O(1)})$, and complexity at most $\exp\exp(O(1/\delta^{O(1)}))$, a $1$-bounded Lipschitz function $F \colon G/\Gamma \to \C$ of Lipschitz constant at most $\exp\exp(O(1/\delta^{O(1)}))$, and a polynomial map $g \colon \Z \to G$, such that
$$ |\E_{n \in [N]} f(n) \overline{F}(g(n) \Gamma)| \gg \exp(-\exp(O(1/\delta^{O(1)}))).$$
\end{theorem}

\begin{proof}  As in the proof of Theorem~\ref{manners-inv}, we pick a prime $N'$ with $10N \leq N' \leq 20N$ and extend $f$ by zero to $\Z/N'\Z$; we also extend $\nu$ by $1$ to $\Z/N'\Z$, and observe that $\| \nu - 1 \|_{U^{2k}(\Z/N'\Z)} \ll \delta^{C_0}$.

To apply Theorem~\ref{transf}, we will need an inverse theorem that has polynomial correlation bounds.  This is not directly provided by Theorem~\ref{manners-inv}; however, such an inverse theorem does appear in the work of Manners~\cite{manners}.  Indeed, we see from~\cite[Lemmas 5.4.1, 5.5.1]{manners} (applying~\cite[Lemma 5.5.1]{manners} inductively, as in~\cite[p. 102]{manners}), that if $f \colon \Z/N'\Z \to \C$ is $1$-bounded with $ \|f\|_{U^k(\Z/N'\Z)} \geq \delta$, then there exists a $1$-bounded function $\psi \colon \Z/N'\Z \to \C$ with the polynomial correlation bound
$$ |\E_{n \in \Z/N'\Z} f(n) \overline{\psi(n)}| \gg \delta^{O(1)}$$
such that $\psi$ is of the form
$$ \psi(n) = \sum_{i=1}^T \alpha_i \overline{F}_i(g_i(n) \Gamma_i) $$
with $T \ll \exp(\exp(\delta^{-O(1)}))$, the $\alpha_i$ complex numbers with $|\alpha_i| \leq 1$, and for each $i$, $G_i/\Gamma_i$ is a filtered nilmanifold of degree $k-1$, dimension $O(\delta^{-O(1)})$, and complexity at most $\exp\exp(O(1/\delta^{O(1)}))$, $F_i \colon G_i/\Gamma_i \to \C$ is a $1$-bounded Lipschitz function $F: G/\Gamma \to \C$ of Lipschitz constant at most $\exp\exp(O(1/\delta^{O(1)}))$, and $g_i \colon \Z \to G_i$ is a polynomial map with $g_i \Gamma$ periodic with period $N'$.  Let us call the collection of all such $\psi$ (with appropriate choices of implied constants) ${\mathcal F}_\delta$; note that this collection is invariant under translation and complex conjugation.  We may now apply Theorem~\ref{transf} to the $\nu$-bounded function $f$ in the hypotheses of this theorem, and conclude that there exist $\psi_1,\dots,\psi_{2^k-1} \in {\mathcal F}_{\delta^{O(1)}}$ such that
$$\left|\E_{x \in \Z/N'\Z} f(x) \prod_{j=1}^{2^k-1} \overline{\psi_j}(x) \right| \gg \delta^{O(1)}.$$
Applying the pigeonhole principle, and taking the tensor product of various nilsequences, we conclude a correlation
$$ |\E_{n \in \Z/N'\Z} f(n) \overline{F}(g(n) \Gamma)| \gg \exp(-\exp(\delta^{-O(1)}))$$
where $G/\Gamma$ is a filtered nilmanifold of degree $k-1$, dimension $O(\delta^{-O(1)})$, and complexity at most $\exp\exp(O(1/\delta^{O(1)}))$, $F \colon G/\Gamma \to \C$ is a $1$-bounded Lipschitz function $F: G/\Gamma \to \C$ of Lipschitz constant at most $\exp\exp(O(1/\delta^{O(1)}))$, and $g \colon \Z \to G$ is a polynomial map with $g \Gamma$ periodic of period $N'$.  Now argue as in the proof of Theorem~\ref{manners-inv} to conclude.
\end{proof}

\subsection{Completing the proof of the main theorem}\label{subsec:mainproof}

Now we can show how the bound \eqref{lasi} in Theorem~\ref{main-2} follows from the bound \eqref{mmus-2} given by Theorem~\ref{quant-ortho}. This will complete the proof of Theorem~\ref{main-2} and hence that of Theorem~\ref{main}.  We begin with an application of the ``$W$-trick''.  Let $W \coloneqq P( \log^\eps N)$, where $\eps > 0$ is a small constant depending on $k$ to be chosen later; we may assume that $N$ is sufficiently large depending on $\eps$.  Observe that the set $\{ n \in [N]: (n,W)=1\}$ contains the entire support of $\Lambda_{\Siegel}$, as well as the support of $\Lambda$ except for $O(\log^{O(1)} N)$ numbers which give a negligible contribution to the $U^k[N]$ norm. Thus it will suffice to show the doubly logarithmic decay bound
$$
 \| (\Lambda - \Lambda_{\Siegel}) 1_{(\cdot,W)=1} \|_{U^k[N]} \ll (\log\log N)^{-c}.$$
By Corollary~\ref{slice}, this will follow once we show that
\begin{align}\label{lambda-siegel} \left\| \frac{\phi(W)}{W} (\Lambda - \Lambda_{\Siegel})(W\cdot+b) \right\|_{U^k[\frac{N-b}{W}]} \ll (\log\log N)^{-c}
\end{align}
for all $1 \leq b \leq W$ coprime to $W$.

Fix $b$.  Now we use a quantitative variant of the well known fact (see~\cite{gt-longaps}) that $\frac{\phi(W)}{W} \Lambda - 1$ can be bounded by a pseudorandom weight, but now observing that we can attain logarithmic accuracy in the pseudorandomness bound.

\begin{proposition}\label{tickle}  $\frac{\phi(W)}{W} (\Lambda - \Lambda_{\Siegel})(W\cdot+b)$ is $C\nu$-bounded for some $C = O(1)$ depending only on $k$ and some $\nu \colon [\frac{N-b}{W}] \to \R^+$ with $\| \nu - 1 \|_{U^{2k}[\frac{N-b}{W}]} \ll \log^{-c\eps} N$.
\end{proposition}

\begin{proof}  By the triangle inequality \eqref{triangle}, it suffices to establish this claim for $\frac{\phi(W)}{W} \Lambda(W\cdot+b)$ and
$\frac{\phi(W)}{W} \Lambda_\Siegel(W\cdot+b)$ separately.  In the latter case, we see from Definition~\ref{siegel-model} that
$$ \left|\frac{\phi(W)}{W} \Lambda_\Siegel(Wn+b)\right| \leq \frac{\phi(W)}{W} \Lambda_{\Cramer,Q}(Wn+b)$$
and the claim in this case follows from Corollary~\ref{cram}.

Now we turn to $\frac{\phi(W)}{W} \Lambda(W\cdot+b)$.  Here we can basically follow the analysis of Goldston--Y{\i}ld{\i}r{\i}m correlation estimates from~\cite[Appendix D]{gt-linear}, though with a slightly more careful accounting in order to obtain suitable estimates.  We choose a smooth function $\chi \colon \R \to \R_{\geq 0}$ supported on $[-2,2]$ that equals $1/2$ on $[-1,1]$ with $\int_1^2 \chi'(x)^2\, dx = 1$.  We set $R \coloneqq N^\gamma$ for some sufficiently small constant $0 < \gamma < 1/2$ depending only on $k$ (and independent of $\eps$).  Following~\cite[Appendix D]{gt-linear}, we introduce the truncated divisor sum
$$ \Lambda_{\chi,R,2}(n) \coloneqq \log R \left( \sum_{d|n} \mu(d) \chi\left(\frac{\log d}{\log R}\right) \right)^2.$$
From~\cite[Lemma D.2]{gt-linear} and the choice of $\chi$, the sieve factor $c_{\chi,2} = \int_0^\infty |\chi'(x)|^2\ dx$ associated to this divisor sum via~\cite[Definition D.1]{gt-linear} is simply
\begin{equation}\label{ch2}
c_{\chi,2} = 1.
\end{equation}  
We then set 
\begin{align}\label{nulambda}
\nu(n) \coloneqq \frac{\phi(W)}{W} \Lambda_{\chi,R,2}(Wn+b).
\end{align}
Let $\Lambda'$ be the restriction of $\Lambda$ to those primes greater than $R^2$.  It is not difficult to see that the error $\frac{\phi(W)}{W} \Lambda(W\cdot+b) - \frac{\phi(W)}{W} \Lambda'(W\cdot+b)$ (supported on primes up to $R^2$, as well as powers of primes, and bounded in size by $O(\log N)$) is non-negative with $U^{2k}[\frac{N-b}{W}]$ norm as small as $O(N^{-c})$, so by \eqref{triangle} we may freely replace $\frac{\phi(W)}{W} \Lambda(W\cdot+b)$ with $\frac{\phi(W)}{W} \Lambda'(W\cdot+b)$. By the definition of $\nu$ in \eqref{nulambda} and  the fact that $\chi(0)=1/2$, we easily verify the pointwise bound
$$ 0 \leq \frac{\phi(W)}{W} \Lambda'(Wn+b) \ll_\gamma \nu(n)$$
for all $n$.  It will thus suffice to show the logarithmic decay bound
$$ \| \nu - 1 \|_{U^{2k}[\frac{N-b}{W}]}^{2^{2k}} \ll \log^{-c\eps} N.$$
Expanding out the left-hand side, it suffices to show that
\begin{equation}\label{enho}
 \E_{(n,\vec h) \in \Omega} \prod_{\omega \in S} \frac{\phi(W)}{W} \Lambda_{\chi,R,2}(W(n+\omega \cdot \vec h)+b)
= \mathrm{vol}(\Omega) + O( (N/W)^{2k+1} \log^{-c\eps} N )
\end{equation}
for all subsets $S$ of $\{0,1\}^{2k}$, where $\Omega \subset \R^{2k+1}$ is the convex body
$$ \Omega \coloneqq \{ (x, \vec y) \in \R^{2k+1} \colon 0 < W(x+\omega \cdot \vec y)+b \leq N\,\, \forall \omega \in \{0,1\}^{2k} \}.$$
Suppose that we directly apply the estimate\footnote{This theorem as stated requires $\gamma$ to be sufficiently small depending on $W$ (represented in \cite{gt-linear} by the parameter $L$), but the bound $R \leq N^\gamma$ is only used before \cite[(D.4)]{gt-linear} to show that an expression of the form $O( L^{O(1)} R^{O(1)} N^{d-1} \log^t R )$ (here we have made the dependence on $L$ explicit) is equal to $o(N^d)$, and this can be achieved with $R \leq N^\gamma$ and $\gamma$ independent of $L$, so long as we also have $L \leq N^\gamma$, which is also the case here since $L = O(W)$ and $N$ is assumed to be sufficiently large.} in~\cite[Theorem D.3]{gt-linear}, using \eqref{ch2} to eliminate the role of the sieve factors.  Then we can express the left-hand side of \eqref{enho} as
\begin{equation}\label{lop}
 \left(\frac{\phi(W)}{W}\right)^{\# S} \left(\mathrm{vol}(\Omega) \prod_p \beta_p + O\left( \frac{(N/W)^{2k+1}}{\log^{1/20} R} e^{O(X)} \right)\right)
\end{equation}
where $\beta_p$ are the usual local factors
$$ \beta_p \coloneqq \E_{(n,\vec h) \in (\Z/p\Z)^{2k+1}} \prod_{\omega \in S} \frac{p}{p-1} 1_{W(n+\omega \cdot \vec h)+b \neq 0},$$
$X$ is the quantity
$$ X \coloneqq \sum_{p \in P} p^{-1/2}$$
and $P$ is the set of primes $p$ which are ``exceptional'' in the sense that at least two of the affine forms 
\begin{equation}\label{linear-form}
(x,\vec y) \mapsto W(x+\omega \cdot \vec y) + b
\end{equation}
for $\vec \omega \in \{0,1\}^{2k}$ are linearly dependent modulo $p$.

Since $W = P(\log^\eps N)$, one has $\beta_p=(\frac{p}{p-1})^{\# S}$ for $p < \log^\eps N$, while from the inclusion-exclusion calculation used in the proof of Proposition~\ref{lineq} one has $\beta_p = 1 + O(1/p^2)$ for $p \geq \log^\eps N$.  Thus
\begin{equation}\label{pbp-0}
 \prod_p \beta_p =  \left(\frac{W}{\phi(W)}\right)^{\# S} (1 + O( \log^{-\eps} N ) ).
\end{equation}
Since $\mathrm{vol}(\Omega) \ll (N/W)^{2k+1}$, the main term in \eqref{lop} is acceptable.  If it were not for the $e^{O(X)}$ term, the error term in \eqref{lop} would similarly be acceptable; unfortunately, as defined in~\cite[Appendix D]{gt-linear}, the exceptional primes consist precisely of all the primes $p$ up to $\log^\eps X$, and this would ostensibly lead to an unacceptably large error term in \eqref{lop}.  But, an inspection of the proof of \cite[Proposition D.4]{gt-linear} reveals that the $e^{O(X)}$ loss arises from three sources. One is from the crude bound
\begin{equation}\label{pbp}
 \prod_p \beta_p \leq e^{O(X)}
\end{equation}
(see \cite[(D.14)]{gt-linear}); one is from the variant
\begin{equation}\label{pbp-alt}
 \prod_{p > \log^{1/10} R} \beta_p \leq 1 + O(  e^{O(X)} \log^{-1/20} R )
\end{equation}
(see \cite[equation after (D.15)]{gt-linear}); and the third arises from the estimate
\begin{equation}\label{pbp-2}
\sum_{p \in P_\Psi: p > \log^{1/10} R} p^{-1} = O(X \log^{-1/20} R)
\end{equation}
appearing in the fourth display after~\cite[(D.16)]{gt-linear}.
 Of course, for the first  estimate \eqref{pbp} we may use the superior bound \eqref{pbp-0} instead in our case.  In our cases none of the exceptional primes exceed $\log^{\eps} N < \log^{1/10} R$, and so one can replace $X$ with $0$ in \eqref{pbp-alt}, \eqref{pbp-2}.  As a consequence of these observations, the $e^{O(X)}$ factor in \cite[Proposition D.4]{gt-linear} may be replaced with $\left(\frac{W}{\phi(W)}\right)^{\# S}$, and the error term in \eqref{lop} is now also acceptable, giving the claim.
\end{proof}

\begin{proof}[Proof of Theorem~\ref{main-2} for $\Lambda$]
Combining Proposition~\ref{tickle} with (the contrapositive of) Theorem~\ref{transfinv}, we see that it suffices to show (for a sufficiently small constant $c_1>0$) that  one has the pseudopolynomial bound
\begin{equation}\label{ewn}
 \E_{n \in [\frac{N-b}{W}]} \frac{\phi(W)}{W} (\Lambda - \Lambda_{\Siegel})(Wn+b) \overline{F}(g(n) \Gamma) \ll \exp(-c\log^c N)
\end{equation}
whenever $G/\Gamma$ is a (filtered) nilmanifold $G/\Gamma$ of degree $k-1$, dimension at most $(\log\log N)^{c_1}$ and complexity at most $\exp(\log^{c_1} N)$, $F \colon G/\Gamma \to \C$ is a $1$-bounded Lipschitz function of Lipschitz constant at most $\exp(\log^{c_1} N)$, and $g \colon \Z \to G$ is a polynomial map.  Using~\cite[Lemma 4.2]{MRTTZ}, we can write $g(n) = \tilde g(Wn+b)$ for another polynomial map $\tilde g \colon \Z \to G$.  But from Theorem~\ref{quant-ortho} we have
$$ \sum_{n \leq N: n = b\ (W)} (\Lambda - \Lambda_{\Siegel})(n) \overline{F}(\tilde g(n) \Gamma) \ll N \exp(-c\log^c N)$$
for some $c>0$ independent of $\eps$, and the claim \eqref{ewn} then follows for $\eps$ small enough.  This (finally!) completes the proof of Theorem~\ref{main-2}, and hence that of Theorem~\ref{main}.
\end{proof}

We can now quickly deduce Corollary~\ref{cor:Wtrick} from our main theorem.

\begin{proof}[Proof of Corollary~\ref{cor:Wtrick}] Let $w=(\log \log N)^{1/2}$. By Theorem~\ref{siegel-uniform}, we have
\begin{align}\label{siegel-cramer}
\|\Lambda_{\Siegel}-\Lambda_{\Cramer,Q}\|_{U^k[N]}&\ll \log^{-c} N.
\end{align}
Using the Fourier expansion $1_{n\equiv b\pmod W}=\frac{1}{W}\sum_{1\leq a\leq W}e\left(\frac{a(n-b)}{W}\right)$, the triangle inequality for the Gowers norms, and the fact that $\|fe(\xi\cdot)\|_{U^k[N]}=\|f\|_{U^k[N]}$ for any function $f$ and any $\xi \in \mathbb{R}$, we deduce from \eqref{siegel-cramer} that
\begin{align}\label{siegel-cramer2}
\left\|\frac{\phi(W)}{W}\left(\Lambda_{\Siegel}(W\cdot+b)-\Lambda_{\Cramer,Q}(W\cdot+b)\right)\right\|_{U^k[\frac{N-b}{W}]}&\ll W^{(k+1)/2^k}\log^{-c} N\ll \log^{-c} N.
\end{align}
From Proposition~\ref{cram}, we have
\begin{align}\label{lambda-cramer}
\left\|\frac{\phi(W)}{W}\Lambda_{\Cramer,Q}(W\cdot+b)-1\right\|_{U^k[\frac{N-b}{W}]}\ll w^{-c}.
\end{align}
Let $w'=\log^{\varepsilon} N$ where $\varepsilon$ is as in Subsection~\ref{subsec:mainproof}. Also let $W'=\prod_{p\leq w'}p$. Then by Corollary~\ref{slice} and \eqref{lambda-siegel} we have
\begin{align*}
&\max_{(b,W)=1}\left\| \frac{\phi(W)}{W} (\Lambda - \Lambda_{\Siegel})(W\cdot+b) \right\|_{U^k[\frac{N-b}{W}]}\ll \max_{(b',W')=1} \left\| \frac{\phi(W')}{W'} (\Lambda - \Lambda_{\Siegel})(W'\cdot+b') \right\|_{U^k[\frac{N-b'}{W'}]}\\
\ll& (\log \log N)^{-c}.   
\end{align*}
Now the claim follows by combining this with~\eqref{siegel-cramer2},~\eqref{lambda-cramer} and applying the triangle inequality for Gowers norms. 
\end{proof}

\section{Quantitative linear equations in primes result}\label{sec:linprimes} 

In this section we sketch the derivation of Theorem~\ref{quant-thm} from Theorem~\ref{main}.  The arguments follow those in~\cite{gt-linear} extremely closely, and we will assume familiarity with those arguments in this section.

In~\cite[\S 4]{gt-linear}, the qualitative version of Theorem~\ref{quant-thm} was derived from~\cite[Theorem 4.5]{gt-linear} using some elementary linear algebra and convex geometry.  The same arguments, replacing all qualitative decay terms with doubly logarithmic ones instead, show that Theorem \ref{quant-thm} will follow if one shows the following.

\begin{theorem}[Primes in affine lattices in normal form]  The statement of~\cite[Theorem 4.5]{gt-linear} continues to hold if the qualitative error term $o(N^d)$ in that theorem is replaced with the doubly logarithmic term $O_{s,d,t}((\log\log N)^{-c} N^d)$ for some $c = c_{s,d,t}>0$ depending only on the parameters $s,d,t$.  (Also one ignores the references to the now proven conjectures $\mathrm{GI}(s), \mathrm{MN}(s)$ in that theorem.)
\end{theorem}

Next, we apply the $W$-trick arguments in~\cite[\S 5]{gt-linear}, setting $w$ equal\footnote{Note that for this choice of $w$, the prime number theorem in arithmetic progressions of modulus $W = P(w)$ has an effective error term with good decay, as we can use the effective lower bounds on $L(1,\chi)$ in this case rather than Siegel's theorem.  It should however be possible to work with larger choices of $w$ by incorporating the contribution of a $Q$-Siegel zero, as is done elsewhere in this paper.} to $(\log\log N)^\eta$ for a sufficiently small $\eta>0$ depending on $s,d,t$ rather than the more conservative choice of $\log\log\log N$.  These arguments then reduce matters to showing

\begin{theorem}[$W$-tricked primes in affine lattices]  The statement of~\cite[Theorem 5.2]{gt-linear} continues to hold if the qualitative error term $o(N^d)$ in that theorem is replaced with the doubly logarithmic term $O_{s,d,t}((\log\log N)^{-c} N^d)$ for some $c = c_{s,d,t}>0$ depending only on the parameters $s,d,t$.  (Again one ignores the references to the now-proven conjectures $\mathrm{GI}(s), \mathrm{MN}(s)$ in that theorem.)
\end{theorem}

The statement of~\cite[Theorem 5.2]{gt-linear} involves the functions
$$ \Lambda'_{b_i,W}(n) := \frac{\phi(W)}{W} \Lambda'(Wn+b_i)$$
where $W \coloneqq P(w)$ and $\Lambda'$ is the restriction of $\Lambda$ to the primes.  From Corollary~\ref{cor:Wtrick}, we have the doubly logarithmic bound
$$ \|\Lambda'_{b_i,W} - 1 \|_{U^{s+1}[\frac{N-b_i}{W}]} \ll_{s,\eta} (\log\log N)^{-c \eta}$$
for some $c>0$ depending only on $s$ (and assuming as we may that $N$ is sufficiently large depending on $s,d,t,\eta$).  On the other hand, a routine modification of Proposition~\ref{tickle} (see also~\cite[Proposition 6.4]{gt-linear}) reveals that for any $D$, the function $1 + \Lambda'_{b_1,W} + \dots + \Lambda'_{b_t,W}$ on the interval $[N^{3/5},N]$ can be bounded by $C \nu$ for some $C = O_{D,\eta}(1)$ and some $\nu$ that obeys the $(D,D,D)$ linear forms condition from~\cite[Definition 6.2]{gt-linear} with the $o_D(1)$ term in~\cite[(6.2)]{gt-linear} replaced by $O_D((\log \log N)^{-c_{D,\eta}})$ for some $c_{D,\eta}>0$.  (We will not need the now largely obsolete ``correlation condition'' in~\cite[Definition 6.3]{gt-linear}.)  The claim now follows from the generalized von Neumann theorem in~\cite[Theorem 7.1]{gt-linear} proven in~\cite[Appendix C]{gt-linear}, after replacing all $o(1)$ type terms with $O((\log\log N)^{-c})$ type terms, noting that all the functions denoted $\kappa$ in that appendix can be taken to be polynomial in nature; we leave the details to the interested reader.

\begin{remark} It seems likely that one can improve Theorem~\ref{quant-thm} further, by allowing the parameter $L$ to be as large as $(\log\log N)^c$ with uniform control on error terms; one may even be able to handle significantly larger values of the linear coefficients $\dot \psi_i$ than this by incorporating the various methods used in this paper.  We will not pursue such refinements here, however.
\end{remark}

\section{Arithmetic progressions with shifted prime difference}\label{sec:sarkozy}

In this section we prove Theorem~\ref{thm_sarkozy}.

\begin{proof}[Proof of Theorem~\ref{thm_sarkozy}] In what follows, let $\Lambda'$ stand for the von Mangoldt function restricted to the primes. Let $A\subset [N]$ be any set with $|A|\geq \delta N$ and $\delta=(\log \log \log \log N)^{-c}$ for small enough $c>0$ depending on $k$. Let $w=(\log \log N)^{1/2}$, and let $W=\prod_{p\leq w}p$. By the pigeonhole principle, we can pick $1\leq b\leq W$ such that
$A':=\{n:Wn+b\in A\}$
has size $\geq \delta N/W$. Then the count of $k$-term arithmetic progressions in $A$ with shifted prime difference  is
\begin{align*}
&\geq \frac{1}{\log N}\sum_{n\leq N/W}\sum_{d\leq N/W}1_{A}(Wn+b)1_{A}(Wn+b+Wd)\cdots 1_{A}(Wn+b+Wd(k-1))\Lambda'(Wd+1)\nonumber\\
&=\frac{1}{\log N}\sum_{n\leq N/W}\sum_{d\leq N/W}1_{A'}(n)1_{A'}(n+d)\cdots 1_{A'}(n+d(k-1))\Lambda'(Wd+1):=T.
\end{align*}

Note that we have the trivial bound $\sum_{n\leq N}|\Lambda(n)-\Lambda'(n)|\ll N^{1/2}\log N$. Using this and our quantitative Gowers uniformity result in the form of Corollary~\ref{cor:Wtrick}, we have
\begin{align*}
\left\|\frac{\phi(W)}{W}\cdot \Lambda'(W\cdot+1)-1\right\|_{U^k[N/W]}&=\left\|\frac{\phi(W)}{W}\Lambda(W\cdot+1)-1\right\|_{U^k[N/W]}+O(N^{-1/2+o(1)})\\
&\ll (\log \log N)^{-c'}
\end{align*}
for some $c'>0$ depending on $k$.  Therefore, by applying the generalized von Neumann theorem for pseudorandomly majorized functions \cite[Theorem 7.1]{gt-linear} (with similar remarks on quantitative error terms as in the proof of Theorem~\ref{quant-thm}), we see that $T$ is  equal to
\begin{align}\label{eq_varnavides}\begin{split}
&\frac{W}{\varphi(W)(\log N)}\sum_{n\leq N/W}\sum_{0\leq d\leq N/W} 1_{A'}(n)1_{A'}(n+d)\cdots 1_{A'}(n+d(k-1))\\
&+O\left(\left(\frac{N}{W}\right)^2(\log \log N)^{-c'}\right).
\end{split}
\end{align}
For $\rho>0$, Let $N_k(\rho)$ denote the smallest positive integer such that, for any $m\geq N_k(\rho)$, any subset of $[m]$ of size $\geq \rho m$ contains a non-trivial $k$-term arithmetic progression. Let $c(k,\delta):=\delta^2/(16N_k(\delta/2)^3)$. Then, by a well-known argument of Varnavides for quantifying Szemer\'edi's theorem (see e.g.~\cite[Theorem 18, Remark 1]{shkredov}), for $N/W> 2N_k(\delta/2)$ the expression \eqref{eq_varnavides} is
\begin{align*}
\geq \frac{W}{\varphi(W)} c(k,\delta)\left(\frac{N}{W}\right)^2+O\left(\left(\frac{N}{W}\right)^2(\log \log N)^{-c'}\right).    
\end{align*}
We have $c(k,\delta)\gg \exp(-\exp(\delta^{-C}))$ for some $C\geq 1$ (depending on $k$) by Gowers’s bound $N_k(\rho)\ll \exp(\exp(\rho^{-C'}))$, proved in~\cite{gowers}. Now, if $c$ is chosen small enough in the definition of $\delta$, we have $c(k,\delta)\gg (\log \log N)^{-o(1)}$, which proves the statement of the theorem for $k\geq 4$. For $k=4$, the same argument works, except that we now use the bound $N_4(\rho)\ll \exp(\rho^{-C})$ from~\cite{green-tao-4AP} to get $c(4,\delta)\gg \exp(-C\delta^{-C})$, which enables taking $\delta=(\log \log \log N)^{-c}$ for some $c>0$. Finally, for $k=3$, using the very recent bound~\cite{Kelley-Meka} $N_3(\rho)\ll \exp((\log(1/\rho))^{C})$ we have $c(3,\delta)\gg \exp(-C(\log(1/\delta))^{C})$, which enables taking $\delta=\exp(-(\log \log \log N)^{c})$ for some $c>0$.
\end{proof}

\appendix

\section{Quantitative Leibman theory with explicit dimension dependence} \label{appendix: a}

In this appendix we refine the equidistribution theory on nilmanifolds from~\cite{gt-leibman}, tracking more carefully the dependence on dimension $m$ (but allowing all constants to depend on the degree $d$, which in our context will equal to $k-1$).  The key point is that all bounds will be at most double exponential in this dimension parameter, basically because the arguments rely on applying the Cauchy--Schwarz inequality (or variants such as the van der Corput inequality) a number of times that is polynomial in the dimension. (Many of the estimates here require only single exponential dependence on $m$ at worst, but the induction on dimension we use only closes if we allow double exponential dependence.)  In order to improve this double exponential dependence it would seem necessary to adopt a different approach to equidistribution that is not as reliant on so many applications of the Cauchy--Schwarz inequality.

We freely use the notation from~\cite{gt-leibman}, and let $m$ be a dimensional parameter.  To conveniently track bounds that depend in double-exponential fashion on the dimension we adopt the following notation.  For any $0 < \delta < 1/2$ let $\poly_m(\delta)$ to be any quantity lower bounded by $\gg \exp(-\exp(m^{O(1)})) \delta^{\exp(m^{O(1)})}$, and for any $Q > 2$ let $\poly_m(Q)$ be any quantity upper bounded by $\ll \exp(\exp(m^{O(1)})) Q^{\exp(m^{O(1)})}$.  In particular $\poly_m(1/\delta)$ is any quantity upper bounded by $\ll \exp(\exp(m^{O(1)})) \delta^{-\exp(m^{O(1)})}$.  

We begin with a more quantitative version of~\cite[Lemma 3.1]{gt-leibman}:

\begin{lemma}[Quantitative Kronecker Theorem]\label{qkt}  Let $m \geq 1$ and $0 < \delta < 1/2$, $\alpha \in \R^m$, $N \geq 1$.  If $(\alpha n \mod \Z^m)_{n \in [N]}$ is not $\delta$-equidistributed in $\R^m/\Z^m$, then there exists $k \in \Z^m$ with $0 < |k| \ll \poly_m(1/\delta)$ such that $\|k \cdot \alpha \|_{\R/\Z} \ll \poly_m(1/\delta) / N$.
\end{lemma}
 
\begin{proof}  The ``simple calculation'' used to establish~\cite[(3.3)]{gt-leibman}, when done a little more carefully, gives
\begin{equation}\label{33}
 \sum_{\substack{k \in \Z^m\\ |k| \geq M}} |\hat K(k)| \ll \poly_m(1/\delta) M^{-1}
\end{equation}
and by chasing through the argument with this bound we obtain the claim.
\end{proof}

This gives a version of~\cite[Lemma 3.7]{gt-leibman}:

\begin{lemma}[Vertical oscillation reduction]\label{vor} Let  $G/\Gamma$ be a filtered nilmanifold of degree $d$, with vertical torus dimension $m_d$. Let $0 < \delta < 1/2$, and let $g: \Z \to G$ be a polynomial sequence for which $(g(n)\Gamma)_{n \in [N]}$ is not $\delta$-equidistributed.  Then there is a vertical character $\xi$ with $|\xi| \leq \poly_{m_d}(1/\delta)$ such that $(g(n)\Gamma)_{n \in [N]}$ is not $\poly_{m_d}(\delta)$-equidistributed along the vertical oscillation $\xi$.
\end{lemma}

\begin{proof} Repeat the proof of~\cite[Lemma 3.7]{gt-leibman} verbatim, using the estimate \eqref{33} in place of~\cite[(3.3)]{gt-leibman}.
\end{proof}

Now we state the main technical theorem on quantitative Leibman theory (a version of~\cite[Theorem 7.1]{gt-leibman}):

\begin{theorem}[Variant of Main Theorem]\label{main-variant}  Let $m \geq m_* \geq 0$ be integers, $0 < \delta < 1/2$, $N \geq 1$.  Let $G/\Gamma$ be a filtered nilmanifold of degree $d$, nonlinearity dimension $m_*$ (defined in~\cite[Section 7]{gt-leibman}), and complexity at most $1/\delta$.  Let $g: \Z \to G$ be a polynomial sequence. If $(g(n)\Gamma)_{n \in [N]}$ is not $\delta$-equidistributed then there exists a horizontal character $\eta$ with $0 < |\eta| \leq \delta^{-\exp((m+m_*)^{C_d})}$ such that 
$$ \| \eta \circ g \|_{C^\infty[N]} \leq \delta^{-\exp((m+m_*)^{C_d})}$$
where $C_d$ is a sufficiently large constant depending only on $d$.
\end{theorem}

We now prove this theorem.  We assume inductively that the claim has already been established for smaller values of $d$, or for the same value of $d$ and smaller values of $m_*$.  Henceforth we refine the $\poly_m$ notation by permitting the implied constants to depend on the constant $C_{d-1}$, but not on $C_d$.

By repeating the derivation of~\cite[(7.1)]{gt-leibman} (using Lemma~\ref{vor} in place of~\cite[Lemma 3.7]{gt-leibman}) we may find some function $F: G/\Gamma \to \C$ with $\|F\|_{\Lip} \ll \poly_m(1/\delta)$ and vertical frequency $\xi$ with $|\xi| \ll \poly_m(1/\delta)$ such that $(g(n)\Gamma)_{n\in [N]}$ is not $\delta^{O(1)}$-equidistributed along $\xi$, and such that
$$ \left|\E_{n \in [N]} F(g(n)\Gamma) - \int_{G/\Gamma} F\right| \gg \poly_m(\delta).$$
If $\xi=0$ then a repetition of the arguments after~\cite[(7.1)]{gt-leibman} gives the claim from the induction hypothesis, so without loss of generality we assume $\xi \neq 0$, thus we now have
$$ |\E_{n \in [N]} F(g(n)\Gamma)| \gg \poly_m(\delta).$$
Repeating the reductions after~\cite[(7.2)]{gt-leibman} we may assume that $g(0)=\id_G$ and $|\psi(g(1))| \leq 1$, where $\psi: G \to \R^m$ is the Mal'cev coordinate map.  Continuing the argument down to~\cite[(7.8)]{gt-leibman} we conclude that
$$ |\E_{n \in [N]} \overline{F_h^\Box}(\overline{g_h^\Box(n)} \overline{\Gamma^\Box})| \gg \poly_m(\delta)$$
with $\overline{F_h^\Box}, \overline{g_h^\Box}, \overline{\Gamma^\Box}$ defined as in~\cite{gt-leibman}.

One can rather tediously verify that all the estimates in~\cite[Appendix A]{gt-leibman} can be refined by replacing all estimates of the form $X \ll_m Q^{O_m(1)} Y$ with $X \ll \poly_m(Q) Y$.  As a consequence we can refine~\cite[Lemma 7.4]{gt-leibman} (by exact repetition of the proof) to

\begin{lemma}[Rationality bounds for the relative square] There is a $\poly_m(1/\delta)$-rational Mal'cev basis ${\mathcal X}^\Box$ for $G^\Box/\Gamma^\Box$ adapted to the filtration $(G^\Box)_\bullet$ with the property that $\psi_{X^\Box}(x,x')$ is a polynomial of degree $O(1)$ with rational coefficients of height $\poly_m(1/\delta)$ in the coordinates $\psi(x),\psi(x')$.  With respect to the metric $d_{{\mathcal X}^\Box}$ we have $\|F_h^\Box\|_\Lip \ll \poly_m(1/\delta)$ uniformly in $h$.
\end{lemma}

Continuing the arguments down to~\cite[Lemma 7.5]{gt-leibman}, one can find horizontal characters $\eta_1: G \to \R/\Z$ , $\eta_2: G_2 \to \R/\Z$ with $\eta_2$ annihilating $[G,G_2]$ and $|\eta_1|, |\eta_2| \ll \poly_m(1/\delta)$ such that the character $\eta: G^\Box \to \R/\Z$ defined by
$$ \eta(g',g) \coloneqq \eta_1(g) + \eta_2(g'g^{-1})$$
is such that
$$ \| \eta \circ g_h^\Box \|_{C^\infty([N])} \ll \poly_m(1/\delta)$$
for $\gg \poly_m(\delta) N$ values of $h \in [N]$.

Continuing the argument down to~\cite[(7.16)]{gt-leibman}, and using the induction hypothesis for Theorem~\ref{main-variant} (with $d$ replaced by $d-1$, and $m,m_*$ replaced by quantities not exceeding $2m$), we can find $1 \leq q \ll \poly_m(1/\delta)$ such that
$$ \| \eta_1(g(1)) + \zeta \cdot \{ \gamma h \} + q \alpha h \|_{\R/\Z} \ll \poly_m(1/\delta)/N$$
for $\gg \poly_m(\delta) N$ values of $h \in [N]$, where 
\begin{itemize}
\item[(i)] $\alpha \in \R/\Z$ is the quantity $\alpha \coloneqq \partial^2 (\eta_2 \circ g_2)(0)$, where $g_2(n) = g(n) g(1)^{-n}$ is the nonlinear part of $g$.
\item[(ii)] $\gamma \in (\R/\Z)^{m_\lin}$ is (the first $m_\lin$ components of) $\psi(g(1))$.
\item[(iii)] $\zeta \in \R^{m_\lin}$ is the vector such that
$$ \eta_2([g(1),x]) = \zeta \cdot \psi(x) \mod \Z$$
for all $x \in G$ (extending $\zeta$ by zero to $\R^m$).
\end{itemize}
Here it is important that the implied constants in the $\poly_m$ notation are allowed to depend on $C_{d-1}$ (but not $C_d$).

It is routine to verify that $|\zeta| \ll \poly_m(1/\delta)$.  An inspection of the proof of~\cite[Proposition 5.3]{gt-leibman} and~\cite[Claim 7.7]{gt-leibman}, using Lemma~\ref{qkt} in place of~\cite[Lemma 3.1]{gt-leibman}, shows that we may replace all bounds of the form $X \ll_m \delta^{-O_m(1)} Y$ appearing in these statements by $X \ll \poly_m(1/\delta) Y$, to obtain one of the following claims:

\begin{itemize}
\item[(i)] There is $r \ll \poly_m(1/\delta)$ such that $\|r \zeta_i \mod \Z\|_{\R/\Z} \ll \poly_m(1/\delta)/N$ for all $i=1,\dots,m_\lin$; or
\item[(ii)] There exists $k \in \Z^{m_\lin}$, $0 < |k| \ll \poly_m(1/\delta)$ such that $\| k \cdot \gamma \|_{\R/\Z} \ll \poly_m(1/\delta)/N$.
\end{itemize}

In case (ii) we conclude exactly as in~\cite{gt-leibman}, so suppose we are in case (i).  Arguing as in~\cite{gt-leibman} we can easily close the induction except in the case when $\eta_2$ (and hence $\eta$) annihilates $[G,G]$, at which point the arguments in~\cite{gt-leibman} lead one to conclude that
$$ \| \eta_2 \circ g_2 \|_{\R/\Z} \ll \poly_m(1/\delta)$$
(possibly after first multiplying $\eta_1, \eta_2$ by a positive integer of size $\poly_m(1/\delta)$.  

Repeating the rest of the proof of~\cite[Theorem 7.1]{gt-leibman} (replacing all bounds of the form $X \ll_m \delta^{-O_m(1)} Y$ with $X \ll \poly_m(1/\delta) Y$) and using the induction hypothesis with $(d,m_*)$ replaced by $(d,m_*-1)$, we see that
$$ \| \eta \circ g \|_{C^\infty([N])} \leq \poly_m(1/\delta)^{\exp((m+m_*-1)^{C_d})} $$
for some  horizontal character $\eta: G \to \R/\Z$ with $|\eta| \leq \poly_m(1/\delta)^{\exp((m+m_*-1)^{C_d})}$.  For $C_d$ large enough, we have
$$ \poly_m(1/\delta)^{\exp((m+m_*-1)^{C_d})} \leq \delta^{-\exp((m+m_*)^{C_d})},$$
and Theorem~\ref{main-variant} follows.

Repeating the proof of~\cite[Proposition 9.2]{gt-leibman} (specializing to the single-parameter case $t=1$), we then obtain

\begin{proposition}[Factorization of poorly-distributed polynomial sequences] Let $m \geq 1$, $0 < \delta < 1/2$, $N \geq 1$, $d \geq 0$, let $G/\Gamma$ be a $m$-dimensional filtered nilmanifold of complexity at most $1/\delta$, and let $g \colon \Z \to G$ be a polyonmial sequence.  Suppose that $(g(n) \Gamma)_{n \in [N]}$ is not totally $\delta$-equidistributed.  Then there is a factorization $g = \eps g' \gamma$ with $\eps,g',\gamma \colon \Z \to G$ polynomials
such that
\begin{itemize}
\item[(i)] $\eps: \Z \to G$ is $(\poly_m(1/\delta),N)$-smooth;
\item[(ii)] $g': \Z \to G$ takes values in a connected proper $\poly_m(1/\delta)$-rational subgroup $G'$ of $G$;
\item[(iii)] $\gamma: \Z \to G$ is $\poly_m(1/\delta)$-rational.
\end{itemize}
\end{proposition}

In~\cite[Lemma 10.1]{gt-leibman} with $t=1$, one easily verifies that the bound $M^{O_m(1)}$ in the conclusion can be sharpened to $\poly_m(M)$.  We now claim the following quantitative version of~\cite[Theorem 1.19]{gt-leibman}:

\begin{theorem}[Factorization theorem]\label{factor-thm} Let $m \geq 0$, $M_0 \geq 2$, $A \geq 2$, $N \geq 1$, $d \geq 0$.  Let $G/\Gamma$ be an $m$-dimensional filtered nilmanifold of degree $d$ and complexity at most $M_0$, and let $g \colon \Z \to G$ be a polynomial sequence. 	Then there is some $M$ with $M_0 \leq M \leq M_0^{A^{(2+m)^{O_d(1)}}}$, a subgroup $G' \subset G$ which is $M$-rational with respect to ${\mathcal X}$, and a decomposition $g = \eps g' \gamma$ with $\eps,g',\gamma \colon \Z \to G$ polynomials such that
\begin{itemize}
\item[(i)] $\eps$ is $(M,N)$-smooth;
\item[(ii)] $g'$ takes values in $G'$ and $(g'(n)\Gamma)_{n \in [N]}$ is totally $1/M^A$-equidistributed in $G'/\Gamma'$, with respect to a Mal'cev basis ${\mathcal X}'$ consisting of $M$-rational linear combinations of the basis elements of the Mal'cev basis for $G$;
\item[(iii)] $\gamma$ is $M$-rational and $\gamma(n) \Gamma$ is periodic with period at most $M$.
\end{itemize}
\end{theorem}

\begin{proof}  Repeat the proof of~\cite[Theorem 10.2]{gt-leibman} with $t=1$, setting $\delta_{i+1} := \delta_i^{A^{m^C}}$ for a sufficiently large constant $C = C_d$ depending only on $d$ (in particular, $1/\delta_{i+1}$ is much larger than any quantity of the form $\poly_m(1/\delta_i)^A$ if $C$ is large enough).
\end{proof}

\section{Proof of Theorem \ref{known-unif}}

\begin{proof}[Proof of Theorem~\ref{known-unif}]  Part (i) follows easily from the prime number theorem with Vinogradov--Korobov error terms (for the M\"obius case, see~\cite[Satz 3 in Section V.5]{walfisz}).  Part (ii) for the M\"obius function follows from the strongly logarithmic exponential sum estimates
$$ \sup_\theta |\E_{n \in [N]} \mu(n) e(\theta n)| \ll^\ineff_A \log^{-A} N$$
of Davenport~\cite{davenport}, the Plancherel estimate
\begin{equation}\label{mu-planch}
 \int_0^1 \left|\E_{n \in [N]} \mu(n) e(\theta n)\right|^2 \ d\theta \ll N,
\end{equation}
the circle method, and Cauchy--Schwarz.  For the second part of (ii), observe from Proposition~\ref{cramer-stable} and \eqref{triangle} that we may take $w = \log^{1/100} N$ (say) without loss of generality. The standard Vinogradov estimates for exponential sums over primes (see e.g.,~\cite[Ch. 13]{iwaniec-kowalski}) eventually reveal the logarithmic bounds
$$ \sup_\theta \left|\E_{n \in [N]} (\Lambda(n) - \Lambda_{\Cramer,w}(n)) e(\theta n)\right| \ll^\ineff \log^{-c} N,$$
while the Fourier restriction estimate from~\cite[Proposition 4.2]{gt-restriction} gives
$$ \int_0^1 \left|\E_{n \in [N]} (\Lambda(n) - \Lambda_{\Cramer,w}(n)) e(\theta n)\right|^q\ d\theta \ll_q 1$$
for any $2 < q < \infty$, and the claim now follows from the circle method and H\"older's inequality.  Finally, for (iii), we see from Proposition~\ref{cramer-stable} and \eqref{triangle} that we may assume that $w$ grows sufficiently slowly in $N$, and then the bounds in (iii) follow easily from the main theorems in~\cite{gt-linear} as well as Corollary~\ref{slice}, after inserting the resolution of the inverse conjecture for the Gowers norms (first proven in~\cite{gtz-uk}) and the strong orthogonality of the M\"obius function to nilsequences (first proven in~\cite{gt-mobius}).
\end{proof}

\begin{remark}  An alternate approach to \eqref{lambdam} proceeds by comparing $\Lambda(n) = -\sum_{d|n} \mu(d) \log d$ first with a truncated divisor sum $\Lambda^\sharp(n) \coloneqq - \sum_{d|n: d \leq N^{c_1}} \mu(d) \log d$ for some small absolute constant $c_1>0$, and establishing the strongly logarithmic estimate
$$ \| \Lambda - \Lambda^\# \|_{U^2[N]} \ll^\ineff_A \log^{-A} N$$
from the circle method (here we can use a Plancherel bound analogous to \eqref{mu-planch} that loses a factor of $\log N$, thus avoiding the need to invoke the restriction theory from~\cite{gt-restriction}), and the logarithmic estimate
$$ \| \Lambda^\# - \Lambda_{\Cramer,w} \|_{U^2[N]} \ll \log^{-c} N$$
from sieve theory with (say) $w = \log^{1/100} N$, and then applying the triangle inequality \eqref{triangle}; we leave the details to the interested reader.  In this paper we found the Cram\'er models $\Lambda_{\Cramer,w}$ to be slightly more convenient technically to work with than the truncated divisor sum model $\Lambda^\sharp$, and therefore made no further use of $\Lambda^\sharp$ here.
\end{remark}

\label{appendix_qualitative}

\bibliography{gowersrefs}

\begin{thebibliography}{10}

\bibitem{bloom-survey}
T.~F. Bloom.
\newblock Quantitative inverse theory of {G}owers uniformity norms.
\newblock {\em Ast\'{e}risque}, (430):237--273, 2021.

\bibitem{chinis}
J.~{Chinis}.
\newblock {Siegel Zeros and Sarnak's Conjecture}.
\newblock {\em arXiv e-prints}, page arXiv:2105.14653, May 2021.

\bibitem{cladek}
L.~Cladek and T.~Tao.
\newblock Additive energy of regular measures in one and higher dimensions, and the fractal uncertainty principle.
\newblock {\em Ars Inven. Anal.}, pages Paper No. 1, 38, 2021.

\bibitem{cfz}
D.~Conlon, J.~Fox, and Y.~Zhao.
\newblock The {G}reen-{T}ao theorem: an exposition.
\newblock {\em EMS Surv. Math. Sci.}, 1(2):249--282, 2014.

\bibitem{cfz-rel}
D.~Conlon, J.~Fox, and Y.~Zhao.
\newblock A relative {S}zemer\'{e}di theorem.
\newblock {\em Geom. Funct. Anal.}, 25(3):733--762, 2015.

\bibitem{cramer}
H.~Cram\'er.
\newblock On the order of magnitude of the difference between consecutive prime numbers.
\newblock {\em Acta Arith.}, 2:23--46, 1936.

\bibitem{davenport}
H.~Davenport.
\newblock On some infinite series involving arithmetical functions. {II}.
\newblock {\em Quart. J. Math. Oxf.}, 8:313--320, 1937.

\bibitem{davenport-book}
H.~Davenport.
\newblock {\em Multiplicative number theory}, volume~74 of {\em Graduate Texts in Mathematics}.
\newblock Springer-Verlag, New York, third edition, 2000.
\newblock Revised and with a preface by Hugh L. Montgomery.

\bibitem{dodos-k}
P.~Dodos and V.~Kanellopoulos.
\newblock Uniformity norms, their weaker versions, and applications.
\newblock {\em Acta Arith.}, 203(3):251--270, 2022.

\bibitem{fhk}
N.~Frantzikinakis, B.~Host, and B.~Kra.
\newblock Multiple recurrence and convergence for sequences related to the prime numbers.
\newblock {\em J. Reine Angew. Math.}, 611:131--144, 2007.

\bibitem{friedlander-iwaniec}
J.~Friedlander and H.~Iwaniec.
\newblock {\em Opera de cribro}, volume~57 of {\em American Mathematical Society Colloquium Publications}.
\newblock American Mathematical Society, Providence, RI, 2010.

\bibitem{german-katai}
L.~Germ\'{a}n and I.~K\'{a}tai.
\newblock On multiplicative functions on consecutive integers.
\newblock {\em Lith. Math. J.}, 50(1):43--53, 2010.

\bibitem{gowers}
W.~T. Gowers.
\newblock A new proof of {S}zemer\'{e}di's theorem.
\newblock {\em Geom. Funct. Anal.}, 11(3):465--588, 2001.

\bibitem{gowers-hb}
W.~T. Gowers.
\newblock Decompositions, approximate structure, transference, and the {H}ahn-{B}anach theorem.
\newblock {\em Bull. Lond. Math. Soc.}, 42(4):573--606, 2010.

\bibitem{granville}
A.~Granville.
\newblock Harald {C}ram\'{e}r and the distribution of prime numbers.
\newblock {\em Scand. Actuar. J.}, (1):12--28, 1995.
\newblock Harald Cram\'{e}r Symposium (Stockholm, 1993).

\bibitem{green-sarkozy}
B.~{Green}.
\newblock {On S{\'a}rk{\"o}zy's theorem for shifted primes}.
\newblock {\em \textnormal{To appear in} J. Am. Math. Soc.}, page arXiv:2206.08001, June 2022.

\bibitem{gt-restriction}
B.~Green and T.~Tao.
\newblock Restriction theory of the {S}elberg sieve, with applications.
\newblock {\em J. Th\'{e}or. Nombres Bordeaux}, 18(1):147--182, 2006.

\bibitem{gt-U3}
B.~Green and T.~Tao.
\newblock An inverse theorem for the {G}owers {$U^3(G)$} norm.
\newblock {\em Proc. Edinb. Math. Soc. (2)}, 51(1):73--153, 2008.

\bibitem{gt-longaps}
B.~Green and T.~Tao.
\newblock The primes contain arbitrarily long arithmetic progressions.
\newblock {\em Ann. of Math. (2)}, 167(2):481--547, 2008.

\bibitem{gt-quadratic}
B.~Green and T.~Tao.
\newblock Quadratic uniformity of the {M}\"{o}bius function.
\newblock {\em Ann. Inst. Fourier (Grenoble)}, 58(6):1863--1935, 2008.

\bibitem{gt-sumset}
B.~Green and T.~Tao.
\newblock An equivalence between inverse sumset theorems and inverse conjectures for the {$U^3$} norm.
\newblock {\em Math. Proc. Cambridge Philos. Soc.}, 149(1):1--19, 2010.

\bibitem{gt-linear}
B.~Green and T.~Tao.
\newblock Linear equations in primes.
\newblock {\em Ann. of Math. (2)}, 171(3):1753--1850, 2010.

\bibitem{gt-mobius}
B.~Green and T.~Tao.
\newblock The {M}\"obius function is strongly orthogonal to nilsequences.
\newblock {\em Ann. of Math. (2)}, 175(2):541--566, 2012.

\bibitem{gt-leibman}
B.~Green and T.~Tao.
\newblock The quantitative behaviour of polynomial orbits on nilmanifolds.
\newblock {\em Ann. of Math. (2)}, 175(2):465--540, 2012.

\bibitem{green-tao-4AP}
B.~Green and T.~Tao.
\newblock New bounds for {S}zemer\'{e}di's theorem, {III}: a polylogarithmic bound for {$r_4(N)$}.
\newblock {\em Mathematika}, 63(3):944--1040, 2017.

\bibitem{gt-U4}
B.~Green, T.~Tao, and T.~Ziegler.
\newblock An inverse theorem for the {G}owers {$U^4$}-norm.
\newblock {\em Glasg. Math. J.}, 53(1):1--50, 2011.

\bibitem{gtz-uk}
B.~Green, T.~Tao, and T.~Ziegler.
\newblock An inverse theorem for the {G}owers {$U^{s+1}[N]$}-norm.
\newblock {\em Ann. of Math. (2)}, 176(2):1231--1372, 2012.

\bibitem{hildebrand-tenenbaum}
A.~Hildebrand and G.~Tenenbaum.
\newblock Integers without large prime factors.
\newblock {\em J. Th\'eor. Nombres Bordeaux}, 5(2):411--484, 1993.

\bibitem{iwaniec-kowalski}
H.~Iwaniec and E.~Kowalski.
\newblock {\em Analytic number theory}, volume~53 of {\em American Mathematical Society Colloquium Publications}.
\newblock American Mathematical Society, Providence, RI, 2004.

\bibitem{Kelley-Meka}
Z.~{Kelley} and R.~{Meka}.
\newblock {Strong Bounds for 3-Progressions}.
\newblock {\em arXiv e-prints}, page arXiv:2302.05537, February 2023.

\bibitem{leng}
J.~{Leng}.
\newblock {Improved Quadratic Gowers Uniformity for the M{\"o}bius Function}.
\newblock {\em arXiv e-prints}, page arXiv:2212.09635, December 2022.

\bibitem{lucier}
J.~Lucier.
\newblock Difference sets and shifted primes.
\newblock {\em Acta Math. Hungar.}, 120(1-2):79--102, 2008.

\bibitem{manners}
F.~{Manners}.
\newblock {Quantitative bounds in the inverse theorem for the Gowers $U^{s+1}$-norms over cyclic groups}.
\newblock {\em arXiv e-prints}, page arXiv:1811.00718, November 2018.

\bibitem{MRTTZ}
K.~Matom\"{a}ki, M.~Radziwi{\l}{\l}, T.~Tao, J.~Ter\"{a}v\"{a}inen, and T.~Ziegler.
\newblock Higher uniformity of bounded multiplicative functions in short intervals on average.
\newblock {\em Ann. of Math. (2)}, 197(2):739--857, 2023.

\bibitem{mv}
H.~L. Montgomery and R.~C. Vaughan.
\newblock {\em Multiplicative number theory. {I}. {C}lassical theory}, volume~97 of {\em Cambridge Studies in Advanced Mathematics}.
\newblock Cambridge University Press, Cambridge, 2007.

\bibitem{RTTV}
O.~Reingold, L.~Trevisan, M.~Tulsiani, and S.~Vadhan.
\newblock Dense subsets of pseudorandom sets.
\newblock {\em Electronic Colloquium on Computational Complexity, Proceedings of 49th IEEE FOCS}, 2008.

\bibitem{ruzsa-sanders}
I.~Z. Ruzsa and T.~Sanders.
\newblock Difference sets and the primes.
\newblock {\em Acta Arith.}, 131(3):281--301, 2008.

\bibitem{sanders}
T.~Sanders.
\newblock On the {B}ogolyubov-{R}uzsa lemma.
\newblock {\em Anal. PDE}, 5(3):627--655, 2012.

\bibitem{sarkozy}
A.~S\'{a}rk\"{o}zy.
\newblock On difference sets of sequences of integers. {III}.
\newblock {\em Acta Math. Acad. Sci. Hungar.}, 31(3-4):355--386, 1978.

\bibitem{schmidt}
W.~M. Schmidt.
\newblock {\em Small fractional parts of polynomials}.
\newblock American Mathematical Society, Providence, R.I., 1977.
\newblock Regional Conference Series in Mathematics, No. 32.

\bibitem{shkredov}
I.~D. Shkredov.
\newblock Szemer\'{e}di's theorem and problems of arithmetic progressions.
\newblock {\em Uspekhi Mat. Nauk}, 61(6(372)):111--178, 2006.

\bibitem{tt-oddorder}
T.~{Tao} and J.~{Ter{\"a}v{\"a}inen}.
\newblock Odd order cases of the logarithmically averaged chowla conjecture.
\newblock {\em J. Th\'eor. Nombres Bordeaux}, 30(3):997--1015, 2018.

\bibitem{vaughan}
R.-C. Vaughan.
\newblock Sommes trigonom\'{e}triques sur les nombres premiers.
\newblock {\em C. R. Acad. Sci. Paris S\'{e}r. A-B}, 285(16):A981--A983, 1977.

\bibitem{walfisz}
A.~Walfisz.
\newblock {\em Weylsche {E}xponentialsummen in der neueren {Z}ahlentheorie}.
\newblock Mathematische Forschungsberichte, XV. VEB Deutscher Verlag der Wissenschaften, Berlin, 1963.

\bibitem{walker}
A.~Walker.
\newblock Linear inequalities in primes.
\newblock {\em J. Anal. Math.}, 145(1):29--127, 2021.

\bibitem{wang}
R.~Wang.
\newblock On a theorem of {S}\'{a}rk\"{o}zy for difference sets and shifted primes.
\newblock {\em J. Number Theory}, 211:220--234, 2020.

\bibitem{wooley-ziegler}
T.~D. Wooley and T.~D. Ziegler.
\newblock Multiple recurrence and convergence along the primes.
\newblock {\em Amer. J. Math.}, 134(6):1705--1732, 2012.

\end{thebibliography}
\bibliographystyle{plain}

\end{document}